\documentclass{amsart}%
\usepackage{amssymb}
\usepackage{amsmath}
\usepackage{amsfonts}
\usepackage[margin=1.5in]{geometry}%
\usepackage{graphicx}

\newtheorem{theorem}{Theorem}
\newtheorem*{theoremAnn}{Theorem}
\numberwithin{theorem}{section}
\theoremstyle{plain}
\newtheorem*{acknowledgement}{Acknowledgement}

\newtheorem*{assumption}{Assumption}

\newtheorem{corollary}[theorem]{Corollary}

\newtheorem*{heuristic}{Heuristic}
\newtheorem{definition}[theorem]{Definition}
\newtheorem{example}[theorem]{Example}

\newtheorem{lemma}[theorem]{Lemma}

\newtheorem{proposition}[theorem]{Proposition}
\theoremstyle{remark}
\newtheorem{remark}[theorem]{Remark}

\newtheorem{aside}[theorem]{Aside}

\numberwithin{equation}{section}
\begin{document}
\title[Torsion homology growth]{Torsion homology growth beyond asymptotics}
\author{Oliver Braunling}
\curraddr{University of Freiburg, Eckerstrassee 1, 79104 Freiburg im Breisgau, Germany}
\email{oliver.braeunling@math.uni-freiburg.de}
\thanks{The author was supported by the DFG GK1821 \textquotedblleft Cohomological
Methods in Geometry\textquotedblright.}

\begin{abstract}
We show that (under mild assumptions) the generating function of log homology
torsion of a knot exterior has a meromorphic continuation to the entire
complex plane. As corollaries, this gives new proofs of (a) the
Silver--Williams asymptotic, (b) Fried's theorem on reconstructing the
Alexander polynomial (c)\ Gordon's theorem on periodic homology. Our results
generalize to other rank $1$ growth phenomena, e.g. Reidemeister--Franz
torsion growth for higher-dimensional knots. We also analyze the exceptional
cases where the meromorphic continuation does not exist.

\end{abstract}
\date{{\today}}
\maketitle

Let $K\subset S^{3}$ be a knot and $X_{K}:=S^{3}-K$ its knot complement. Write
$X_{r}$ for the $r$-th cyclic covering of $X_{K}$. The Silver--Williams
theorem asserts that%
\begin{equation}
\underset{r\rightarrow\infty}{\lim}\frac{\log\left\vert H_{1}(X_{r}%
,\mathbf{Z})_{\operatorname*{tor}}\right\vert }{r}=\log\mathcal{M}(\Delta
_{K})\text{,}\label{li_E1}%
\end{equation}
if $\mathcal{M}(\Delta_{K})>1$ is the Mahler measure of the Alexander
polynomial $\Delta_{K}$ of the knot. Instead of just asking about the
asymptotic behaviour of torsion homology growth in $H_{1}$, we could ask about
\textit{all} values $H_{1}(X_{r},\mathbf{Z})_{\operatorname*{tor}}$.\ Define%
\begin{equation}
E(z):=\sum_{r\geq1}\log\left\vert H_{1}(X_{r},\mathbf{Z})_{\operatorname*{tor}%
}\right\vert \cdot z^{r}\text{.}\label{li_E2}%
\end{equation}
This is a power series around $z=0$. A heuristic argument shows that the
Silver--Williams asymptotics suggest that $E$ might have a meromorphic
continuation beyond the unit circle with a pole of order $1$ or $2$ at $z=1$.
Indeed, whenever $E$ has said property, the asymptotics of Equation
\ref{li_E1} are an immediate consequence. Inspired by this, we seek to
understand whether $E$ has such a meromorphic continuation.

Let us call a root $\beta$ of $\Delta_{K}$ \textsl{diophantine} if it lies on
the unit circle, but is not a root of unity.

\begin{theoremAnn}
Suppose the Alexander polynomial $\Delta_{K}$ of a knot has no diophantine
roots. Then $E$ admits a meromorphic continuation to the entire complex plane.

\begin{enumerate}
\item The pole locus is%
\[
\{\beta^{n}\mid\Delta_{K}(\beta)=0\text{ and }n\text{ an integer}%
\}\setminus(\text{open unit disc})\text{.}%
\]
For each pole, its residue encodes the multiplicity of $\beta$ as a root of
$\Delta_{K}$.

\item At $z=1$, it has a pole of order $1$ or $2$. All other poles have order
$1$.
\end{enumerate}
\end{theoremAnn}

This is our first main result (Theorem \ref{thm_A_for_knots}). If the (rather
mild) assumptions of the theorem are met, it affirms our heuristic about a
pole at $z=1$, and in fact it gives a new proof of the Silver--Williams
asymptotic. But it implies more. A theorem of Fried says that the knowledge of
the torsion orders $\left\vert H_{1}(X_{r},\mathbf{Z})_{\operatorname*{tor}%
}\right\vert $ for all $r$ allows us to reconstruct the Alexander polynomial
of the knot. However, this also follows at once from the above theorem because
we just need to look at the poles of the meromorphic continuation.

The theorem is not just a theoretical result. We can explicitly compute this
analytic continuation. For example, for the knot \textquotedblleft%
\textrm{K8}$_{256}$\textquotedblright\footnote{census tabulation along size of
triangulation as used in SnapPea, \cite{MR3298204}}\ we get%
\[
{\includegraphics[
natheight=3.657700in,
natwidth=3.657700in,
height=1.8968in,
width=1.8968in
]%
{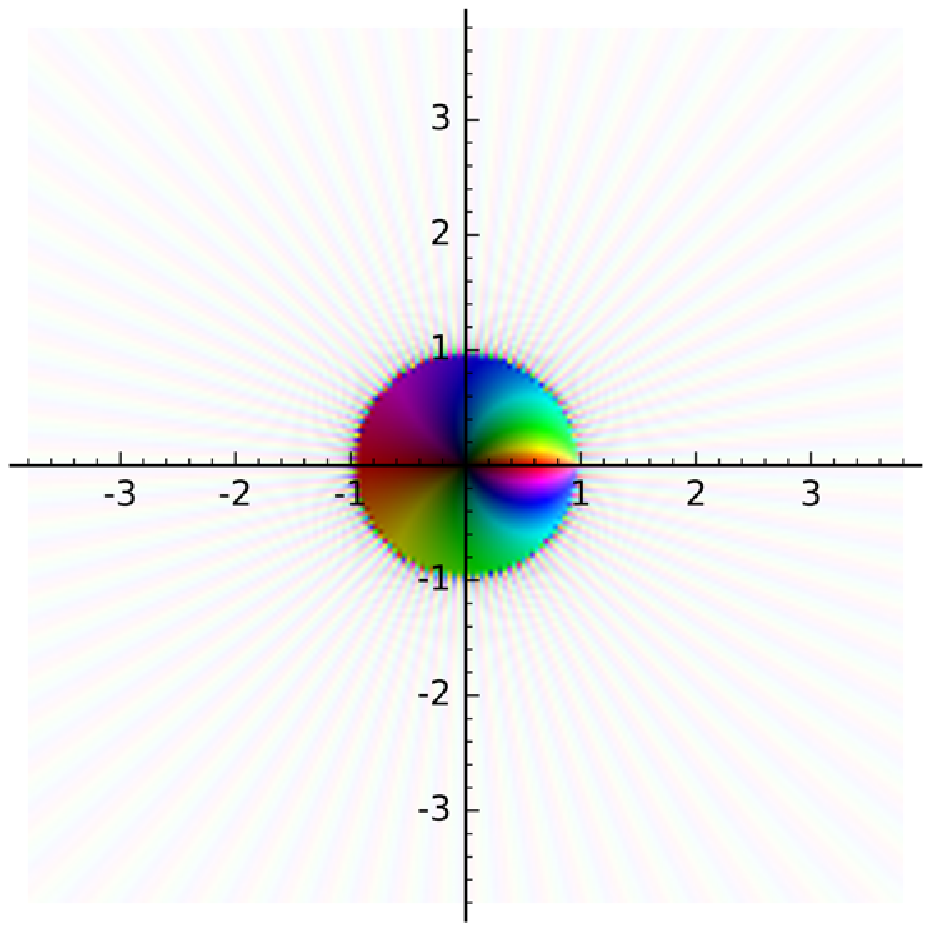}%
}
\qquad\qquad%
{\includegraphics[
natheight=3.657700in,
natwidth=3.657700in,
height=1.9042in,
width=1.9042in
]%
{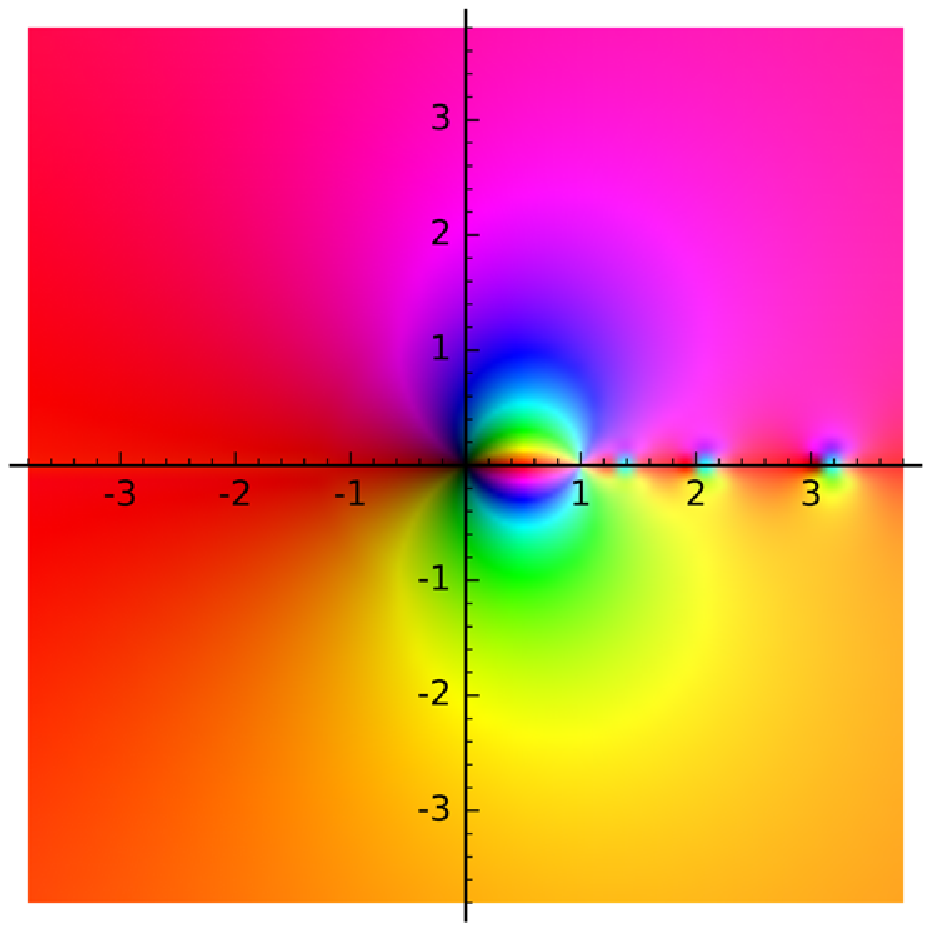}%
}
\]
On the left, we see the evaluation of $E$ as the power series of line
\ref{li_E2}. One can clearly see how the series $E$ diverges outside the unit
circle (as is forced by the pole at $z=1$). On the right, we see our analytic
continuation. The Alexander polynomial has roots at $\frac{2}{3}$ and
$\frac{3}{2}$, and its smallest integer powers outside the unit disc are at
$1.5$, $2.25$, $3.37\ldots$, as one can also read off the plot.

We also obtain a strengthening of Gordon's classical result on periodic
torsion homology \cite{MR0295327}:

\begin{theoremAnn}
For a given knot, the following are equivalent:

\begin{enumerate}
\item The values $\left\vert H_{1}(X_{r},\mathbf{Z})_{\operatorname*{tor}%
}\right\vert $ are periodic in $r$.

\item All Alexander roots are roots of unity.

\item $E$ is a rational function.

\item $E$ has an analytic continuation to the entire complex plane with only
finitely many poles.

\item The values $\log\left\vert H_{1}(X_{r},\mathbf{Z})_{\operatorname*{tor}%
}\right\vert $ satisfy a linear recurrence equation.
\end{enumerate}
\end{theoremAnn}

This will be Theorem \ref{thm_refined_gordon_thm}. The equivalence
$(1)\Leftrightarrow(2)$ is Gordon's classical result.

So far, we have described results under the assumption that no root of the
Alexander polynomial is diophantine. What happens in the rare case if there is
a diophantine root? In this case, everything changes drastically. We prove:

\begin{theoremAnn}
Suppose the Alexander polynomial $\Delta_{K}$ of a knot has at least one
diophantine root. Then $E$ has the unit circle as its natural boundary, i.e.
no analytic continuation is possible. Moreover,%
\[
\operatorname*{limrad}_{z\rightarrow p}\,(1-\left\vert z\right\vert )E(z)\neq0
\]
if $\left\vert p\right\vert =1$ lies in the multiplicative span of the
diophantine roots, and it is zero if $p$ is multiplicatively independent of
all diophantine roots.
\end{theoremAnn}

Here \textquotedblleft$\operatorname*{limrad}_{z\rightarrow p}$%
\textquotedblright\ refers to the limit under all sequences of constant
complex argument. This will be Theorem \ref{thm_A_for_knots_NatBdryCase}.

In fact, we get a relation between the singular values on the unit
circle and special $L$-values, see \S \ref{subsect_Intro_LVals}.

\subsection{Application to Reidemeister torsion}

As mentioned before, our results do not just apply to knots in the $3$-sphere,
but also to other rank one growth phenomena governed by the Mahler measure. As
is explained in many places, e.g. \cite{MR832411}, it is natural to view the
torsion homology in $H_{1}$ in the\ Silver--Williams theorem as a special
instance of the growth of Reidemeister--Franz torsion.

For example, using a variation of the arguments of our Theorems, we can also
show the following:

\begin{theoremAnn}
Let $K^{n}\subset M^{n+2}$ be an $n$-knot, where $M^{n+2}$ is a\ $(n+2)$%
-dimensional homology sphere (in the PL\ category). If $\Delta_{K^{n},i}$
denotes the $i$-th Alexander polynomial, and none of the $\Delta_{K^{n},i}$
has a root in $\mu_{\infty}$, then the generating function of the Reidemeister
torsion%
\[
J_{K^{n}}(z):=\sum_{r\geq1}\log(\tau_{r})\cdot z^{r}%
\]
with%
\[
\tau_{r}:=\prod_{i=1}^{n}\left\vert H_{i}(\widehat{X}_{r},\mathbf{Z}%
)\right\vert ^{(-1)^{i+1}}\text{,}%
\]
where $\widehat{X_{r}}$ is the $r$-th cyclic branched covering, has the
following property:

\begin{enumerate}
\item If no root of any of the Alexander polynomials $\Delta_{K^{n},i}$ has
absolute value $1$, the function admits a meromorphic continuation to the
entire complex plane. Its poles are located at most at all integer powers of
all roots of all $\Delta_{K^{n},i}$ which lie outside the open unit disc.

\item If some $\Delta_{K^{n},i}$ has a root of absolute value $1$ and no other
$\Delta_{K^{n},j}$ (with $j\neq i$) has a root at the same value, then
$J_{K^{n}}$ has the unit circle as its natural boundary. An analytic
continuation beyond the unit circle is impossible.
\end{enumerate}
\end{theoremAnn}

See Theorem \ref{thm_PortiFoxFormula}. This result arises immediately from
combining Porti's Mayberry--Murasugi type formula of \cite{MR2073320} with the
tools which we develop in this paper.

Many similar variations around Reidemeister torsion will be possible.\medskip

There are also applications which are less connected to geometry. For example,
Hillar \cite{MR2167674} studied polynomials which have the same cyclic resultants:

\begin{theoremAnn}
[Hillar]Let $f,g\in\mathbf{R}[t]$ be polynomials such that their cyclic
resultants are all non-zero. Then the absolute values of the cyclic resultants
agree if and only if there exist $u,v\in\mathbf{C}[t]$ with $u(0)\neq0$ and
integers $\ell_{1},\ell_{2}\geq0$ such that%
\begin{align*}
f(t)  & =\pm t^{\ell_{1}}v(t)u(t^{-1})t^{\deg u}\\
g(t)  & =t^{\ell_{2}}v(t)u(t)\text{.}%
\end{align*}

\end{theoremAnn}

For polynomials which have no roots on the unit circle (this is the generic
case), our methods give a new proof of this theorem. See
\S \ref{sect_CyclicResultants}.

\subsection{Open questions}

We do not know how the generation functions behave in rank $d\geq2$
situations,%
\[
\pi_{1}(X,\ast)\twoheadrightarrow\mathbf{Z}^{d}\text{,}%
\]
as in Le \cite{MR3260847} and\ Raimbault \cite{MR2966689}, where one also has
an asymptotic governed by the Mahler measure. More broadly, one could dream
about studying the generating functions of torsion coming from lattice
quotients in Lie groups, inspired by the asymptotics \`{a} la \cite{MR3028790}%
. Unfortunately, at present, this seems completely our of reach.

\subsection{Relation to special $L$-values\label{subsect_Intro_LVals}}

As an accidental finding along the way, we find a new relation to special
$L$-values. So far, it is known that there is \textsl{some} relation between
Mahler measures and special $L$-values through the Beilinson conjectures. A
popular example is the two-variable Mahler measure%
\[
\mathcal{M}(1+x+y)=\frac{3\sqrt{3}}{4\pi}L(2,\chi)\text{,}%
\]
where $\chi$ is a certain Dirichlet character. This was discovered by Smyth,
and later theoretically explained by Deninger \cite{MR1415320}. However, to
the best of my knowledge, this was the only suggestion of a possible
connection between the\ Silver--Williams theorem and special $L$-values.

However, when evaluating the singular limit values of $E$ for a knot with
diophantine roots, other special $L$-values at $s=1$ show up:

\begin{theoremAnn}
Let $K\subset S^{3}$ be a knot and $\Delta_{K}$ its Alexander polynomial.
Suppose $\Delta_{K}$ has at least one diophantine root. Let $p$ be a point of
the unit circle which lies in the multiplicative span\footnote{We understand
multiplicative dependency as allowing rational exponents, e.g. it is fine to
take some $l$-th root of one of the diophantine roots. See Definition
\ref{def_MultiplicativeIndependence}.} of the diophantine roots of $\Delta
_{K}$. Then%
\[
\operatorname*{limrad}_{z\rightarrow p}\,(1-\left\vert z\right\vert
)E(z)\in\mathbf{Q}(\mu_{\infty},\pi,\{L(1,\chi)\}_{\chi})\text{,}%
\]
where $\chi$ runs through a finite set (depending on $p$) of non-principal
Dirichlet characters of various moduli.
\end{theoremAnn}

In \S \ref{sect_Orthogonality} we provide (complicated) formulae which allow
us the explicit evaluation of these limits. I have no philosophical
explanation why special $L$-values show up in this context. It is mysterious.
See Theorem \ref{thm_A_for_knots_NatBdryCase}.

\subsection{Technical results of independent interest}

In order to prove our main theorems, we need to establish various results
which might be interesting in their own right -- and a priori have little to
do with torsion homology growth.

The principal result in this direction is an evaluation of certain time
averages of ergodic nature:

\begin{theoremAnn}
Suppose $\theta$ is a real number such that either

\begin{itemize}
\item $e^{2\pi i\theta}$ is an algebraic number, or

\item $\theta$ is badly approximable.
\end{itemize}

Then the following holds:

\begin{enumerate}
\item If $\dim_{\mathbf{Q}}\left\langle 1,\theta\right\rangle =2$: For all
$m\in\mathbf{Z}$, we have%
\[
\underset{N\rightarrow\infty}{\lim}\frac{1}{N}\sum_{n=1}^{N}\log\left\vert
1-e^{2\pi in\theta}\right\vert \cdot e^{2\pi imn\theta}=-\frac{1}{2\left\vert
m\right\vert }\delta_{m\neq0}\text{.}%
\]
If $m\in\mathbf{Q}\setminus\mathbf{Z}$, we get a value%
\[
C_{m}\in\mathbf{Q}(\mu_{\infty},\pi,\{L(1,\chi)\}_{\chi})\text{,}%
\]
where $\chi$ ranges over a set of non-principal Dirichlet characters modulo
$2v$ for $v\geq1$ the denominator of $m$ in lowest terms. The values $C_{m}$
only depend on $m$, and are independent of $\theta$.

\item If $\alpha$ is a real number and $\dim_{\mathbf{Q}}\left\langle
1,\theta,\alpha\right\rangle =3$, then for all $m\in\mathbf{Z}$,%
\[
\underset{N\rightarrow\infty}{\lim}\frac{1}{N}\sum_{n=1}^{N}\log\left\vert
1-e^{2\pi in\theta}\right\vert \cdot e^{2\pi imn\alpha}=0\text{.}%
\]

\end{enumerate}
\end{theoremAnn}

This will be Theorem \ref{thm_ortho}. We will use this theorem in order to
understand the behaviour of the function $E$ for knots whose Alexander
polynomial has a diophantine root.

The proof is based on a (very strong)\ version of Weyl Equidistribution due to
Baxa and Schoi\ss engeier \cite{MR1914805}. Their result is only available in
dimension one, but we also need a two-variable version. For our purposes, a
rather minimalistic extension of their proof is sufficient. It is just about
strong enough to treat the computation which we need. This formulation might
be of independent interest:

\begin{theoremAnn}
\emph{(Baxa--Schoi\ss engeier-type Equidistribution)} Suppose $1,\theta
_{1},\ldots,\theta_{d}$ are $\mathbf{Q}$-linearly independent real numbers.
Suppose $F\subseteq\lbrack0,1]\cap\mathbf{Q}$ is finite. Suppose
$f:[0,1]^{d}\rightarrow\mathbf{R}$ is a function in Class $\operatorname*{BSU}%
^{d}(F)$ which admits a singular weight $g$ (see Definition
\ref{def_class_BSU} in the main body of the text) such that
\[
\underset{n\rightarrow\infty}{\lim}\frac{g(\{n\theta_{1}\})}{n}=0\text{.}%
\]
Then%
\[
\underset{N\rightarrow\infty}{\lim}\frac{1}{N}\sum_{n=1}^{N}f(\{n\theta
_{1}\},\ldots,\{n\theta_{d}\})=\int_{[0,1]^{d}}f(\underline{s})\,\mathrm{d}%
\underline{s}\text{.}%
\]

\end{theoremAnn}

See Theorem \ref{thm_bs_multi}. The proof is a mild variation of the method of
\cite{MR1914805}. For $d=1$, we get nothing new.

\begin{acknowledgement}
I heartily thank G.\ W\"{u}stholz for his suggestions on how to use
equidistribution techniques, based on an earlier version of this manuscript in
the spring of 2016. I thank William Stein and the SAGE team.
\end{acknowledgement}

\section{\label{sect_HeuristicMotivation}Heuristics and Motivation}

As we had explained in the introduction, the Silver--Williams asymptotic
(which in the case of knots was developed in earlier work of
Gonz\'{a}lez-Acu\~{n}a and Short \cite{MR1142552} as well as Riley
\cite{MR1041145})%
\[
\underset{r\rightarrow\infty}{\lim}\frac{\log\left\vert H_{1}(X_{r}%
,\mathbf{Z})_{\operatorname*{tor}}\right\vert }{r}=\log\mathcal{M}(\Delta_{K})
\]
suggests that the function%
\[
E(z)=\sum_{r\geq1}\log\left\vert H_{1}(X_{r},\mathbf{Z})_{\operatorname*{tor}%
}\right\vert \cdot z^{r}%
\]
might have a meromorphic continuation to some disc of radius $>1$ with,
likely, a pole of order $2$ at $z=1$. As we shall prove something stronger
later, let us only sketch (under simplifying assumptions) why this speculation
is natural. In rigorous form, we observe the following:

\begin{lemma}
Suppose a power series%
\[
f(z):=\sum_{r\geq1}\log(a_{r})\cdot z^{r}%
\]
converges and admits a meromorphic continuation to some disc of radius $>1$
with a single pole at $z=1$ of order $N+1$ and Laurent expansion%
\[
f(z-1)=\frac{\log C}{(z-1)^{N+1}}+\text{higher order terms}%
\]
at $z=1$. If we additionally know that the limit%
\[
L:=\underset{r\longrightarrow\infty}{\lim}\frac{\log(a_{r})}{r}\qquad\text{(a
\textquotedblleft Silver--Williams asymptotic\textquotedblright)}%
\]
exists, we must have $N=0$ or $1$. If $N=0$, the limit is $L=0$, and if $N=1$,
it is $\log C$.
\end{lemma}

\begin{proof}
Suppose $R>1$ is within the disc of meromorphic continuation. For all $r\geq
0$, the Residue Theorem implies that%
\begin{align*}
\frac{1}{2\pi i}\int_{\left\vert \zeta\right\vert =R}\frac{f(\zeta)}%
{\zeta^{r+1}}\mathrm{d}\zeta & =\log(a_{r})+\frac{1}{2\pi i}\int_{\left\vert
\zeta\right\vert =R}\frac{\log C}{(\zeta-1)^{N+1}\zeta^{r+1}}\mathrm{d}\zeta\\
& =\log(a_{r})+(-1)^{N}\log(C)\cdot r(r-1)\cdots(r-N+1)\\
& \qquad+\text{(degree}<N\text{ polynomial in }r\text{).}%
\end{align*}
Thus, $\log(a_{r})$ is a degree $\leq N$ polynomial in the variable $r$ up to
an error term which can be bounded by%
\[
\left\vert \int_{\left\vert \zeta\right\vert =R}\frac{E(\zeta)}{\zeta^{r+1}%
}\mathrm{d}\zeta\right\vert \leq2\pi R\cdot\int\frac{\left\vert E(\zeta
)\right\vert }{R^{r+1}}\mathrm{d}\zeta=\frac{2\pi}{R^{r}}\cdot\text{const.}%
\]
The constant depends on $R$, but not on $r$. Thus, since $R>1$, so that
$\frac{2\pi}{R^{r}}$ converges to zero as $r\rightarrow+\infty$, we obtain%
\[
\underset{r\longrightarrow\infty}{\lim}\frac{\log(a_{r})}{r}=-(-1)^{N}%
\log(C)\underset{r\longrightarrow\infty}{\lim}\frac{r(r-1)\cdots(r-N+1)}%
{r}\text{.}%
\]
Hence, if the limit on the left-hand side exists at all, we must have $N=0$ or
$1$. If $N=0$, the limit is zero, and if $N=1$, it is $\log C$.
\end{proof}

In summary: The Silver--Williams asymptotic hints at the fact that
\textit{some} analytic continuation might exist. Theorem \ref{thm_A_for_knots}
will then settle this (for a generic knot).

\section{\label{sect_BasicPowerSeries}Preparations}

Let $\mu_{r}\subset\mathbf{C}$ denote the set of all $r$-th roots of unity,
$\mu_{\infty}$ all roots of unity.

\begin{definition}
\label{Def_FunctionR}Let $x\in\mathbf{C}$ be given. Define, at first formally,
a complex power series%
\[
R_{x}(z):=\left.  \sum\limits_{r\geq1}\right.  ^{\prime}\log\left\vert
1-x^{r}\right\vert \cdot z^{r}\text{,}%
\]
where the notation $\left.  \sum\right.  ^{\prime}$ means: We omit the $r$-th
summand if $x^{r}=1$.
\end{definition}

Before we can continue, we need the following deep result of Gelfond:

\begin{proposition}
[Gelfond]\label{Prop_Gelfond}Let $\alpha$ be an algebraic number with
$\left\vert \alpha\right\vert =1$ and which is not a root of unity. Then there
exist real numbers $A,B>0$ such that $\left\vert \alpha^{n}-1\right\vert
>An^{-B}$ holds for all $n\geq1$.
\end{proposition}

This version is sufficiently strong for our purposes. Nevertheless, much
stronger results are known. See Baker--W\"{u}stholz \cite{MR1234835} for a
concrete estimate (also involving the Mahler measure), or one of the several
papers that have appeared since and improve on these bounds (e.g.
\cite{MR2457264}).

\begin{lemma}
\label{Lemma_CriticalUpperBound}Suppose $x\in\mathbf{C}$ is a complex
number.{}

\begin{enumerate}
\item If $\left\vert x\right\vert <1$, then $\left\vert \log\left\vert
1-x^{r}\right\vert \right\vert \leq2\left\vert x\right\vert ^{r}$ for all
sufficiently large real numbers $r$.

\item If $\left\vert x\right\vert >1$, then $\left\vert \log\left\vert
1-x^{r}\right\vert \right\vert <1+r\log\left\vert x\right\vert $ for all
sufficiently large real numbers $r$.

\item If $\left\vert x\right\vert =1$ and $x$ is an algebraic integer, but not
a root of unity, then there exists a constant $C_{x}>0$ such that $\left\vert
\log\left\vert 1-x^{r}\right\vert \right\vert \leq C_{x}\cdot\log(r)$ for all
sufficiently large natural numbers $r$.
\end{enumerate}
\end{lemma}

The first two claims of the lemma are a harmless exercise, the third part
depends on Gelfond's result.

\begin{lemma}
\label{Lemma_ConvergeOfRx}Let $x\in\mathbf{C}^{\times}$ be a given algebraic
number and suppose $x$ is not a root of unity. Then the series $R_{x}$ has
radius of convergence $\geq1$. If $\left\vert x\right\vert <1$, it even has
radius of convergence $\geq\left\vert x\right\vert ^{-1}>1$.
\end{lemma}

\begin{proof}
This follows directly from Lemma \ref{Lemma_CriticalUpperBound} and
$\underset{r\rightarrow\infty}{\lim}\sqrt[r]{\log(r)}=1$.
\end{proof}

The case where $x$ is a root of unity is the only case where a meromorphic
continuation of $R_{x}$ to the entire complex plane is easy to achieve, so let
us handle this case right away.

Below, we will write $(\ldots)$ to denote a holomorphic term.

\begin{proposition}
\label{prop_Rx_AtRootOfUnity}Let $x\in\mu_{\infty}$ be some root of unity.
Then $R_{x}$ admits a meromorphic continuation to the entire complex plane,
given by%
\[
R_{x}(z)=\sum_{l=1}^{m-1}\log\left\vert 1-\zeta_{m}^{l}\right\vert \frac
{z^{l}}{1-z^{m}}\text{,}%
\]
where $m$ denotes the (primitive) order of $x$ and $\zeta_{m}:=x$. Near $z=1
$, $R_{x}$ has the Laurent expansion%
\begin{align*}
& =\frac{1}{m}\log\left(  \frac{1}{m}\right)  \frac{1}{z-1}\\
& -\left(  \frac{m-1}{2}\frac{1}{m}\log\left(  \frac{1}{m}\right)  +\frac
{1}{m}\sum_{l=1}^{m-1}l\cdot\log\left\vert 1-\zeta_{m}^{l}\right\vert \right)
+(z-1)\cdot(\ldots)\text{.}%
\end{align*}
In particular, if $x=1$ then $R_{x}$ is the zero function. Otherwise, $R_{x}$
has a meromorphic continuation to the entire complex plane, with poles
precisely at the finite set $\{x^{n}\mid n\in\mathbf{Z}\}$, all having order
$1$, and the pole at $z=1$ has residue $\frac{1}{m}\log\left(  \frac{1}%
{m}\right)  $.
\end{proposition}

\begin{remark}
This formulation of the proposition is best for our purposes, but there is
also a different perspective relating this to special $L$-values, see Prop.
\ref{Prop_RxAtRootOfUnity_ViaLValuesAtOne}.
\end{remark}

\begin{proof}
Suppose $x$ is a primitive $m$-th root of unity. We write $x=\zeta_{m}$. If
$m=1$, the function $R_{x}$ is zero by definition, so we may assume $m\geq2$.
Then%
\begin{align*}
R_{x}(z)  & =\left.  \sum_{r\geq1}\right.  ^{\prime}\log\left\vert 1-\zeta
_{m}^{r}\right\vert \cdot z^{r}=\sum_{n\geq0}\sum_{l=1}^{m-1}\log\left\vert
1-\zeta_{m}^{mn+l}\right\vert \cdot z^{mn+l}\\
& =\sum_{l=1}^{m-1}\log\left\vert 1-\zeta_{m}^{l}\right\vert \frac{z^{l}%
}{1-z^{m}}\text{.}%
\end{align*}
Now,%
\begin{align*}
\frac{z^{l}}{1-z^{m}}  & =\frac{\left(  \left(  z-1\right)  +1\right)  ^{l}%
}{1-\left(  \left(  z-1\right)  +1\right)  ^{m}}=\frac{\sum_{t=0}^{l}\binom
{l}{t}(z-1)^{t}}{1-\sum_{k=0}^{m}\binom{m}{k}(z-1)^{k}}\\
& =-\frac{1}{m}\frac{1}{z-1}\left(  \sum_{t=0}^{l}\binom{l}{t}(z-1)^{t}%
\right)  \left(  \sum_{r=0}^{\infty}\left(  -\frac{1}{m}\sum_{k=2}^{m}%
\binom{m}{k}(z-1)^{k-1}\right)  ^{r}\right) \\
& =-\frac{1}{m}\frac{1}{z-1}-\frac{l}{m}+\frac{m-1}{2m}+(z-1)(\ldots)\text{.}%
\end{align*}
Thus, expanding $R_{x}$ at $z=1$ yields%
\begin{align*}
& =\sum_{l=1}^{m-1}\log\left\vert 1-\zeta_{m}^{l}\right\vert \left(  -\frac
{1}{m}\frac{1}{z-1}-\frac{l}{m}+\frac{m-1}{2m}\right)  +(z-1)(\ldots)\\
& =-\frac{1}{m}\frac{1}{z-1}\sum_{l=1}^{m-1}\log\left\vert 1-\zeta_{m}%
^{l}\right\vert +\frac{m-1}{2m}\sum_{l=1}^{m-1}\log\left\vert 1-\zeta_{m}%
^{l}\right\vert \\
& -\frac{1}{m}\sum_{l=1}^{m-1}l\cdot\log\left\vert 1-\zeta_{m}^{l}\right\vert
+(z-1)\cdot(\ldots)\text{.}%
\end{align*}
Finally, $\sum_{l=1}^{m-1}\log\left\vert 1-\zeta_{m}^{l}\right\vert
=\log\left\vert \prod_{l=1}^{m-1}(1-\zeta_{m}^{l})\right\vert =\log(m)$ for
$m\geq2$, so%
\[
=\frac{1}{m}\log\left(  \frac{1}{m}\right)  \frac{1}{z-1}-\frac{m-1}{2}%
\frac{1}{m}\log\left(  \frac{1}{m}\right)  -\frac{1}{m}\sum_{l=1}^{m-1}%
l\cdot\log\left\vert 1-\zeta_{m}^{l}\right\vert +(z-1)\cdot(\ldots)\text{,}%
\]
finishing the proof.
\end{proof}

\section{Analytic continuation of $R_{x}$ for $\left\vert x\right\vert >1$}

The next `easy' case is $R_{x}$ for $\left\vert x\right\vert >1$, because it
can be reduced to the case $\left\vert x\right\vert <1$:

\begin{lemma}
\label{Lemma_RSwapXToXInverse}If $\left\vert x\right\vert >1$, then the series
$R_{x}$ converges inside the unit disc and inside this disc, we have the
identity%
\[
R_{x}(z)=R_{x^{-1}}(z)+\log\left\vert x\right\vert \cdot\frac{z}{(z-1)^{2}%
}\text{.}%
\]

\end{lemma}

\begin{proof}
If $\left\vert x\right\vert >1$, we have $\left\vert x^{-1}\right\vert <1$ and
so below the left-hand side converges in the unit disc, $R_{x^{-1}}(z)$ equals%
\begin{align*}
& =\sum_{r\geq1}\log\left\vert 1-x^{-r}\right\vert z^{r}=\sum_{r\geq1}%
\log\left(  \left\vert x^{-r}\right\vert \left\vert x^{r}-1\right\vert
\right)  z^{r}\\
& =\sum_{r\geq1}\log\left(  \left\vert x^{-r}\right\vert \right)  z^{r}%
+\sum_{r\geq1}\log\left\vert 1-x^{r}\right\vert z^{r}=-\log\left\vert
x\right\vert \sum_{r\geq1}rz^{r}+\sum_{r\geq1}\log\left\vert 1-x^{r}%
\right\vert z^{r}\text{,}%
\end{align*}
which is exactly the claim that we wished to prove.
\end{proof}

Next, we want to determine the first terms of the expansion of $R_{x}(z)$
around $z=1$. To this end, let $p$ be the partition function, i.e. $p(n)$ is
the number of distinct presentations of $n$ as a sum of integers $\geq1$,
irrespective of the order. Let $F$ be its generating function, i.e.%
\begin{equation}
F(z)=\sum_{n\geq0}p(n)z^{n}=\prod_{n=1}^{\infty}\frac{1}{1-z^{n}}%
\text{.}\label{lPartFunc}%
\end{equation}

\begin{lemma}
\label{Lemma_RxLocalExpansionAtZEquals1}Suppose $x\in\mathbf{C}^{\times}$ with
$\left\vert x\right\vert <1$. Then the expansion of $R_{x}(z)$ around $z=1$ is
given by%
\[
R_{x}(z)=-\log\left\vert F(x)\right\vert +\sum_{\ell=0}^{\infty}(z-1)^{\ell
+1}\left(  \sum_{r\geq1}\log\left\vert x^{r}-1\right\vert \binom{r}{\ell
+1}\right)
\]
and all the sums in the round brackets on the right converge.
\end{lemma}

\begin{proof}
The convergence of the sums is harmless: By Lemma
\ref{Lemma_CriticalUpperBound}, since $\left\vert x\right\vert <1$, we have%
\[
\left\vert \log\left\vert x^{r}-1\right\vert \binom{r}{\ell+1}\right\vert
\leq2\left\vert x\right\vert ^{r}\binom{r}{\ell+1}%
\]
and then $\sum_{r\geq1}\log\left\vert x^{r}-1\right\vert \binom{r}{\ell+1} $
is dominated by $2\sum_{r\geq1}\left\vert x\right\vert ^{r}\binom{r}{\ell+1}$
and the latter is the standard estimate for convergence of a power series. In
particular, since $\left\vert x\right\vert <1$, this converges. We compute%
\[
R_{x}(z)=\sum_{r\geq1}\log\left\vert 1-x^{r}\right\vert z^{r}=\sum_{r\geq
1}\log\left\vert 1-x^{r}\right\vert (z^{r}-1)+\sum_{r\geq1}\log\left\vert
1-x^{r}\right\vert \text{,}%
\]
but in view of Equation \ref{lPartFunc} the rightmost term is just a special
value of the generating function of the partition function. Thus,%
\begin{align*}
& =-\log\left\vert F(x)\right\vert +\sum_{r\geq1}\log\left\vert x^{r}%
-1\right\vert \cdot(z-1)(1+z+\cdots+z^{r-1})\\
& =-\log\left\vert F(x)\right\vert +\sum_{r\geq1}\log\left\vert x^{r}%
-1\right\vert \cdot(z-1)\sum_{a=0}^{r-1}((z-1)+1)^{a}%
\end{align*}
and expanding this using the binomial formula yields%
\[
=-\log\left\vert F(x)\right\vert +\sum_{\ell=0}^{\infty}(z-1)^{\ell+1}\left(
\sum_{r\geq1}\log\left\vert x^{r}-1\right\vert \cdot\sum_{a=0}^{r-1}\binom
{a}{\ell}\right)  \text{.}%
\]
By a standard identity on binomial coefficients, the innermost sum equals
$\binom{r}{\ell+1}$, confirming our claim.
\end{proof}

\begin{aside}
The modular discriminant $\triangle$ is a weight $12$ modular form for
$\operatorname*{SL}_{2}(\mathbf{Z})$ and given by%
\[
\triangle(\tau)=q\prod_{m\geq1}(1-q^{m})^{24}\qquad\text{for}\qquad q=e^{2\pi
i\tau}\text{, }\tau\in\mathbf{H}\text{.}%
\]
For its logarithm, we obtain $\frac{1}{24}\left(  \log\triangle(\tau)-\log
q\right)  =\sum_{m\geq1}\log(1-q^{m})$ and thus for $F(x)\in\mathbf{C}%
\setminus\mathbf{R}_{\leq0}$, we get%
\[
-\log\left\vert F(x)\right\vert =\sum_{n=1}^{\infty}\log(1-x^{n})+\sum
_{n=1}^{\infty}\log(1-\overline{x}^{n})\text{.}%
\]
Thus, we may spell out $-\log\left\vert F(x)\right\vert $ as an expression in
the function $\log\triangle$. For $\gamma\in\operatorname*{SL}_{2}%
(\mathbf{Z})$, the transformation behaviour of $\log\triangle$ has an explicit
description in terms of Dedekind symbols/sums, \cite[Ch. 4]{MR0357299}. This
gives us a rich structure on the constant coefficient of the Laurent expansion
of $R_{x}$ at $z=1$. I have been wondering whether this structure would give
rise to some visible patterns in the behaviour of $R_{x}$ when changing $x$,
but I\ have not been able to isolate anything meaningful. Maybe somebody else
has an idea.
\end{aside}

\section{Analytic continuation of $R_{x}$ for $\left\vert x\right\vert <1$}

\subsection{Choice for complex exponentiation}

We shall mostly work with the principal branch of the logarithm. For us, this
means that it is defined on $\mathbf{C}^{\times}$ and given by%
\begin{equation}
\log(re^{i\theta})=\log r+i\theta\qquad\text{for }\theta\in(-\pi,+\pi
]\text{.}\label{lbranch1}%
\end{equation}
Based on this choice of a logarithm, $x^{s}:=\exp(s\cdot\log x)$ is our choice
of the meaning of complex exponentiation. We use capital letters
\textquotedblleft$\operatorname*{Log}$\textquotedblright\ whenever we want to
stress that an arbitrary branch of the logarithm can be used.

\begin{remark}
Usually, both $x$ and $s$ will be complex numbers, so we will have to be very
careful with deceptive functional equations, e.g. $e^{st}\neq(e^{s})^{t}$ for
general $s,t\in\mathbf{C}$.
\end{remark}

\begin{lemma}
For $x\in\mathbf{C}^{\times}$ and $s\in\mathbf{C}$, one has $\left\vert
x^{s}\right\vert =\left\vert x\right\vert ^{\operatorname*{Re}s}\cdot
e^{-\arg(x)\operatorname{Im}s}$.
\end{lemma}

\begin{proof}
We have, for $x=\left\vert x\right\vert e^{i\theta}$, $\theta:=\arg x\in
(-\pi,\pi]$ and $s=a+bi$ with $a,b\in\mathbf{R}$,%
\begin{align*}
x^{s}  & =e^{s\log x}=e^{(a+bi)(\log\left\vert x\right\vert +i\theta
)}=e^{(a\log\left\vert x\right\vert -b\theta)+i(a\theta+b\log\left\vert
x\right\vert )}\\
& =e^{a\log\left\vert x\right\vert }\cdot e^{-b\theta}\cdot e^{i(a\theta
+b\log\left\vert x\right\vert )}=\left\vert x\right\vert ^{\operatorname*{Re}%
s}\cdot e^{-\arg(x)\operatorname{Im}s}\cdot e^{i(a\theta+b\log\left\vert
x\right\vert )}%
\end{align*}
and taking absolute values, we get the claim.
\end{proof}

\subsection{Integral forms}

\begin{proposition}
\label{prop_TinyBoxAbelPlana}Let $K>0$ be real. Suppose $h:[1,+\infty)\times
i[-K,K]\rightarrow\mathbf{C}$ is a function which admits a holomorphic
continuation to an open neighbourhood of this box. Then for all integers
$1\leq a<b$, we have
\begin{align*}
\sum_{n=a}^{b}h(n)  & =\frac{1}{2}h(a)+\frac{1}{2}h(b)+\int_{a}^{b}%
h(s)\mathrm{d}s\\
& +i\int_{0}^{K}\frac{h(a+iy)-h(a-iy)}{e^{2\pi y}-1}\mathrm{d}y-i\int_{0}%
^{K}\frac{h(b+iy)-h(b-iy)}{e^{2\pi y}-1}\mathrm{d}y\\
& -\int_{a+iK}^{b+iK}\frac{h(s)}{1-e^{-2\pi is}}\mathrm{d}s+\int_{a-iK}%
^{b-iK}\frac{h(s)}{e^{2\pi is}-1}\mathrm{d}s\text{.}%
\end{align*}

\end{proposition}

This is a modification of the Abel--Plana contour. A modern reference is
Olver's book \cite[Ch. 8, \S 3]{MR1429619}. However, our version is a little
more complicated as we cannot let $K$ go to infinity.

\begin{proof}
We begin with the cotangent series, spelled out below, which is compactly
convergent in $\mathbf{C}\backslash\mathbf{Z}$. We may consider the positively
oriented contour $\mathsf{C}$ around the set $\mathsf{Box}:=[a,b]\times
i[-K,K]$, and slightly modify it at $a$ and $b$ by cutting out little
semi-circles of radius $\delta>0$, chosen sufficiently small, say $<\frac
{1}{4}K$, as in the following figure:%
\[%
\begin{array}
[c]{ccc}%
\raisebox{-0.4583in}{\includegraphics[
height=0.8371in,
width=1.299in
]%
{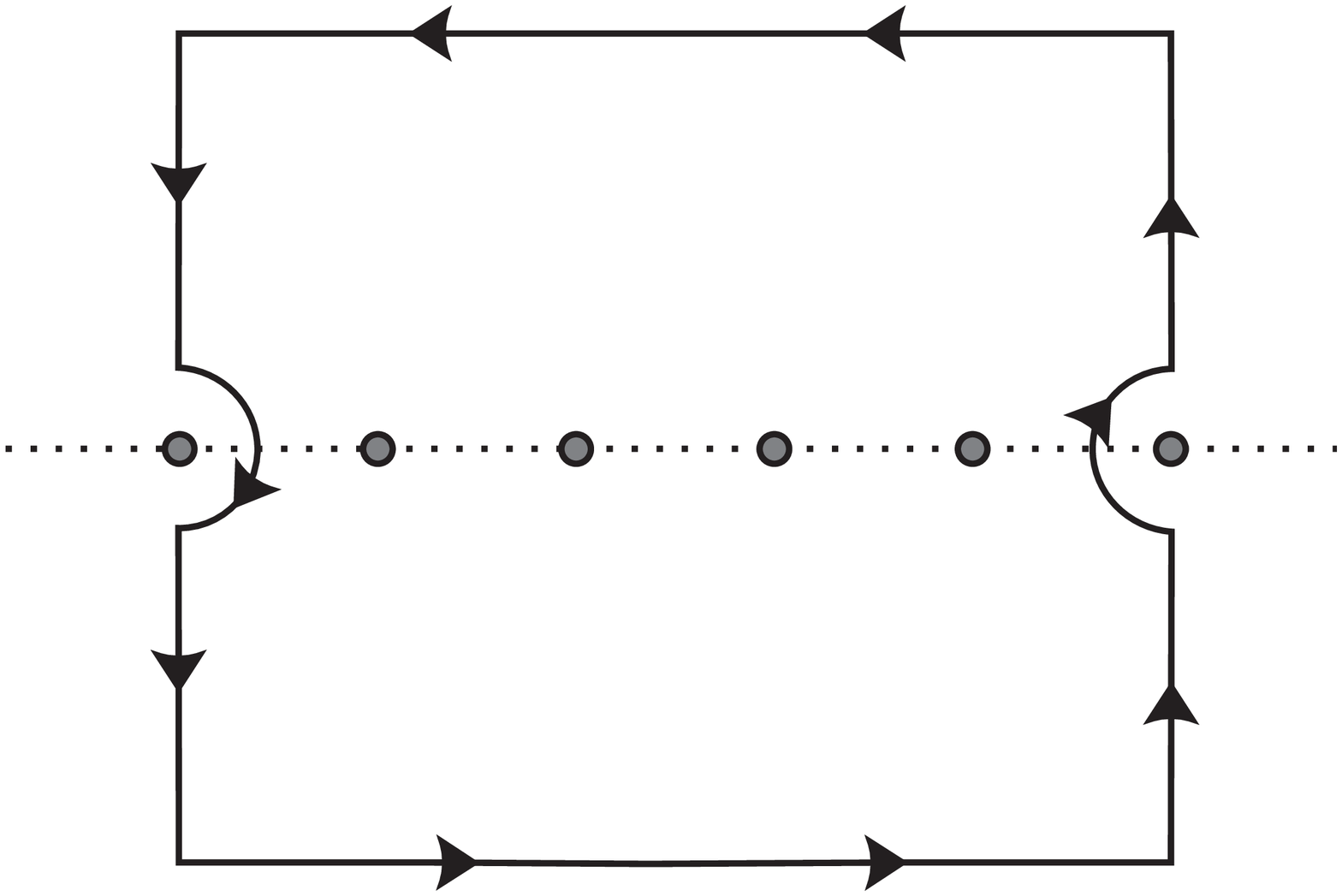}%
}
& \qquad & \pi\cot(\pi z)=\frac{1}{z}+%
{\displaystyle\sum_{n=1}^{\infty}}
\left(  \frac{1}{z+n}+\frac{1}{z-n}\right)  \text{,}%
\end{array}
\]
The grey dots represent the integers $\mathbf{Z}$, that is the poles of the
cotangent and we need these semi-circles to avoid passing right through a
pole. If $\mathbf{H}$ denotes the upper half-plane, the Residue\ Theorem gives
us%
\begin{align}
\sum_{n=a+1}^{b-1}h(n)-\int_{a+\delta}^{b-\delta}h(z)\mathrm{d}z  & =\frac
{1}{2i}\int_{\mathsf{C}}h(z)\cot(\pi z)\mathrm{d}z\label{lvins17}\\
& +\frac{1}{2}\int_{\mathbf{H}\cap\mathsf{C}}h(z)\mathrm{d}z-\frac{1}{2}%
\int_{-\mathbf{H}\cap\mathsf{C}}h(z)\mathrm{d}z\text{.}\nonumber
\end{align}
The second line arises from writing the integral along $[a,b]$ either by going
through the upper half of the contour, or the lower half, and since both work
equally well, we can just as well take the average of these two (equal)
values. As is well-known, we have $\frac{1}{i}\cot(z)=\frac{1}{i}\frac
{\cos(z)}{\sin(z)}=\frac{e^{iz}+e^{-iz}}{e^{iz}-e^{-iz}}=\frac{e^{2iz}%
+1}{e^{2iz}-1}$, which may be simplified to either of the following two
expressions $1+2\frac{1}{e^{2iz}-1}$ or $2\frac{1}{1-e^{-2iz}}-1$, whichever
is more convenient. We subdivide the contour $\mathsf{C}$ of the first
integral into its upper and lower part as well. This yields
\begin{align*}
& =\frac{1}{2i}\int_{\mathbf{H}\cap\mathsf{C}}h(z)\cot(\pi z)\mathrm{d}%
z+\frac{1}{2i}\int_{-\mathbf{H}\cap\mathsf{C}}h(z)\cot(\pi z)\mathrm{d}z\\
& +\frac{1}{2}\int_{\mathbf{H}\cap\mathsf{C}}h(z)\mathrm{d}z-\frac{1}{2}%
\int_{-\mathbf{H}\cap\mathsf{C}}h(z)\mathrm{d}z
\end{align*}
and then, using the aforementioned two presentations, this can be rewritten as%
\begin{align*}
& =\int_{\mathbf{H}\cap\mathsf{C}}h(z)\left(  \frac{1}{1-e^{-2\pi iz}}%
-\frac{1}{2}\right)  \mathrm{d}z+\int_{-\mathbf{H}\cap\mathsf{C}}h(z)\left(
\frac{1}{2}+\frac{1}{e^{2\pi iz}-1}\right)  \mathrm{d}z\\
& +\frac{1}{2}\int_{\mathbf{H}\cap\mathsf{C}}h(z)\mathrm{d}z-\frac{1}{2}%
\int_{-\mathbf{H}\cap\mathsf{C}}h(z)\mathrm{d}z=\int_{\mathbf{H}\cap
\mathsf{C}}\frac{h(z)}{1-e^{-2\pi iz}}\mathrm{d}z+\int_{-\mathbf{H}%
\cap\mathsf{C}}\frac{h(z)}{e^{2\pi iz}-1}\mathrm{d}z\text{,}%
\end{align*}
involving a convenient cancellation of terms. Firstly, one checks using the
continuity of $h$ that for $\delta\rightarrow0$ one gets additional terms of
$-\frac{1}{2}h(a)$ and $-\frac{1}{2}h(b)$, so adding the summands for $n=a,b$
to the left-hand side of Equation \ref{lvins17} (so that the sum now reads
$\sum_{a}^{b}$), we need to add $\frac{1}{2}h(a)+\frac{1}{2}h(b)$ on the
right-hand side to keep it balanced. Replace $\mathsf{C}$ by a straight left
and right edge with a tiny omission of radius $\delta>0$ instead of the
semi-circles, and this will be justified later as we shall see that the
remaining limit will exist for $\delta\rightarrow0$. For the left edge, pick
the curve $z(y):=a+iy$ and use that $e^{\pm2\pi i(a+yi)}=e^{\mp2\pi y} $
because $a\in\mathbf{Z}$. Thus,%
\begin{align}
\int_{\mathbf{H}\cap\mathsf{C}_{left}}\frac{h(z)}{1-e^{-2\pi iz}}\mathrm{d}z
& =-i\int_{\delta}^{K}\frac{h(a+iy)}{1-e^{2\pi y}}\mathrm{d}y\label{lvins21}\\
\int_{-\mathbf{H}\cap\mathsf{C}_{left}}\frac{h(z)}{e^{2\pi iz}-1}\mathrm{d}z
& =-i\int_{\delta}^{K}\frac{h(a-iy)}{e^{2\pi y}-1}\mathrm{d}y\text{.}\nonumber
\end{align}
The same works for the right edge with $b$ in place of $a$. Finally, take the
limit $\delta\rightarrow0$. This is harmless: Both numerator and denominator
are zero for $y=0$, but by L'H\^{o}pital's rule, the limit of the integrand
for $y\rightarrow0$ agrees with $\underset{y\rightarrow0}{\lim}\,\frac{i}%
{2\pi}(h^{\prime}(b+iy)+h^{\prime}(b-iy))e^{-2\pi y}=\frac{i}{\pi}h^{\prime
}(b)$, and in particular this limit exists.
\end{proof}

Below, we shall repeatedly need a case distinction \textquotedblleft$\pm
$\textquotedblright. Either case comes with an assumption, which we shall
repeatedly need, so we give it a name:

\begin{assumption}
[$\mathfrak{C}^{\pm}$]Suppose $x\in\mathbf{C}^{\times}$ and $\left\vert
x\right\vert <1$ is chosen. If $\mp\arg(x)>0$, choose a real number $K$ such
that%
\[
0<K<-\frac{\log\left\vert x\right\vert }{\mp\arg x}\text{.}%
\]
If $\mp\arg(x)\leq0$, we only assume $K>0$ and do not impose an upper bound.
\end{assumption}

\begin{lemma}
\label{lemma_KBound}Fix a case \textquotedblleft$\pm$\textquotedblright. If
Assumption $\mathfrak{C}^{\pm}$ is satisfied, then for all $s=u\pm iv$ with
$u\geq1$ and $0\leq v<K$, we have $\left\vert x^{s}\right\vert <1$. If $u>1$
and $0\leq v\leq K$, we also have $\left\vert x^{s}\right\vert <1$.
\end{lemma}

\begin{proof}
We have $\left\vert x^{s}\right\vert =\left\vert e^{s\log x}\right\vert
=\left\vert e^{(u\pm iv)(\log\left\vert x\right\vert +i\arg x)}\right\vert
=\left\vert e^{u\log\left\vert x\right\vert \mp v\arg x}\right\vert
=\left\vert x\right\vert ^{u}\cdot e^{\mp(\arg x)v}$. If $\mp\arg(x)\leq0 $,
we surely have $\left\vert e^{\mp(\arg x)v}\right\vert \leq1$ and $\left\vert
x\right\vert ^{u}<1$ by $u\geq1$ and $\left\vert x\right\vert <1$, proving the
claim in this case. Otherwise, if $\mp\arg(x)>0$, we have $v<K<-\frac
{\log\left\vert x\right\vert }{\mp\arg x}$ (resp. equality) and thus
$\left\vert x\right\vert ^{u}\cdot e^{\mp(\arg x)v}<\left\vert x\right\vert
^{u-1}$ and for $u\geq1$ (resp. $>1$) the claim follows from $\left\vert
x\right\vert <1$.
\end{proof}

\begin{proposition}
\label{W_prop_realaxis_continue}Suppose $x\in\mathbf{C}^{\times}$ and
$\left\vert x\right\vert <1$. Then for all $\operatorname{Re}w>0$, the
integral%
\[
A_{w,x}(w):=\int_{1}^{\infty}\log(1-x^{s})e^{-ws}\mathrm{d}s
\]
converges and defines a holomorphic function on the right half-plane. The
series%
\[
\tilde{A}_{w,x}(w):=\sum_{n\geq1}\frac{1}{n}\frac{e^{n\log x-w}}{n\log x-w}%
\]
is compactly convergent everywhere outside $\mathbf{Z}_{\geq1}\cdot\log x$ in
the complex plane and defines a meromorphic continuation with poles at these points.
\end{proposition}

Note that all these poles lie in the open left half-plane.

\begin{proof}
\textit{(Step 1)} For real numbers $s\geq1$, we have $\left\vert
x^{s}\right\vert =\left\vert x\right\vert ^{s}<1$, so inside the range of
integration, $\log(1-x^{s})$ can be expanded as a uniformly convergent power
series.\ Concretely, for $1\leq a<b$ we get
\[
\int_{a}^{b}\log(1-x^{s})e^{-ws}\mathrm{d}s=-\int_{a}^{b}\sum_{n\geq1}%
\frac{x^{sn}}{n}e^{-ws}\mathrm{d}s
\]
and swapping integral and sum, this can be integrated in an explicit fashion:%
\begin{align*}
& =-\sum_{n\geq1}\frac{1}{n}\int_{a}^{b}x^{sn}e^{-ws}\mathrm{d}s=-\sum
_{n\geq1}\frac{1}{n}\int_{a}^{b}e^{(n\log x-w)s}\mathrm{d}s\\
& =-\sum_{n\geq1}\frac{1}{n}\left\{
\begin{array}
[c]{ll}%
\left.  \frac{e^{(n\log x-w)s}}{n\log x-w}\right\vert _{s=a}^{s=b} & \text{for
}n\log x-w\neq0\\
b-a & \text{otherwise.}%
\end{array}
\right.
\end{align*}
\textit{(Step 2)} Now, assume $\operatorname{Re}w>0$ and consider the case
$b\rightarrow+\infty$. Then thanks to $\operatorname{Re}w>0$, we have
$e^{-bw}\rightarrow0$, and so $\left\vert e^{n\log x-w}\right\vert
^{b}=\left\vert x^{n}\right\vert ^{b}e^{-bw}\longrightarrow0$. Thus,
\textit{Step 1} implies that $A_{w,x}=\tilde{A}_{w,x}$ in the right
half-plane. However, $\tilde{A}_{w,x}$ is uniformly convergent in any
compactum outside the poles, even without assuming $\operatorname{Re}w>0$. To
see this, note that $\left\vert e^{n\log x-w}\right\vert =\left\vert
x\right\vert ^{n}\cdot e^{-\operatorname{Re}w}$ and since the denominator can
be bounded in any compactum outside the poles, the convergence is dominated by
a convergent geometric series because of $\left\vert x\right\vert <1$. Thus,
$\tilde{A}_{w,x}$ is meromorphic in the entire complex plane.
\end{proof}

\begin{proposition}
\label{prop_analytic_M}Suppose we are in the situation of Assumption
$\mathfrak{C}^{\pm}$. Let $a\geq1$ be an integer and $0<\delta<K$. Then for
all $\operatorname{Re}w>0$, the integral%
\[
M_{a,K}^{\pm}(w):=\int_{\delta}^{K}\frac{\log(1-x^{a\pm iy})e^{-w(a\pm iy)}%
}{e^{2\pi y}-1}\mathrm{d}y
\]
converges and defines a holomorphic function. The series%
\[
\tilde{M}_{a,K}^{\pm}(w):=-e^{-wa}\sum_{m,n\geq1}\frac{x^{ma}}{m}\left\{
\begin{array}
[c]{ll}%
\left.  \frac{e^{(\pm im\log x\mp iw-2\pi n)y}}{\pm im\log x\mp iw-2\pi
n}\right\vert _{y=\delta}^{y=K} & \text{for }\pm im\log x\mp iw-2\pi n\neq0\\
K-\delta & \text{otherwise.}%
\end{array}
\right.
\]
is compactly convergent in the entire complex plane and defines a holomorphic
continuation of $M_{a}^{\pm}$.
\end{proposition}

\begin{proof}
\textit{(Step 1)} We assume $\operatorname{Re}w>0$. Then the integral%
\[
M_{a,K}^{\pm}(w):=\int_{\delta}^{K}\frac{\log(1-x^{a\pm iy})e^{-w(a\pm iy)}%
}{e^{2\pi y}-1}\mathrm{d}y=\int_{\delta}^{K}\frac{\log(1-x^{a\pm
iy})e^{-w(a\pm iy)}}{e^{2\pi y}(1-e^{-2\pi y})}\mathrm{d}y\text{.}%
\]
is convergent. Since $y\geq\delta>0$ within the range of integration, we have
$\left\vert e^{-2\pi y}\right\vert <1$ and we may expand $1/(1-e^{-2\pi y})$
as a uniformly convergent geometric series. We obtain%
\[
=\int_{\delta}^{K}\log(1-x^{a\pm iy})e^{-w(a\pm iy)}\sum_{n\geq1}e^{-2\pi
ny}\mathrm{d}y\text{.}%
\]
By Assumption $\mathfrak{C}^{\pm}$ and Lemma \ref{lemma_KBound}, we have
$\left\vert x^{a\pm iy}\right\vert <1$ since $a\geq1$ and $0\leq y<K$. Thus,
again we may expand the logarithm and we obtain%
\begin{align*}
& =-\int_{\delta}^{K}\sum_{m\geq1}\frac{x^{(a\pm iy)m}}{m}e^{-w(a\pm iy)}%
\sum_{n\geq1}e^{-2\pi ny}\mathrm{d}y\\
& =-\sum_{m\geq1}\sum_{n\geq1}\frac{1}{m}\int_{\delta}^{K}e^{am\log x}e^{\pm
iym\log x}e^{-w(a\pm iy)}e^{-2\pi ny}\mathrm{d}y
\end{align*}
by uniform convergence. Then,%
\begin{align}
& =-e^{-wa}\sum_{m,n\geq1}\frac{x^{ma}}{m}\int_{\delta}^{K}e^{(\pm im\log x\mp
iw-2\pi n)y}\mathrm{d}y\nonumber\\
& =-e^{-wa}\sum_{m,n\geq1}\frac{x^{ma}}{m}\left\{
\begin{array}
[c]{ll}%
\left.  \frac{e^{(\pm im\log x\mp iw-2\pi n)y}}{\pm im\log x\mp iw-2\pi
n}\right\vert _{y=\delta}^{y=K} & \text{for }\pm im\log x\mp iw-2\pi n\neq0\\
K-\delta & \text{otherwise.}%
\end{array}
\right. \label{l_gW1}%
\end{align}
\textit{(Step 2)} Next, we claim that this last expression converges for all
$w\in\mathbf{C}$, dropping the assumption $\operatorname{Re}w>0$. To this end,
we compute%
\[
X_{n,m}:=\left\vert x^{ma}\right\vert \cdot\left\vert e^{(\pm im\log x\mp
iw-2\pi n)y}\right\vert =\left\vert e^{m(a\log x\pm iy\log x)}\right\vert
\cdot\left\vert e^{\mp iwy}\right\vert \cdot\left\vert e^{-2\pi y}\right\vert
^{n}%
\]
Now, since $y>0$, we have $\left\vert e^{-2\pi y}\right\vert ^{n}<1$ for all
$n$. The term $\left\vert e^{\mp iwy}\right\vert $ does not depend on $n$ nor
$m$. Finally,%
\[
\left\vert e^{m(a\log x\pm iy\log x)}\right\vert =\left\vert e^{m(a(\log
\left\vert x\right\vert +i\arg x)\pm iy(\log\left\vert x\right\vert +i\arg
x))}\right\vert =\left\vert e^{(a\log\left\vert x\right\vert \mp y\arg
x)}\right\vert ^{m}%
\]
If $\mp\arg(x)\leq0$, then $\mp y\arg x\leq0$ and by $a\geq1$ and $\left\vert
x\right\vert <1$, it follow that $e$ has a negative exponent. Thus, we get a
term of the shape $\theta^{m}$ for some $0<\theta<1$. Now: Since $\left\vert
e^{-2\pi y}\right\vert <1$ and $\theta<1$, the terms $X_{n,m}$ are dominated
by a geometric series both in the variables $n$ and $m $. It is easy to check
that the apparent poles in the top case of Equation \ref{l_gW1} are all
removable, and in fact the holomorphic continuation is given by switching to
the second case. In particular, we get an everywhere compactly convergent
series of holomorphic functions.
\end{proof}

\begin{proposition}
Suppose we are in the situation of Assumption $\mathfrak{C}^{\pm}$. Let
$a\geq1$ be an integer. Then for all $\operatorname{Re}w>0$, the integral%
\[
T_{K}^{+}(w):=-\int_{1+iK}^{\infty+iK}\frac{\log(1-x^{s})e^{-ws}}{1-e^{-2\pi
is}}\mathrm{d}s
\]
resp.%
\[
T_{K}^{-}(w):=+\int_{1-iK}^{\infty-iK}\frac{\log(1-x^{s})e^{-ws}}{e^{2\pi
is}-1}\mathrm{d}s
\]
converges and defines a holomorphic function in the open right half-plane. The
series%
\begin{equation}
\tilde{T}_{K}^{\pm}(w)=\pm\sum_{m\geq1}\sum_{n\geq1}\frac{1}{m}\frac{e^{(m\log
x-w\pm2\pi in)(1\pm iK)}}{m\log x-w\pm2\pi in}\label{ltg1}%
\end{equation}
is compactly convergent everywhere outside $\mathbf{Z}_{\geq1}\cdot\log
x\pm\mathbf{Z}_{\geq1}\cdot2\pi i$ in the entire complex plane. It defines a
meromorphic continuation with poles at the said points.
\end{proposition}

Here the integration from \textquotedblleft$1\pm iK$\textquotedblright\ to
\textquotedblleft$\infty\pm iK$\textquotedblright\ is meant to denote
integration along the curve $\gamma(t):=t\pm iK$ for $t\in\lbrack1,+\infty)$.

\begin{proof}
\textit{(Step 1)} For $T^{+}$ we get%
\[
-\int_{a+iK}^{b+iK}\frac{\log(1-x^{s})e^{-ws}}{1-e^{-2\pi is}}\mathrm{d}%
s\text{.}%
\]
We compute%
\[
-\int_{a+iK}^{b+iK}\frac{\log(1-x^{s})e^{-ws}}{1-e^{-2\pi is}}\mathrm{d}%
s=-\int_{a+iK}^{b+iK}\frac{\log(1-x^{s})e^{-ws}}{e^{-2\pi is}(e^{2\pi is}%
-1)}\mathrm{d}s
\]
and since we are in the upper half-plane $\left\vert e^{2\pi is}\right\vert
=e^{-2\pi K}<1$ by $K>0$. Thus, we may expand this as a geometric series which
is uniformly convergent in the range of integration,%
\[
=\int_{a+iK}^{b+iK}\log(1-x^{s})e^{-ws}\sum_{n\geq1}e^{2\pi ins}%
\mathrm{d}s\text{.}%
\]
The case of $T^{-}$ is very much analogous: We get%
\[
-\int_{a-iK}^{b-iK}\log(1-x^{s})e^{-ws}\sum_{n\geq1}e^{-2\pi ins}\mathrm{d}s
\]
instead. Thus, up to signs, we may handle both cases simultaneously. By
Assumption $\mathfrak{C}^{\pm}$, $1\leq a<b$ and Lemma \ref{lemma_KBound}, we
can expand the logarithm,%
\begin{align*}
& =\mp\int_{a\pm iK}^{b\pm iK}\sum_{m\geq1}\frac{x^{sm}}{m}e^{-ws}\sum
_{n\geq1}e^{\pm2\pi ins}\mathrm{d}s=\mp\sum_{m\geq1}\sum_{n\geq1}\frac{1}%
{m}\int_{a\pm iK}^{b\pm iK}e^{(m\log x-w\pm2\pi in)s}\mathrm{d}s\\
& =\mp\sum_{m\geq1}\sum_{n\geq1}\frac{1}{m}\left\{
\begin{array}
[c]{ll}%
\left.  \frac{e^{(m\log x-w\pm2\pi in)s}}{m\log x-w\pm2\pi in}\right\vert
_{s=a\pm iK}^{s=b\pm iK} & \text{for }m\log x-w\pm2\pi in\neq0\\
b-a & \text{otherwise.}%
\end{array}
\right.
\end{align*}
\textit{(Step 2)} We will estimate the magnitute of the integrand/summands.
Suppose $s=u\pm iv$ with $u,v\in\mathbf{R}$, $v\geq0$ (it makes sense to
assume $v\geq0$ because for handling $T^{+}$ it suffices to have estimates
which are valid in the upper half-plane, and correspondingly for $T^{-}$ in
the lower half-plane). Now,%
\[
\left\vert e^{(m\log x-w\pm2\pi in)s}\right\vert =\left\vert e^{mu\log
\left\vert x\right\vert \mp mv\arg(x)}\right\vert \cdot\left\vert
e^{-ws}\right\vert \cdot\left\vert e^{-2\pi v}\right\vert ^{n}\text{.}%
\]
For all $s$ in the range of integration, the imaginary part is $\pm K$, so
$\left\vert e^{-2\pi K}\right\vert <1$ by $K>0$. It follows that our terms are
dominated by a geometric series in $n$. Moreover, $\left\vert e^{mu\log
\left\vert x\right\vert \mp mv\arg(x)}\right\vert =\left\vert e^{u\log
\left\vert x\right\vert \mp v\arg(x)}\right\vert ^{m}$ and by Assumption
$\mathfrak{C}^{\pm}$ we have: In the range of integration, $u\geq1$ and
$\log\left\vert x\right\vert $ is negative, so if $\mp\arg(x)\leq0$, the
exponent is negative and thus we also have domination by a geometric series in
$m$. If $\mp\arg(x)>0$ on the other hand, we have
\[
v<K<-\frac{\log\left\vert x\right\vert }{\mp\arg x}%
\]
by Assumption $\mathfrak{C}^{\pm}$. Thus, $\log\left\vert x\right\vert \mp
v\arg x<0$, and since $\log\left\vert x\right\vert <0$ and $u\geq1$, adding
$(u-1)\log\left\vert x\right\vert \leq0$ yields%
\[
u\log\left\vert x\right\vert \mp v\arg x<(u-1)\log\left\vert x\right\vert
\leq0\text{.}%
\]
Thus, again the exponent is negative, giving domination by a geometric series
also in $m$. Finally, note that for $w=c+di$, the term $\left\vert
e^{-ws}\right\vert $ can be evaluated to be%
\[
\left\vert e^{-ws}\right\vert =\left\vert e^{-(c+di)u\mp iv(c+di)}\right\vert
=\left\vert e^{-cu}\right\vert \cdot\left\vert e^{\pm vd}\right\vert \text{.}%
\]
So, if in \textit{Step 1} we let $b\rightarrow+\infty$, and for $c>0$,
$\left\vert e^{-cu}\right\vert $ goes to zero. As this universally bounds all
coefficients, we see that the right-hand side boundary term of the integration
vanishes \textit{if} we assume $\operatorname{Re}w>0$ (i.e. $c>0$). Moreover,
our upper bound of exponential decay in both $n$ and $m$ shows that outsides
the poles, we have uniform convergence of the series in line \ref{ltg1} in any compactum.
\end{proof}

\begin{proposition}
\label{prop_AnalyticContinuationForQx}Suppose $x\in\mathbf{C}^{\times}$ and
$\left\vert x\right\vert <1$. Then for all $\operatorname{Re}w>0$, the series%
\[
Q_{x}(w):=\sum_{n=1}^{\infty}\log(1-x^{n})e^{-wn}%
\]
is compactly convergent and defines a holomorphic function in the right
half-plane. It admits a meromorphic continuation to the entire complex plane
with poles at $\mathbf{Z}_{\geq1}\cdot\log x+\mathbf{Z}\cdot2\pi i$.
\end{proposition}

\begin{proof}
Firstly, we choose some $K$ such that the Assumptions $\mathfrak{C}^{+}$ and
$\mathfrak{C}^{-}$ are both met. This is always possible: If $\mp\arg(x)\leq
0$, we only need $K>0$, and if $\mp\arg(x)>0$, note that the right-hand side
in%
\[
0<K<-\frac{\log\left\vert x\right\vert }{\mp\arg x}%
\]
is strictly positive since $\left\vert x\right\vert <1$. Thus, some $K$ in
between these bounds exists and we fix a choice. Assume $\operatorname{Re}%
w>0$. Now we apply Prop. \ref{prop_TinyBoxAbelPlana} for $b\rightarrow+\infty
$. Easy estimates show that $h(b)\rightarrow0$ and the right edge term
vanishes as well, and by combining the previous propositions, we get%
\begin{align*}
& \sum_{n=1}^{\infty}\log(1-x^{n})e^{-wn}=\frac{1}{2}\log(1-x)e^{-w}%
+\sum_{n\geq1}\frac{1}{n}\frac{e^{n\log x-w}}{n\log x-w}+i(\tilde{M}_{1,K}%
^{+}(w)-\tilde{M}_{1,K}^{-}(w))\\
& \qquad-\sum_{m\geq1}\sum_{n\geq1}\frac{1}{m}\frac{e^{(m\log x-w+2\pi
in)(1+iK)}}{m\log x-w+2\pi in}-\sum_{m\geq1}\sum_{n\geq1}\frac{1}{m}%
\frac{e^{(m\log x-w-2\pi in)(1-iK)}}{m\log x-w-2\pi in}\text{.}%
\end{align*}
This identity holds only for $\operatorname{Re}w>0$, but the right-hand side
is a meromorphic function in the entire complex plane thanks to the quoted
propositions. Its poles are located at%
\[
\{\mathbf{Z}_{\geq1}\cdot\log x\}\cup\{\mathbf{Z}_{\geq1}\cdot\log
x+\mathbf{Z}_{\neq0}\cdot2\pi i\}=\mathbf{Z}_{\geq1}\cdot\log x+\mathbf{Z}%
\cdot2\pi i\text{,}%
\]
since $\tilde{M}_{1}^{\pm}$ is holomorphic in the entire complex plane.
Finally, connectedness of $\mathbf{C}$ minus these poles and the identity
principle imply that our choice of $K$ does not affect the continuation.
\end{proof}

\begin{remark}
It is not surprising that the pole locus is $2\pi i$-periodic, since $Q_{w}$
is clearly periodic under $w\mapsto w+2\pi i$. Note that most of the summands
that we had individually analytically continued do not enjoy such a
periodicity by themselves. Only their sum is periodic.
\end{remark}

Now we are ready to prove our first key ingredient for the analytic continuation.

\begin{theorem}
\label{thm_AnalyticContinuationForRx}Suppose $x\in\mathbf{C}^{\times}$ and
$\left\vert x\right\vert <1$.

\begin{enumerate}
\item Then the series $R_{x}(z)$ admits a meromorphic continuation to the
entire complex plane with poles at $\{x^{\mathbf{Z}_{\leq-1}},\overline
{x}^{\mathbf{Z}_{\leq-1}}\}$. These poles have order $1$.

\item In explicit terms, pick any sufficiently small choice of some $K>0$.
Then for all $z\in\mathbf{C}$ outside this set of poles, this continuation is
given by%
\[
R_{x}(z)=\frac{1}{2}\left(  \tilde{q}_{x}(z)+\tilde{q}_{\overline{x}%
}(z)\right)  \text{,}%
\]
where%
\begin{align*}
& \tilde{q}_{x}(z):=\frac{1}{2}\log(1-x)z+z\sum_{n\geq1}\frac{1}{n}\frac
{x^{n}}{n\log x+\operatorname*{Log}z}+i(\tilde{M}_{1,K}^{+}%
(-\operatorname*{Log}z)-\tilde{M}_{1,K}^{-}(-\operatorname*{Log}z))\\
& \qquad-\sum_{m,n\geq1}\frac{1}{m}\frac{e^{(m\log x+\operatorname*{Log}z+2\pi
in)(1+iK)}}{m\log x+\operatorname*{Log}z+2\pi in}-\sum_{m,n\geq1}\frac{1}%
{m}\frac{e^{(m\log x+\operatorname*{Log}z-2\pi in)(1-iK)}}{m\log
x+\operatorname*{Log}z-2\pi in}\text{,}%
\end{align*}
with $\tilde{M}_{1,K}^{\pm}$ as in Prop. \ref{prop_analytic_M}, and
$\operatorname*{Log}z$ is any choice of a logarithm defined in a neighbourhood
of $z $. In particular, the value of $\tilde{q}_{x}$ is independent of this
choice and the choice of $K$.
\end{enumerate}
\end{theorem}

\begin{proof}
We use a trick:\textit{\ (Trick)} Firstly, define $\tilde{h}_{w,x}%
(s):=\log(1-x^{n})e^{-ws}$. Then for all \textsl{integers} $n$, we have%
\begin{align*}
& \tilde{h}_{w,x}(n)+\tilde{h}_{w,\overline{x}}(n)=\log(1-e^{n\log x}%
)e^{-wn}+\log(1-e^{n\log\overline{x}})e^{-wn}\\
& \qquad=\left(  \log(1-e^{n\log x})+\log(\overline{1-e^{n\log x}})\right)
e^{-wn}=2\log\left\vert 1-e^{n\log x}\right\vert \cdot e^{-wn}\text{.}%
\end{align*}
The key point that we have used is that we only need this formula for
$n\in\mathbf{N}$, in particular $n$ is \textit{real}. Thus, $\overline{n}=n$,
which would be \textsl{false} for a general $s$. Moreover, we have used the
identities $\log z+\log\overline{z}=2\log\left\vert z\right\vert $ and
$\log\overline{z}=\overline{\log z}$, which follow from our choice of the
logarithm, Equation \ref{lbranch1}, and which need \textsl{not} hold for other
branches. For all $w\in\mathbf{C}$ with $\operatorname{Re}w>0$, we have
$\left\vert e^{-w}\right\vert <1$, so by\ Lemma \ref{Lemma_ConvergeOfRx} the
series%
\begin{align*}
R_{x}(e^{-w})  & =\sum_{n=1}^{\infty}\log\left\vert 1-x^{n}\right\vert \cdot
e^{-wn}=\frac{1}{2}\sum_{n=1}^{\infty}\log(1-x^{n})\cdot e^{-wn}\\
& \qquad+\frac{1}{2}\sum_{n=1}^{\infty}\log(1-\overline{x}^{n})\cdot
e^{-wn}=\frac{1}{2}\left(  Q_{x}(w)+Q_{\overline{x}}(w)\right)
\end{align*}
is uniformly convergent and defines a holomorphic function in the right
half-plane. Thanks to Prop. \ref{prop_AnalyticContinuationForQx}, $w\mapsto
R_{x}(e^{-w})$ has a meromorphic continuation to the entire complex plane,
call it $\widetilde{RW}_{x}(w)$.\newline\textit{(Conclusion)} This analytic
continuation must be periodic under $w\rightarrow w+2\pi i$. To see this, note
that it is true for $R_{x}(e^{-w})$ in the right half-plane, simply since it
is true for $e^{-w}$. Thus, $\widetilde{RW}_{x}$ must also be periodic. Let
$\operatorname*{Log}:U\rightarrow\mathbf{C}$ be some branch of the logarithm,
defined on some domain $U\subset\mathbf{C}$. Define a function%
\[
\tilde{R}_{x}:U\longrightarrow\mathbf{C}\text{,}\qquad\tilde{R}_{x}%
(z):=\widetilde{RW}_{x}(-\operatorname*{Log}z)\text{.}%
\]
This makes $\tilde{R}_{x}$ a meromorphic function on $U$. Since all branches
of the logarithm differ by multiples of $2\pi i$, and $\widetilde{RW}_{x}$ is
$2\pi i$-periodic, it follows that choosing different $U$ and different
branches, the definitions of $\tilde{R}_{x}$ on the various opens $U$ glue.
This defines a meromorphic function on the entire complex plane. Finally, we
observe that $\tilde{R}_{x}(e^{-w})=R_{x}(e^{-w})$ for $\operatorname{Re}w>0$,
so this is an analytic continuation. The poles of $Q_{x}$ become poles at
$\{x^{\mathbf{Z}_{\leq-1}}\}$, and analogously for $Q_{\overline{x}}$.\newline
For the explicit formula, unravel our construction. For $x$ and $\overline{x}
$ we may have picked different constants $K$, but in view of Assumption
$\mathfrak{C}^{\pm}$, taking the minimum of these choices, will be fine for both.
\end{proof}

\section{Equidistribution arguments}

Let us recall Weyl Equidistribution for $n$ variables.

\begin{theorem}
[Weyl Equidistribution]\label{Thm_WeylEquidist_multidim}Suppose $(t_{n}%
)_{n\geq1}$ is a uniformly distributed sequence in $[0,1]^{d}$. Then for every
Riemann-integrable $f:[0,1]^{d}\rightarrow\mathbf{R}$, one has equality
\begin{equation}
\underset{N\rightarrow\infty}{\lim}\frac{1}{N}\sum_{n=1}^{N}f(t_{n}%
)=\int_{[0,1]^{d}}f(\underline{s})\,\mathrm{d}\underline{s}\text{.}%
\label{lssa}%
\end{equation}

\end{theorem}

This type of equidistribution statement will be an important tool later.
However, we will need to apply it in situations where the function $f$ is not Riemann-integrable.

\begin{remark}
This is a considerable problem even for $d=1$: The left-hand side only depends
on countably many values of $f$, so any notion of integrability which is
preserved under changing $f$ at countably many points is inevitably too weak
to keep the conclusion of the theorem intact. There is a clarifying
\textsl{No-Go Theorem:} Given any Lebesgue-integrable function
$f:[0,1]\rightarrow\mathbf{R}$ which does \emph{not} admit a
Riemann-integrable representative, there must exist a uniformly distributed
sequence $(t_{n})$ such that Equation \ref{lssa} fails \cite{MR0225946}.
\end{remark}

There are more refined and flexible versions of equidistribution theorems
which allow relaxing the assumption of Riemann-integrability when working with
$d=1$ and the sequence is%
\[
t_{n}:=\{n\theta\}\text{,}%
\]
where we write%
\[
\{x\}:=x-\left\lfloor x\right\rfloor
\]
for the fractional part of a real number $x$.

\begin{remark}
\label{rmk_ergodic_approach}If $\theta$ is irrational, this sequence stems
from an ergodic discrete dynamical system on the unit circle, so one can get a
result similar to Equation \ref{lssa} by Birkhoff's Ergodic\ Theorem, however
it is only valid almost everywhere, and that is not good enough for our purposes.
\end{remark}

Inspired by work of Hardy and\ Littlewood, Oskolkov introduced his notion of
functions of Class H and identified conditions for Equation \ref{lssa} to
hold, \cite{MR1044053}, \cite{MR1275906}. We shall work with a stronger
version due to Baxa and Schoi\ss engeier. We will also have to do a little
extra work since we need the result to hold for the sequence $(\{n\theta
_{1}\},\{n\theta_{2}\})$, i.e. dimension $d=2$, as well. Either way, the
specific arithmetic properties of the number $\theta$ become relevant, so we
need to recall some material:

Suppose $\theta\in(0,1)$ is a real number. The (simple) continued fraction
presentation of $\theta$ is%
\begin{equation}
\theta=\frac{1}{a_{1}+\frac{1}{a_{2}+\frac{1}{\ddots}}}\text{,}\qquad
\text{(with }a_{n}\in\mathbf{Z}_{\geq1}\text{)}\label{luia5}%
\end{equation}
and is customarily abbreviated by $\theta=[a_{1},a_{2},\ldots]$. The $a_{n} $
are the \emph{partial quotients}. Moreover, one defines $p_{n},q_{n}$
recursively by%
\begin{align*}
& p_{0}:=0\text{,}\qquad q_{0}:=1\text{,}\qquad p_{-1}:=1\text{,}\qquad
q_{-1}:=0\text{,}\\
& p_{n+1}:=a_{n+1}p_{n}+p_{n-1}\text{,}\qquad q_{n+1}:=a_{n+1}q_{n}+q_{n-1}%
\end{align*}
for all $n\in\mathbf{Z}_{\geq0}$. The fractions $\frac{p_{n}}{q_{n}}$ are the
\emph{convergents} and they correspond to truncating the continued fraction in
line \ref{luia5} after evaluating the $n$-th interwoven fraction. They are
always in lowest terms, i.e. $(p_{n},q_{n})=1$.

The right-hand side in line \ref{luia5} thus has the tacit meaning to be the
limit of the sequence of convergents (one can show that this always
converges). If $\theta\in(0,1)$ is irrational, its simple continued fraction
always exists and the $a_{1},a_{2},\ldots$ are uniquely determined
\cite[Chapter B]{MR1451873}, and conversely all sequences $(a_{n})$ in
$\mathbf{Z}_{\geq1}$ define an irrational $\theta\in(0,1)$.

\begin{definition}
\label{def_BadlyApproximable}For every real number $\theta$, let $\left\Vert
\theta\right\Vert \in\lbrack0,\frac{1}{2}]$ denote the distance to the closest
integer. An irrational $\theta\in(0,1)$ is called \emph{badly approximable} if
one (then all) of the following equivalent conditions are met:

\begin{enumerate}
\item The infimum $\underset{n\geq1}{\inf}\{n\left\Vert n\theta\right\Vert
\}>0$ is strictly positive.

\item There exists a $C_{\theta}>0$ such that $\left\vert \theta-\frac{p}%
{q}\right\vert >\frac{C_{\theta}}{q^{2}}$ holds for all $p/q\in\mathbf{Q}$.

\item There exists a $K_{\theta}>0$ such that $a_{n}<K_{\theta}$ holds for all
partial quotients $a_{n}$ in the continued fraction $[a_{1},a_{2},\ldots]$.
\end{enumerate}
\end{definition}

The equivalence of these characterizations is shown in \cite[Theorem
23]{MR1451873} or \cite[Theorem 5F]{MR568710}. We also need:

\begin{lemma}
\label{lemma_QLinearIndependentMeansTorusUniformlyDist}Suppose $1,\theta
_{1},\ldots,\theta_{d}$ are $\mathbf{Q}$-linearly independent elements inside
the real line. Then the sequence of vectors%
\[
(\{n\theta_{1}\},\ldots,\{n\theta_{d}\})_{n\geq1}\in\lbrack0,1]^{d}%
\]
is uniformly distributed.
\end{lemma}

\begin{proof}
Standard. A detailed proof is given for example in \cite[Ch. I, Theorem 6.3,
Example 6.1]{MR0419394}.
\end{proof}

We fix $\theta\in(0,1)$ irrational. We shall need to work with the Main Lemma
of Baxa and Schoi\ss engeier, so let us recall its statement:

\begin{definition}
[\cite{MR1914805}]\label{Def_SeqSigmaPermutations}For the fixed irrational
$\theta\in(0,1)$, we use the following notation: Given any $N\in
\mathbf{Z}_{\geq1}$, let $\sigma_{N}\in\mathfrak{S}_{N}$ denote the
permutation such that%
\[
\{\sigma_{N}(1)\theta\}<\{\sigma_{N}(2)\theta\}<\{\sigma_{N}(3)\theta
\}<\cdots<\{\sigma_{N}(N)\theta\}\text{.}%
\]

\end{definition}

With this notation:

\begin{lemma}
[Baxa--Schoi\ss engeier Main Lemma]Fix $\theta\in(0,1)$ irrational and use the
notation as introduced above. Let $\beta=\frac{p}{q}\in\mathbf{Q}\cap(0,1]$ be
a rational number given in lowest terms (with $p,q>0$). Write $n_{N}$ to
denote the largest integer such that%
\[
1\leq n_{N}\leq N\qquad\text{and}\qquad\{\sigma_{N}(n_{N})\theta
\}<\beta\text{.}%
\]
Then there exists $m_{0}\in\mathbf{Z}_{\geq1}$ such that for all $m\geq m_{0}$
and all integers $b$ with $1\leq b\leq\max\{1,\frac{a_{m+1}}{2q}\}$ the
following holds: For every $f:[0,1]\rightarrow\lbrack0,+\infty)$ a
Lebesgue-integrable function, which is monotonously increasing in $[0,\beta) $
and is zero in $(\beta,1]$%
\[
\frac{1}{N}\sum_{n=1}^{N}f(\{n\theta\})\leq7q\int_{0}^{1}f(t)\,\mathrm{d}%
t+\frac{1}{\sigma_{N}(n_{N})}f(\{\sigma_{N}(n_{N})\theta\})
\]
holds for $N:=bq_{m}$.
\end{lemma}

This is \cite[Main Lemma]{MR1914805}. Next, we discuss the Equidistribution
theorem of \textit{loc. cit}. Following their paper, we isolate a well-behaved
class of real functions that go considerable beyond the Riemann-integrable ones:

\begin{definition}
[{\cite[Theorem 1]{MR1914805}}]Let $F\subseteq\lbrack0,1]\cap\mathbf{Q}$ be a
finite subset of the rational numbers. We say that a function
$f:[0,1]\rightarrow\mathbf{R}$ belongs to \emph{Class }$\operatorname*{BS}(F)$
if the following properties hold:

\begin{enumerate}
\item $f$ is Lebesgue-integrable,

\item $f$ is almost everywhere continuous,

\item $f$ locally bounded at every point in $[0,1]\setminus F$,

\item for every $\beta\in F$, there exists some $\varepsilon>0$ such that
$f\mid_{(\beta-\varepsilon,\beta)}$ is bounded or monotone,

\item for every $\beta\in F$, there exists some $\varepsilon>0$ such that
$f\mid_{(\beta,\beta+\varepsilon)}$ is bounded or monotone.
\end{enumerate}

We call $F$ the set of (possible) \emph{singularities}. If $f:[0,1]\rightarrow
\mathbf{C}$ is complex-valued, we say $f\in\operatorname*{BS}(F)$ if both real
and imaginary part belong to $\operatorname*{BS}(F)$.
\end{definition}

\begin{theorem}
[Baxa--Schoi\ss engeier Equidistribution]%
\label{Thm_BaxaSchoissengeierEquidist}Suppose $\theta\in(0,1)$ is irrational
and $f\in\operatorname*{BS}(F)$ for a suitably chosen $F$. Then we have
equality%
\[
\underset{N\rightarrow\infty}{\lim}\frac{1}{N}\sum_{n=1}^{N}f(\{n\theta
\})=\int_{0}^{1}f(t)\,\mathrm{d}t
\]
if and only if%
\begin{equation}
\underset{n\rightarrow\infty}{\lim}\frac{f(\{n\theta\})}{n}=0\text{.}%
\label{lssa2}%
\end{equation}
The condition of Equation \ref{lssa2} is automatically satisfied if $\theta$
is badly approximable.
\end{theorem}

Most of this is \cite[Theorem 1]{MR1914805}. The last claim is \cite[Theorem
2, (2)]{MR1914805}. An even stronger version was established in
\cite{MR2189068}, under broader, but very technical assumptions.

\begin{remark}
If one can ensure stronger conditions on $f$, a precursor of this result is
given in \cite{MR1044053}. If $f$ is a function of Oskolkov's Class H and
$\theta$ is badly approximable, its partial quotients are bounded, so
$\{n\theta\}$ is regularly distributed by \cite[Theorem 3]{MR1275906}. Then
Theorem 1 \textit{loc. cit.} also gives the same conclusion.
\end{remark}

\subsection{An ad-hoc multi-dimensional equidistribution theorem}

As mentioned before, we shall need a slightly stronger form of the
Baxa--Schoi\ss engeier result, imitating Weyl Equidistribution not just for
$d=1$, but also for $d=2$.

\begin{definition}
\label{def_class_BSU}Let $F\subseteq\lbrack0,1]\cap\mathbf{Q}$ be a finite
subset of the rational numbers. We say that a function $f:[0,1]^{d}%
\rightarrow\mathbf{R}$ belongs to \emph{Class }$\operatorname*{BSU}^{d}(F)$ if
the following holds: There exists a function $g\in\operatorname*{BS}(F)$ and a
function $h:[0,1]^{d}\rightarrow\mathbf{R}$ such that%
\[
f(y_{1},\ldots,y_{d})=g(y_{1})\cdot h(y_{1},\ldots,y_{d})\text{,}%
\]
where the function $h$ is Riemann-integrable. We call a choice of such a
function $g$ a \emph{singular weight}. If $f:[0,1]^{d}\rightarrow\mathbf{C}$
is complex-valued, we say $f\in\operatorname*{BSU}^{d}(F)$ if both real and
imaginary part belong to $\operatorname*{BSU}^{d}(F)$.
\end{definition}

We are ready to state our minimalistic extension of Baxa--Schoi\ss engeier
Equidistribution, just about strong enough for what we need:

\begin{theorem}
[Ad-hoc Unidirectional Equidistribution]\label{thm_bs_multi}Suppose
$1,\theta_{1},\ldots,\theta_{d}$ are $\mathbf{Q}$-linearly independent real
numbers. Suppose $F\subseteq\lbrack0,1]\cap\mathbf{Q}$ is finite. Suppose
$f:[0,1]^{d}\rightarrow\mathbf{R}$ is a function in Class $\operatorname*{BSU}%
^{d}(F)$ which admits a singular weight $g$ such that
\begin{equation}
\underset{n\rightarrow\infty}{\lim}\frac{g(\{n\theta_{1}\})}{n}=0\text{.}%
\label{luiav3}%
\end{equation}
Then%
\[
\underset{N\rightarrow\infty}{\lim}\frac{1}{N}\sum_{n=1}^{N}f(\{n\theta
_{1}\},\ldots,\{n\theta_{d}\})=\int_{[0,1]^{d}}f(\underline{s})\,\mathrm{d}%
\underline{s}\text{.}%
\]

\end{theorem}

We follow the proof of the one-dimensional case in \cite{MR1914805} very
closely, but need to perform some minor modifications. We have tailored this
formulation to be sufficient for our purposes. It would be desirable to have
more general theorems of this sort.

\begin{proof}
We write $f(y_{1},\ldots,y_{d})=g(y_{1})\cdot h(y_{1},\ldots,y_{d})$ with $g$
the singular weight such that Equation \ref{luiav3} holds. We may write
$h=h^{+}-h^{-}$ for $h^{\pm}$ non-negative functions. Since left- and
right-hand side of our claim are linear, it suffices to prove our claim under
the additional assumption $h\geq0$. As $h$ is Riemann-integrable by
assumption, thus bounded, we get%
\[
0\leq h(y_{1},\ldots,y_{d})\leq h_{\max}\text{.}%
\]
Without loss of generality, we may assume $\theta_{1}\in(0,1)$. By assumption,
$\theta_{1}$ is linearly independent from $1$ over the rationals, so
$\theta_{1}$ is irrational. Henceforth, we use the notation $p_{i},q_{i},$
$a_{i}$ etc. for the convergents, partial quotients, etc. for $\theta
:=\theta_{1}$.\newline\textit{(Step 1)} As in \cite{MR1914805}, we first deal
with the case $F=\varnothing$. By Lebesgue's integrability criterion, a
bounded function $f:[0,1]^{d}\rightarrow\mathbf{R}$ is Riemann-integrable if
and only if it continuous almost everywhere. Thus, as $F=\varnothing$, $f$ is
locally bounded on the compactum $[0,1]^{d}$ and thus bounded. By Lebesgue's
criterion, it follows that $f$ is Riemann-integrable and we can use
multi-dimensional Weyl Equidistribution, Theorem
\ref{Thm_WeylEquidist_multidim}, since by Lemma
\ref{lemma_QLinearIndependentMeansTorusUniformlyDist} $(\{n\theta_{1}%
\},\ldots,\{n\theta_{d}\})$ is uniformly distributed in $[0,1]^{d}$%
.\newline\textit{(Step 2)} Next, as in \cite{MR1914805}, consider the case
$F=\{\beta\}$. Assume that $\beta=\frac{p}{q}\in(0,1]$ is a rational number in
lowest terms, $p,q>0$. Suppose $\lim_{t\rightarrow\beta,t<\beta}g(t)=+\infty$
and $g\mid_{(\beta,1)}=0$. Define an \emph{admissible} $\varepsilon>0$ to be
any element in%
\[
\{\varepsilon\mid\varepsilon>0\text{, the function }g\mid_{\lbrack
\beta-\varepsilon,\beta)}\text{ is monotone and non-negative}\}\text{.}%
\]
By the assumptions of $\operatorname*{BS}(F)$, admissible $\varepsilon$ exist.
For any sufficiently large integer $N$, we may choose some $m\in
\mathbf{Z}_{\geq1}$ with $q_{m}\leq N<q_{m+1}$, and then pick $b:=\left\lfloor
N/q_{m}\right\rfloor \geq1$ (cf. renewal time). For every admissible
$\varepsilon$, define%
\[
f_{\varepsilon}(y_{1},\ldots,y_{n}):=f(y_{1},\ldots,y_{n})\cdot c_{[\beta
-\varepsilon,1]\times\lbrack0,1]^{d-1}}\qquad\qquad g_{\varepsilon}%
(y_{1}):=g(y_{1})\cdot c_{[\beta-\varepsilon,1]}\text{,}%
\]
where $c_{I}$ denotes the characteristic function of a set $I$. Then%
\[
f_{\varepsilon}(y_{1},\ldots,y_{n})=g_{\varepsilon}(y_{1})h(y_{1},\ldots
,y_{n})\text{.}%
\]
Moreover, since $g_{\varepsilon}$ and $h$ are non-negative (by the choice of
$\varepsilon$), $f_{\varepsilon}$ is non-negative. Thus,%
\begin{align*}
\frac{1}{N}\sum_{n=1}^{N}f_{\varepsilon}(\{n\theta_{1}\},\ldots,\{n\theta
_{d}\})  & =\frac{1}{N}\sum_{n=1}^{N}g_{\varepsilon}(\{n\theta_{1}\})\cdot
h(\{n\theta_{1}\},\ldots,\{n\theta_{d}\})\\
& \leq\frac{h_{\max}}{N}\sum_{n=1}^{N}g_{\varepsilon}(\{n\theta_{1}\})\text{,}%
\end{align*}
Now Baxa and\ Schoi\ss engeier perform a case distinction, using their Main
Lemma. We copy this: If $\frac{a_{m+1}}{2q}-1<b$, the Main Lemma yields:
\begin{align*}
& \frac{1}{N}\sum_{n=1}^{N}g_{\varepsilon}(\{n\theta_{1}\})\leq\quad
(\quad\ldots\quad)\\
& \qquad\leq28q^{2}\int_{0}^{1}g_{\varepsilon}(s)\,\mathrm{d}s+\frac
{4q}{\sigma_{q_{m+1}}(n_{q_{m+1}})}g_{\varepsilon}(\{\sigma_{q_{m+1}%
}(n_{q_{m+1}})\theta_{1}\})\text{,}%
\end{align*}
with $\sigma_{(-)},n_{(-)}$ as in the sense of the Main Lemma. Or, in the
other case $b\leq\frac{a_{m+1}}{2q}-1$, it yields%
\begin{align*}
& \frac{1}{N}\sum_{n=1}^{N}g_{\varepsilon}(\{n\theta_{1}\})\leq\quad\left(
\quad\ldots\quad\right) \\
& \qquad\leq14q\int_{0}^{1}g_{\varepsilon}(s)\,\mathrm{d}s+\frac{2}%
{\sigma_{(b+1)q_{m}}(n_{(b+1)q_{m}})}g_{\varepsilon}(\{\sigma_{(b+1)q_{m}%
}(n_{(b+1)q_{m}})\theta_{1}\})\text{.}%
\end{align*}
We refer to their paper for any details. As we can do this for a sequence of
choices $N$ with $N\rightarrow+\infty$, and we have $\lim_{N\rightarrow
+\infty}\sigma_{N}(n_{N})=+\infty$ (cf. Definition
\ref{Def_SeqSigmaPermutations}), we may use Equation \ref{luiav3} and as in
\cite{MR1914805} we obtain: For every admissible $\varepsilon$, we have the
upper bound%
\begin{align*}
\underset{N\rightarrow+\infty}{\lim\sup}\,\frac{1}{N}\sum_{n=1}^{N}%
f_{\varepsilon}(\{n\theta_{1}\},\ldots,\{n\theta_{d}\})  & \leq\underset
{N\rightarrow+\infty}{\lim\sup}\,\frac{h_{\max}}{N}\sum_{n=1}^{N}%
g_{\varepsilon}(\{n\theta_{1}\})\\
& \leq28\cdot h_{\max}\cdot q^{2}\int_{0}^{1}g_{\varepsilon}(s)\,\mathrm{d}%
s\text{.}%
\end{align*}
Now, since we had assumed that $F=\{\beta\}$, the function $g-g_{\varepsilon}$
is Riemann-integrable, thus%
\[
f-f_{\varepsilon}=\left(  g-g_{\varepsilon}\right)  \cdot h
\]
is itself Riemann-integrable. Multi-dimensional Weyl Equidistribution applies
so that%
\[
\lim_{N\rightarrow+\infty}\frac{1}{N}\sum_{n=1}^{N}(f-f_{\varepsilon
})(\{n\theta_{1}\},\ldots,\{n\theta_{d}\})=\int_{[0,1]^{d}}(f-f_{\varepsilon
})(\underline{s})\mathrm{d}\underline{s}\leq\int_{\lbrack0,1]^{d}}%
f(\underline{s})\mathrm{d}\underline{s}%
\]
since $f_{\varepsilon}$ is a non-negative function, and $f$ is
Lebesgue-integrable. Combining both upper bounds, we obtain%
\[
\underset{N\rightarrow+\infty}{\lim\sup}\,\frac{1}{N}\sum_{n=1}^{N}%
f(\{n\theta_{1}\},\ldots,\{n\theta_{d}\})\leq\int_{\lbrack0,1]^{d}%
}f(\underline{s})\mathrm{d}\underline{s}+28\cdot h_{\max}\cdot q^{2}\cdot
\int_{0}^{1}g_{\varepsilon}(s)\,\mathrm{d}s
\]
for all admissible $\varepsilon$. Thus,%
\[
\underset{N\rightarrow+\infty}{\lim\sup}\,\frac{1}{N}\sum_{n=1}^{N}%
f(\{n\theta_{1}\},\ldots,\{n\theta_{d}\})\leq\int_{\lbrack0,1]^{d}%
}f(\underline{s})\mathrm{d}\underline{s}\text{.}%
\]
Conversely, since $f_{\varepsilon}$ is non-negative,%
\[
\underset{N\rightarrow+\infty}{\lim\inf}\,\frac{1}{N}\sum_{n=1}^{N}%
f(\{n\theta_{1}\},\ldots,\{n\theta_{d}\})\geq\underset{N\rightarrow+\infty
}{\lim\inf}\,\frac{1}{N}\sum_{n=1}^{N}(f-f_{\varepsilon})(\{n\theta
_{1}\},\ldots,\{n\theta_{d}\})=\int_{[0,1]^{d}}f(\underline{s})\mathrm{d}%
\underline{s}%
\]
by Weyl Equidistribution and since $f-f_{\varepsilon}$ is Riemann-integrable.
As both limes superior and inferior exist and coincide, we obtain that the
limit exists and is of said value.\newline\textit{(Step 3)} Now one can do an
induction over the cardinality of $F$. This argument can be carried out
precisely as in \cite{MR1914805} and we leave it to the reader.
\end{proof}

\section{\label{sect_Orthogonality}The orthogonality theorem}

This section is devoted to the proof of the following statement.

\begin{theorem}
[Orthogonality]\label{thm_ortho}Suppose $\theta$ is a real number such that either

\begin{itemize}
\item $e^{2\pi i\theta}$ is an algebraic number, or

\item $\theta$ is badly approximable, i.e. it has a bounded sequence of
partial quotients.
\end{itemize}

Then the following holds:

\begin{enumerate}
\item If $\dim_{\mathbf{Q}}\left\langle 1,\theta\right\rangle =2$: For all
$m\in\mathbf{Z}$, we have%
\[
\underset{N\rightarrow\infty}{\lim}\frac{1}{N}\sum_{n=1}^{N}\log\left\vert
1-e^{2\pi in\theta}\right\vert \cdot e^{2\pi imn\theta}=-\frac{1}{2\left\vert
m\right\vert }\delta_{m\neq0}\text{.}%
\]
If $m\in\mathbf{Q}\setminus\mathbf{Z}$, we get a value%
\[
C_{m}\in\mathbf{Q}(\mu_{\infty},\pi,\{L(1,\chi)\}_{\chi})\text{,}%
\]
where $\chi$ ranges over a set of non-principal Dirichlet characters modulo
$2v$ for $v\geq1$ the denominator of $m$ in lowest terms. The values $C_{m}$
only depend on $m$, and are independent of $\theta$.

\item If $\alpha$ is a real number and $\dim_{\mathbf{Q}}\left\langle
1,\theta,\alpha\right\rangle =3$, then for all $m\in\mathbf{Z}$,%
\[
\underset{N\rightarrow\infty}{\lim}\frac{1}{N}\sum_{n=1}^{N}\log\left\vert
1-e^{2\pi in\theta}\right\vert \cdot e^{2\pi imn\alpha}=0\text{.}%
\]

\end{enumerate}
\end{theorem}

\begin{example}
Let us look at the following polynomial%
\begin{align*}
& f:=x^{8}-120x^{7}+4332x^{6}-86664x^{5}\\
& \qquad\qquad+1311590x^{4}-10994952x^{3}+75494124x^{2}-19704x+1\text{.}%
\end{align*}
It was constructed by A. Dubickas \cite{MR3193953}. It has a special property:
Four complex roots of this polynomial have the same absolute value (this is
not at all obvious) and are (obviously) algebraic units. Thus, the quotients
$\frac{v_{1}}{v_{2}}$ of any two such are concrete complex algebraic numbers
lying on the unit circle. Moreover, $\frac{v_{1}}{v_{2}}$ will not be a root
of unity, so Claim (1) of the theorem applies. We get\footnote{as an aside:
note that by Gelfond--Schneider, both $\theta_{1},\theta_{2}$ are necessarily
transcendental.}%
\begin{align*}
\left.  u_{1}:=e^{2\pi i\theta_{1}}\right.  \qquad & \text{with}\qquad\left.
\theta_{1}\approx0.400842\right. \\
\left.  u_{2}:=e^{2\pi i\theta_{2}}\right.  \qquad & \text{with}\qquad\left.
\theta_{2}\approx0.410383\right.  \text{.}%
\end{align*}
The sequence $\sum_{n=1}^{N}\log\left\vert 1-e^{2\pi in\theta_{i}}\right\vert
$ (for $i=1,2$) oscillates rather wildly, so for example for $m=0$, the
theorem claims that the average values of the sequence still balance out
around zero. Even looking at concrete values, this is not obvious:%
\[%
{\includegraphics[
height=0.8146in,
width=1.6481in
]%
{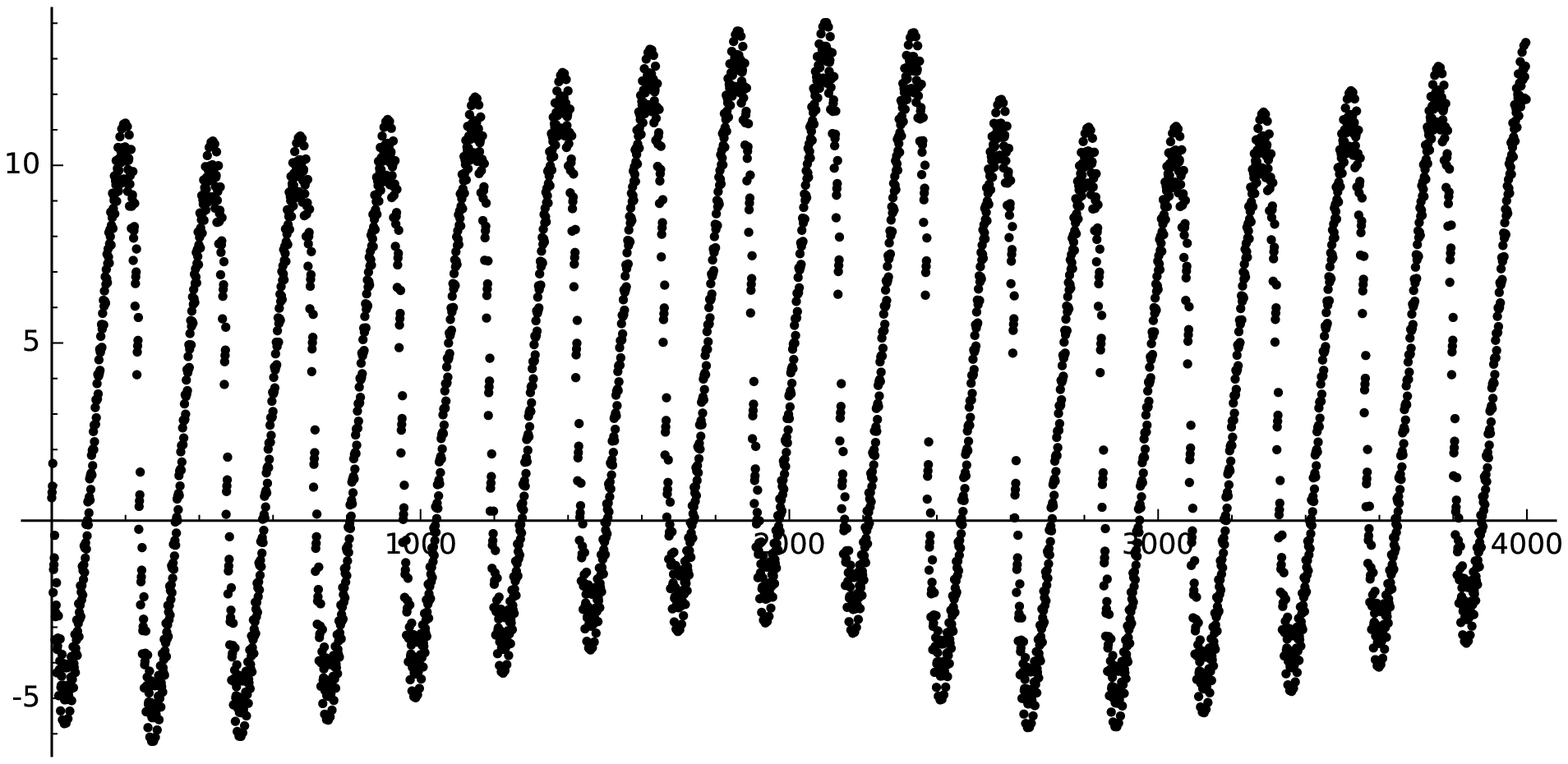}%
}
\qquad\qquad%
{\includegraphics[
height=0.7923in,
width=1.6003in
]%
{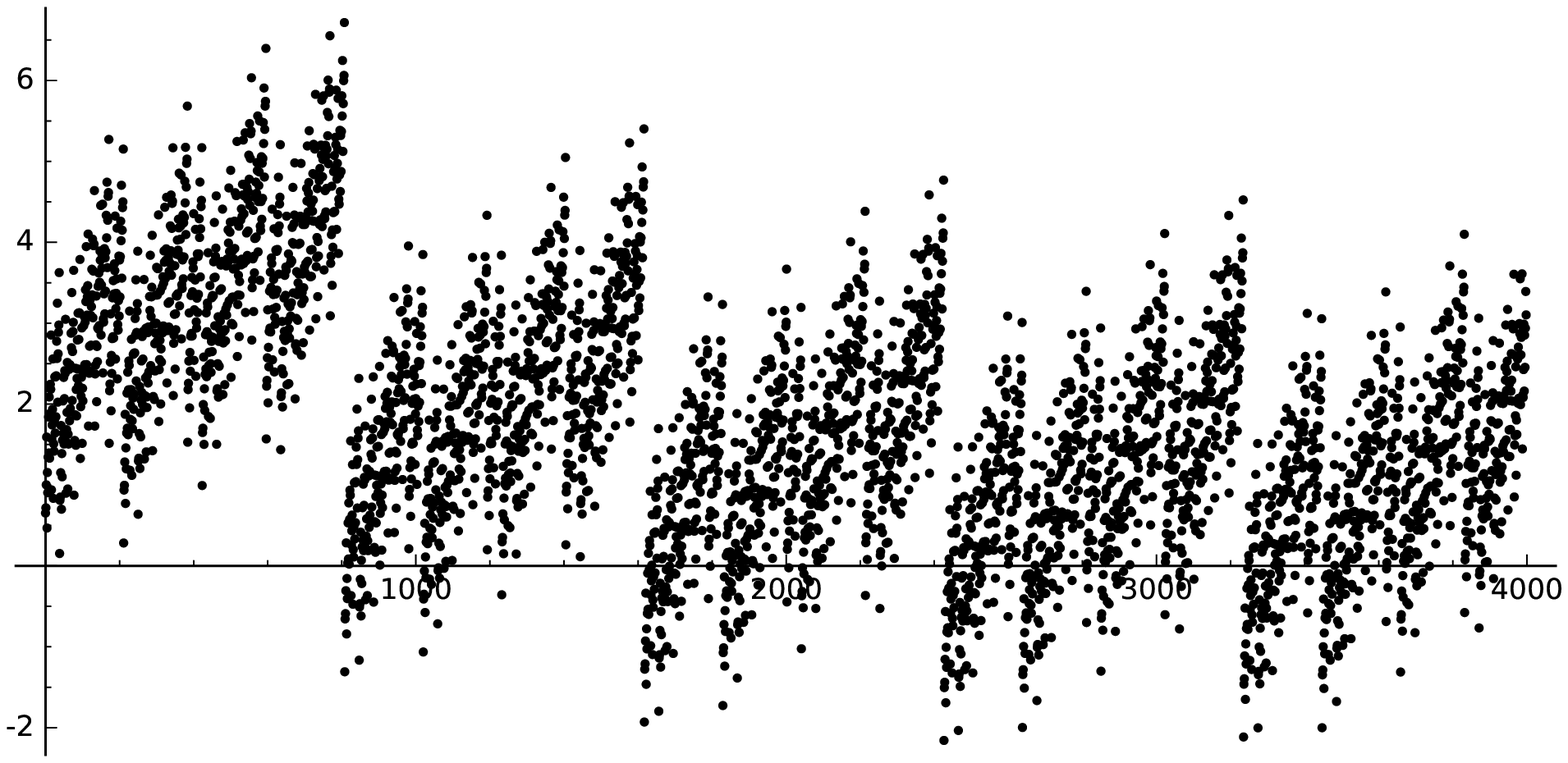}%
}
\]

\end{example}

The main idea will be to use equidistribution results to translate the
statement of the Orthogonality Theorem into statements about integrals. So
first of all, let us compute the relevant integrals.

\subsection{Integral values}

Define%
\[
P_{m}:=\int_{0}^{\pi}\log(\sin t)e^{2imt}\,\mathrm{d}t\qquad\text{for}\qquad
m\in\mathbf{Z}\text{.}%
\]

\begin{lemma}
\label{lemma_aux_integral_Pm}We have $P_{0}=-\pi\log2$ and $P_{m}=-\frac{1}%
{2}\frac{\pi}{\left\vert m\right\vert }$ for all $m\in\mathbf{Z}%
\setminus\{0\}$.
\end{lemma}

\begin{proof}
For $m=0$, this is a not completely trivial, but still standard exercise in
contour integration, presented for example in \cite[Ch. 4, \S 5.3.5]%
{MR510197}. \textit{(Step 1)} For non-negative $m\geq0$, the same technique
works with the appropriate changes made. We give the details for the sake of
completeness: Fix some $m\in\mathbf{Z}_{\geq0}$. Consider the function%
\[
g(z):=\log(1-e^{2iz})e^{2imz}\text{.}%
\]
As in \cite[Ch. 4, \S 5.3.5]{MR510197}, it is easy to see that $1-e^{2iz}%
\in\mathbf{R}_{\leq0}$ holds if and only if $z\in\pi\mathbf{Z}+i\mathbf{R}%
_{\leq0}$. Pick some $K>0$ and consider the box $\mathsf{Box}:=[0,\pi]\times
i[0,K]$. Starting from the contour $\partial\mathsf{Box}$, in order to obtain
that $g$ is holomorphic everywhere inside and on a neighbourhood of the
contour, we need to switch to a modified contour $\mathsf{C}$: Introduce small
quarter-circle indentations at $z=0$ and $z=\pi$. Then by Cauchy's Integral
Theorem,%
\[
0=\int_{\mathsf{C}}g(z)\,\mathrm{d}z=\int_{\mathsf{C}_{top}}g(z)\,\mathrm{d}%
z+\int_{0}^{\pi}\log(1-e^{2it})e^{2imt}\,\mathrm{d}t\text{,}%
\]
where $\mathsf{C}_{top}$ is the top edge. The contributions from the left and
right edge cancel out by the periodicity $g(z+\pi)=g(z)$. Under $K\rightarrow
+\infty$, the top edge integral converges to zero (this needs $m\geq0$). The
treatment of the quarter-circle indentations requires limit considerations, we
refer to loc. cit. for the details. Finally, Euler's formula immediately gives
$1-e^{2it}=-2ie^{iz}\sin z$ and thus%
\begin{align*}
0  & =\int_{0}^{\pi}\log(1-e^{2it})e^{2imt}\,\mathrm{d}t=\int_{0}^{\pi}%
\log(-2ie^{it}\sin t)e^{2imt}\,\mathrm{d}t\\
& \qquad=\log(-2i)\int_{0}^{\pi}e^{2imt}\,\mathrm{d}t+\int_{0}^{\pi}%
\log(e^{it})e^{2imt}\,\mathrm{d}t+\int_{0}^{\pi}\log(\sin t)e^{2imt}%
\,\mathrm{d}t\text{.}%
\end{align*}
The second line follows from the first by the functional equation for the
complex logarithm (it requires a little side thought to be sure that no branch
switch discrepancy of $2\pi i$ gets introduced this way). We have
$\log(-2i)=\log2-\frac{\pi i}{2}$ and $\int_{0}^{\pi}e^{2imt}\,\mathrm{d}%
t=\delta_{m=0}\pi$. Moreover, $\log(e^{it})=it$ as $t\in(0,\pi)$. Hence,%
\[
\int_{0}^{\pi}\log(e^{it})e^{2imt}\,\mathrm{d}t=i\int_{0}^{\pi}te^{2imt}%
\,\mathrm{d}t=\left\{
\begin{array}
[c]{ll}%
\frac{1}{2}\frac{\pi}{m} & \text{for }m\neq0\\
\frac{1}{2}i\pi^{2} & \text{for }m=0
\end{array}
\right.
\]
by straight-forward partial integration. Combining these computations
immediately gives our claim for all $m\geq0$. \textit{(Step 2)} Next, we deal
with $m\in\mathbf{Z}_{\geq-1}$. The contour integration in \textit{Step 1}
does not work anymore (Reason: As one moves the top contour off to infinity,
this contribution does no longer converge to zero. To see this, note that
$\left\vert g(iy)\right\vert =\left\vert \log(1-e^{-2y})e^{-2my}\right\vert $,
and for negative $m$ the term $\left\vert e^{-2my}\right\vert $ grows
exponentially as $y\rightarrow+\infty$). We resolve this as follows: We have%
\[
P_{-m}=\int_{0}^{\pi}\log(\sin t)\overline{e^{2imt}}\,\mathrm{d}%
t=\overline{P_{m}}%
\]
for \textit{all} $m$. Since \textit{Step 1} shows that for $m\geq1$ the value
of $P_{m}$ is real, it is not affected by complex conjugation.
\end{proof}

Next, define auxiliary values%
\begin{equation}
W_{m}:=\int_{0}^{1}\log\left\vert 1-e^{2\pi i\theta}\right\vert e^{2\pi
im\theta}\,\mathrm{d}\theta\qquad\text{for}\qquad m\in\mathbf{Z}%
\text{.}\label{Def_Wm}%
\end{equation}
Recall the standard sine squaring formula, $\sin(t)^{2}=\frac{1}{2}\left(
1-\cos(2t)\right)  $. It implies $2\sin\left(  \frac{t}{2}\right)
^{2}=\left(  1-\cos t\right)  $ and thus%
\[
\left\vert 1-e^{it}\right\vert ^{2}=(1-e^{it})\overline{(1-e^{it})}=2-2\cos
t=4\sin\left(  \frac{t}{2}\right)  ^{2}\text{.}%
\]
Then,%
\begin{equation}
\log\left\vert 1-e^{it}\right\vert =\frac{1}{2}\log\left(  4\sin\left(
\frac{t}{2}\right)  ^{2}\right)  =\log\left(  2\sin\left(  \frac{t}{2}\right)
\right)  \text{.}\label{eq_logabs_to_sine}%
\end{equation}

\begin{lemma}
\label{Lemma_WmCompute}We have $W_{m}=-\frac{1}{2\left\vert m\right\vert
}\delta_{m\neq0}$ for all $m\in\mathbf{Z}$.
\end{lemma}

\begin{proof}
By substitution, we switch from the variable $\theta$ to $\theta/2\pi$. Then
use Formula \ref{eq_logabs_to_sine} in order to obtain
\begin{align*}
W_{m}  & =\frac{1}{2\pi}\int_{0}^{2\pi}\log\left\vert 1-e^{i\theta}\right\vert
e^{im\theta}\,\mathrm{d}\theta=\frac{1}{2\pi}\int_{0}^{2\pi}\log\left(
2\sin\left(  \frac{\theta}{2}\right)  \right)  e^{im\theta}\,\mathrm{d}%
\theta\\
& =\frac{1}{\pi}\int_{0}^{\pi}\log\left(  2\sin\theta\right)  e^{2im\theta
}\,\mathrm{d}\theta\\
& =\frac{1}{\pi}\int_{0}^{\pi}\log(\sin\theta)\cdot e^{2im\theta}%
\,\mathrm{d}\theta+\frac{1}{\pi}\log2\int_{0}^{\pi}e^{2i\theta m}%
\,\mathrm{d}\theta\\
& =\frac{1}{\pi}P_{m}+\frac{1}{\pi}\log2\cdot\pi\delta_{m=0}=-\frac
{1}{2\left\vert m\right\vert }\delta_{m\neq0}\text{.}%
\end{align*}
Here we have used the functional equation of the real logarithm, which yields
a term of the shape $P_{m}$ and then we may evaluate the entire expression
using Lemma \ref{lemma_aux_integral_Pm}.
\end{proof}

\subsection{Fractional values}

So far, we have determined $W_{m}$ for all integer values. For non-integral
values, the structure is more complicated. We will analyze this case now. Let
$m\in\mathbf{R}\setminus\mathbf{Z}_{\leq-1}$ be given. Recall that $w^{m}%
=\exp(m\cdot\log w)$ as a function in $w$ is holomorphic on $\mathbf{C}%
\setminus(-\infty,0]$ and $\log(1-w)$ is holomorphic in $\mathbf{C}%
\setminus\lbrack1,\infty)$. The intersection $X:=\mathbf{C}\setminus
((-\infty,0]\cup\lbrack1,\infty))$ is simply connected, so%
\[
S_{m}(z):=-\int_{0}^{z}\log(1-w)w^{m}\frac{\mathrm{d}w}{w}%
\]
determines a well-defined holomorphic function on $X$, independent of the
choice of a path of integration from $0$ to $z$. For $m=0$, we have
$S_{0}(z)=\operatorname*{Li}_{2}(z)$, the classical dilogarithm. For
$\left\vert z\right\vert <1$, termwise integration of the logarithm series
yields the uniformly convergent series%
\[
S_{m}(z)=\sum_{r=1}^{\infty}\frac{1}{r}\frac{z^{r+m}}{r+m}\qquad
\text{for}\qquad z\in X\text{, }\left\vert z\right\vert <1\text{.}%
\]
Note that this will usually \textsl{not} be a power series since $m$ need not
be a natural number. We may use this series to attach a value to the two
points $\{-1,1\}\notin X$, namely%
\begin{equation}
S_{m}(1):=\sum_{r=1}^{\infty}\frac{1}{r(r+m)}\qquad\text{and}\qquad
S_{m}(-1):=\sum_{r=1}^{\infty}\frac{1}{r}\frac{e^{i\pi(r+m)}}{r+m}%
\text{.}\label{l_def_Sq_AtOneAndMinusOne}%
\end{equation}
Note that these values really hinge on our choice of preferred branches, e.g.%
\begin{equation}
S_{m}(-1)=\lim\limits_{\substack{z\longrightarrow-1, \\z\in
X,\operatorname{Im}z>0 }}S_{m}(z)\text{.}\label{l_rmk_SqMinusOne}%
\end{equation}

\begin{proposition}
\label{prop_log_val_for_rationalarg_helper}Suppose $m\in\mathbf{R}%
\setminus\mathbf{Z}_{\leq0}$.

\begin{enumerate}
\item Then%
\[
i\int_{0}^{\pi}\log\left(  2\sin\left(  \frac{\theta}{2}\right)  \right)
\cdot e^{im\theta}\mathrm{d}\theta=S_{m}(1)-S_{m}(-1)-\frac{1}{2m^{2}}\left(
1-e^{i\pi m}+m\pi i\right)  \text{,}%
\]
where $S_{m}(1)$ and $S_{m}(-1)$ are defined as in Equation
\ref{l_def_Sq_AtOneAndMinusOne}.

\item Moreover,%
\[
i\int_{\pi}^{2\pi}\log\left(  2\sin\left(  \frac{\theta}{2}\right)  \right)
\cdot e^{im\theta}\mathrm{d}\theta=-e^{2\pi im}\cdot\overline{i\int_{0}^{\pi
}\log\left(  2\sin\left(  \frac{\theta}{2}\right)  \right)  \cdot e^{im\theta
}\mathrm{d}\theta}\text{,}%
\]
reducing this integral to (1).
\end{enumerate}
\end{proposition}

\begin{proof}
\textit{(1)} Suppose $0<\delta<\xi<\pi$. We integrate the holomorphic function%
\[
f_{m}(w):=-\log(1-w)\cdot w^{m-1}\qquad\text{(for }w\in X\text{)}%
\]
over a circle segment: we go straight from $0$ to $e^{i\delta}$, follow the
arc from $e^{i\delta}$ to $e^{i\xi}$, and then go back straight from $e^{i\xi
}$ to $0$. This yields%
\begin{equation}%
\begin{tabular}
[c]{ll}%
\raisebox{-0.4341in}{\includegraphics[
height=0.9176in,
width=1.3906in
]%
{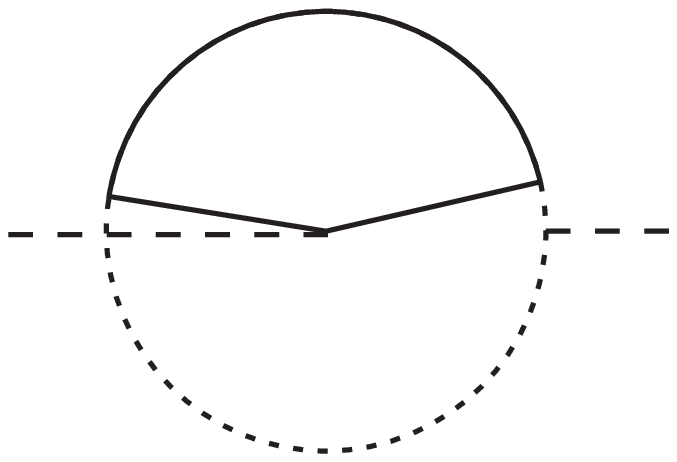}%
}
& $%
{\displaystyle\int_{0}^{e^{i\delta}}}
f_{m}(z)\mathrm{d}z+%
{\displaystyle\int_{e^{i\delta}}^{e^{i\xi}}}
f_{m}(z)\mathrm{d}z-%
{\displaystyle\int_{0}^{e^{i\xi}}}
f_{m}(z)\mathrm{d}z=0$%
\end{tabular}
\label{cDiag1}%
\end{equation}
(some indentation and care is required at $z=0$ since this does not lie in $X
$. We leave this to the reader). Thus, parametrizing the arc through
$\gamma(\theta):=e^{i\theta}$, we get%
\[
S_{m}(e^{i\delta})-i\int_{\delta}^{\xi}\log(1-e^{i\theta})\cdot e^{im\theta
}\mathrm{d}\theta-S_{m}(e^{i\xi})=0\text{.}%
\]
In the range $0<\theta<\pi$, we have%
\[
\log(1-e^{i\theta})=\log\left(  2\sin\left(  \frac{\theta}{2}\right)  \right)
+i\left(  \frac{\theta}{2}-\frac{\pi}{2}\right)  \text{.}%
\]
Thus,%
\[
S_{m}(e^{i\delta})-S_{m}(e^{i\xi})=i\int_{\delta}^{\xi}\log\left(
2\sin\left(  \frac{\theta}{2}\right)  \right)  \cdot e^{im\theta}%
\mathrm{d}\theta-\int_{\delta}^{\xi}\left(  \frac{\theta}{2}-\frac{\pi}%
{2}\right)  \cdot e^{im\theta}\mathrm{d}\theta\text{.}%
\]
The limit $\delta\rightarrow+0$ is harmless, and in fact our Definition of
$S_{m}(1)$ in line \ref{l_def_Sq_AtOneAndMinusOne} is made such that
$S_{m}(e^{i\delta})$ converges to $S_{m}(1)$. Moreover, since $m\neq0$, we
have%
\[
\int_{0}^{\pi}\left(  \frac{\theta}{2}-\frac{\pi}{2}\right)  \cdot
e^{im\theta}\mathrm{d}\theta=-\frac{1}{2m^{2}}\left(  1-e^{i\pi m}+m\pi
i\right)  \text{.}%
\]
by a straightforward computation. Analogously, consider the limit
$\xi\rightarrow\pi$. Again, $S_{m}(-1)$ is defined exactly in such a way to
agree with this limit, cf. Equation \ref{l_rmk_SqMinusOne}. Our first claim
follows. \textit{(2)}\ If we wanted to generalize the treatment of
\textit{Case 1}, we would have to handle the branch switch along the negative
real half-axis (Figure \ref{cDiag1}). We avoid this by exploiting symmetry:
Observe that%
\begin{align*}
& \int_{\pi}^{2\pi}\log\left(  2\sin\left(  \frac{\theta}{2}\right)  \right)
\cdot e^{im\theta}\mathrm{d}\theta=e^{i\pi m}\int_{0}^{\pi}\log\left(
2\sin\left(  \frac{\theta+\pi}{2}\right)  \right)  \cdot e^{im\theta
}\mathrm{d}\theta\\
& =e^{i\pi m}\int_{-\pi}^{0}\log\left(  2\sin\left(  \frac{-\theta+\pi}%
{2}\right)  \right)  \cdot e^{-im\theta}\mathrm{d}\theta=e^{i\pi m}\int_{-\pi
}^{0}\log\left(  2\sin\left(  \frac{\theta+\pi}{2}\right)  \right)  \cdot
e^{-im\theta}\mathrm{d}\theta\\
& =e^{2\pi im}\int_{0}^{\pi}\log\left(  2\sin\left(  \frac{\theta}{2}\right)
\right)  \cdot e^{-im\theta}\mathrm{d}\theta=e^{2\pi im}\overline{\int
_{0}^{\pi}\log\left(  2\sin\left(  \frac{\theta}{2}\right)  \right)  \cdot
e^{im\theta}\mathrm{d}\theta}\text{.}%
\end{align*}
Concretely: First, we shift integration to $[0,\pi]$, then we substitute
$-\theta$ for the variable $\theta$, then we use that $\sin\left(
\frac{\theta+\pi}{2}\right)  =\cos\left(  \frac{\theta}{2}\right)  $, so the
term inside the logarithm is invariant under changing the sign of $\theta$,
and then we shift back to $[0,\pi]$. Finally, we use that the logarithm term
is real-valued.
\end{proof}

\begin{lemma}
\label{Lemma_LerchTranscendent}Suppose $m\in\mathbf{Q}\setminus\mathbf{Z}%
_{\leq0}$ with $m=\frac{u}{v}$ with $u,v>0$ (not necessarily in lowest terms).
The value of $S_{m}(-1)$ is a finite $\mathbf{Q}(\mu_{\infty})$-linear
combination of values $L(1,\chi)$ for non-principal Dirichlet characters
$\chi$ modulo $2v$.
\end{lemma}

\begin{proof}
Define $z:=e^{i\pi\frac{j}{n}}$ (we shall only need the case $n=j=1$, but
dealing with the general case makes the computation clearer). Let
$C\in\{0,m\}$ be arbitrary. For all $s\in\mathbf{C}$ with $\operatorname{Re}%
s>1$, we compute%
\begin{align*}
\sum_{r=1}^{\infty}\frac{z^{r}}{(r+C)^{s}}  & =\sum_{r=0}^{\infty}\sum
_{l=1}^{2n}\frac{z^{2rn+l}}{(2rn+l+C)^{s}}=\sum_{l=1}^{2n}(e^{i\pi\frac{j}{n}%
})^{l}\sum_{r=0}^{\infty}\frac{1}{(2rn+l+C)^{s}}\\
& =\frac{1}{2n}\sum_{l=1}^{2n}(e^{i\pi\frac{j}{n}})^{l}\sum_{r=0}^{\infty
}\frac{1}{\left(  r+\frac{l+C}{2n}\right)  ^{s}}=\frac{1}{2n}\sum_{l=1}%
^{2n}(e^{i\pi\frac{j}{n}})^{l}\boldsymbol{\zeta}\left(  s,\frac{l+C}%
{2n}\right)  \text{,}%
\end{align*}
where $\boldsymbol{\zeta}(s,A):=\sum_{r=0}^{\infty}(r+A)^{-s}$ denotes the
Hurwitz zeta function. Note that we have used that $m\notin\mathbf{Z}_{\leq
-1}$. It is well-known that the Hurwitz zeta function at \textit{rational}
parameters, can be expressed through Dirichlet $L$-values. Concretely,%
\[
\boldsymbol{\zeta}\left(  s,\frac{a}{b}\right)  =\frac{b^{s}}{\varphi(b)}%
\sum_{\chi}\overline{\chi(a)}\cdot L(s,\chi)\text{,}%
\]
where $\chi$ runs through all Dirichlet characters modulo $b$, and $\varphi$
is Euler's totient function. Thus, we may expand%
\[
\sum_{r=1}^{\infty}\frac{z^{r}}{(r+C)^{s}}=\sum_{i\in I}x_{i}\cdot
L(s,\chi^{(i)})\text{,}%
\]
for $I$ some finite index set, $x_{i}\in\overline{\mathbf{Q}}$, $\chi^{(i)} $
Dirichlet characters modulo $2n$ (for $C=0$) resp. $2nv$ (for $C=\frac{u}{v}%
$). Now, restrict to the case $n=j=1$. Since $\left\vert z\right\vert =1$, but
$z\neq1$, the limit of the left-hand side for $s\rightarrow1$ exists. For all
non-principal characters $\tilde{\chi}$, $L(s,\tilde{\chi})$ exists for $s=1$
on the right-hand side. Thus, the principal character $\chi_{0}$ does not
occur among those $i\in I$ with $x_{i}\neq0$ (Reason: Suppose it does. Since
all other summands have a finite limit for $s\rightarrow1$, this would force
$L(s,\chi_{0})$ to have a finite limit for $s\rightarrow1$ as well, but there
is a pole instead). Thus, we can actually carry out the limit $s\rightarrow1$
and obtain%
\begin{equation}
\sum_{r=1}^{\infty}\frac{(e^{i\pi})^{r}}{r+C}=\sum_{i\in I}x_{i}\cdot
L(1,\chi^{(i)})\label{lcca1a}%
\end{equation}
for a collection of non-principal Dirichlet characters $\chi^{(i)}$ modulo $2$
(for $C=0$) resp. $2v$ (for $C=\frac{u}{v}$). (Of course, there is only one
such Dirichlet character for $C=0$, but let us ignore this simplification).
Finally, note that we have%
\[
\frac{1}{r(r+m)}=\frac{1}{m}\left(  \frac{1}{r}-\frac{1}{r+m}\right)
\]
as $m\neq0$. Thus, for $\left\vert w\right\vert <1$ we have%
\[
S_{m}(w)=\frac{w^{m}}{m}\left(  \sum_{r=1}^{\infty}\frac{w^{r}}{r}-\sum
_{r=1}^{\infty}\frac{w^{r}}{r+m}\right)  \text{.}%
\]
We get $S_{m}(-1)$ if we plug in $w:=e^{i\pi}$. Although this does not satisfy
$\left\vert w\right\vert <1$, it is consistent with our definition of
$S_{m}(-1)$ by line \ref{l_rmk_SqMinusOne} and the series are conditionally
convergent. Since $w^{m}\in\overline{\mathbf{Q}}$, Equation \ref{lcca1a}
implies our claim.
\end{proof}

If $\Gamma$ denotes the Gamma function, the \emph{digamma function} is defined
as its logarithmic derivative, i.e.
\[
\psi(z):=\frac{\Gamma^{\prime}(z)}{\Gamma(z)}\text{.}%
\]
The standard functional equation of the Gamma function implies that%
\begin{equation}
\psi(z+1)=\psi(z)+\frac{1}{z}\text{.}\label{lmf3}%
\end{equation}
Write $\gamma$ for the Euler--Mascheroni constant. The Weierstrass product
formula for the Gamma function yields%
\[
\frac{1}{\Gamma(z)}=ze^{\gamma z}\prod_{r=1}^{\infty}\left(  1+\frac{z}%
{r}\right)  e^{-\frac{z}{r}}\text{.}%
\]
Thus, taking its logarithmic derivative,%
\[
-\psi(z)=\gamma+\frac{1}{z}+\sum_{r=1}^{\infty}\left(  \frac{1}{r+z}-\frac
{1}{r}\right)  =\gamma+\frac{1}{z}-z\sum_{r=1}^{\infty}\frac{1}{(r+z)r}%
\text{.}%
\]
Hence, by the functional equation of Equation \ref{lmf3},%
\begin{equation}
\frac{\left(  \psi(z)+\gamma\right)  +\frac{1}{z}}{z}=\sum_{r=1}^{\infty}%
\frac{1}{(r+z)r}\text{.}\label{lmf4}%
\end{equation}

\begin{proposition}
\label{prop_W_m_FractionalVals}For a fraction $m=\frac{u}{v}\in\mathbf{Q}%
\setminus\mathbf{Z}$ (with $u,v\in\mathbf{Z}$ and $v\geq1$), the value of
$W_{m}$ lies in the field%
\[
\mathbf{Q}(\mu_{\infty},\pi,\{L(1,\chi)\}_{\chi})\text{,}%
\]
where $\chi$ runs through a finite set of non-principal Dirichlet characters
modulo $2v$.
\end{proposition}

\begin{proof}
We have%
\[
W_{m}=\frac{1}{2\pi}\left(  \int_{0}^{\pi}+\int_{\pi}^{2\pi}\right)
\log\left(  2\sin\left(  \frac{\theta}{2}\right)  \right)  e^{im\theta
}\,\mathrm{d}\theta\text{,}%
\]
so $W_{m}$ is a sum of two terms whose shape we understand thanks to Prop.
\ref{prop_log_val_for_rationalarg_helper}. In this presentation, we just have
algebraic numbers in a cyclotomic field, $\pi$, $S_{m}(-1)$ whose structure is
settled by\ Lemma \ref{Lemma_LerchTranscendent}, and $S_{m}(1)$. Finally,
by\ Equation \ref{lmf4},%
\[
S_{m}(1)=\sum_{r=1}^{\infty}\frac{1}{r(r+m)}=\frac{\left(  \psi(m)+\gamma
\right)  +\frac{1}{m}}{m}\text{.}%
\]
Now, $\psi(m)+\gamma$ for $m=\frac{u}{v}\in\mathbf{Q}\setminus\mathbf{Z}$ can
be written as a $\mathbf{Q}(\mu_{\infty})$-linear combination of values
$\log(1-\zeta_{v}^{i})$ for $\zeta_{v}$ a primitive $v$-th root of unity
\cite[Lemma 21]{MR2332591}, and these in turn as $L(1,\chi)$ for suitable
$\chi$.
\end{proof}

\begin{aside}
Suppose $m\in\mathbf{Q}\setminus\mathbf{Z}$. Then it was proven by Bundschuh
that $\psi\left(  m\right)  +\gamma$ is transcendental. This is \cite[Korollar
1]{MR555344}. More is known, e.g. on the linear independence of such values
over $\mathbf{Q}$, \cite[Theorem 4]{MR2332591} (however, over $\mathbf{Q}%
(\mu_{\infty})$ the situation is less clear).
\end{aside}

\subsection{Proof of Theorem \ref{thm_ortho}}

\begin{proof}
[Proof of Theorem \ref{thm_ortho}]\textit{(Claim 1) }We claim that the
function%
\begin{equation}
f(t):=\log\left\vert 1-e^{2\pi it}\right\vert \cdot e^{2\pi imt}%
\qquad\text{for}\qquad t\in(0,1)\label{lftog}%
\end{equation}
satisfies $f\in\operatorname*{BS}(\{0,1\})$: It is clearly continuous on
$(0,1)$. Next, treat the real and imaginary parts separately. They have a very
different boundary behaviour, as witnessed by the following figure of the
graphs for $m=0,1,2$.%
\[%
{\includegraphics[
height=0.6685in,
width=1.3578in
]%
{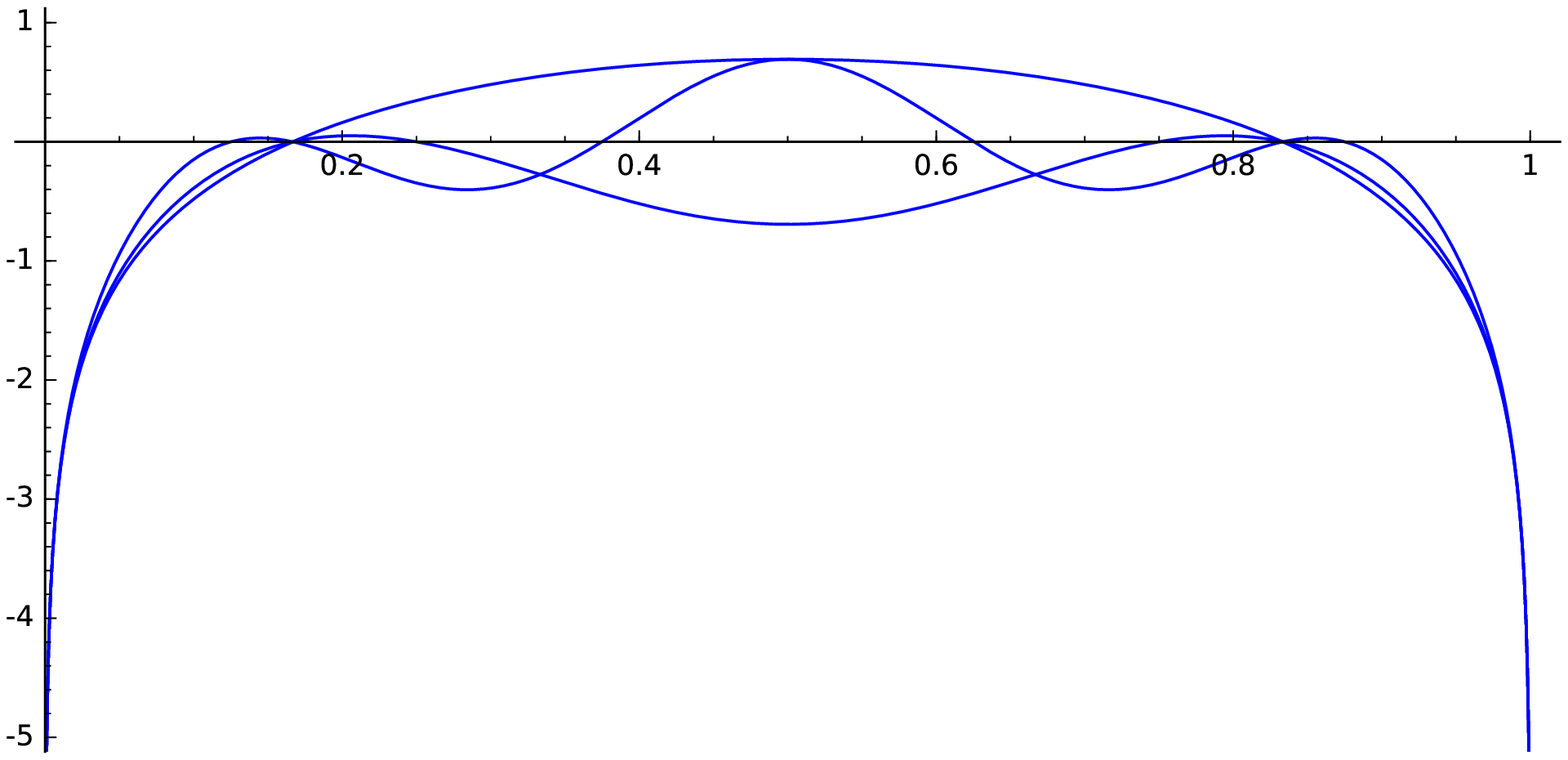}%
}
\qquad\qquad%
{\includegraphics[
height=0.6849in,
width=1.3923in
]%
{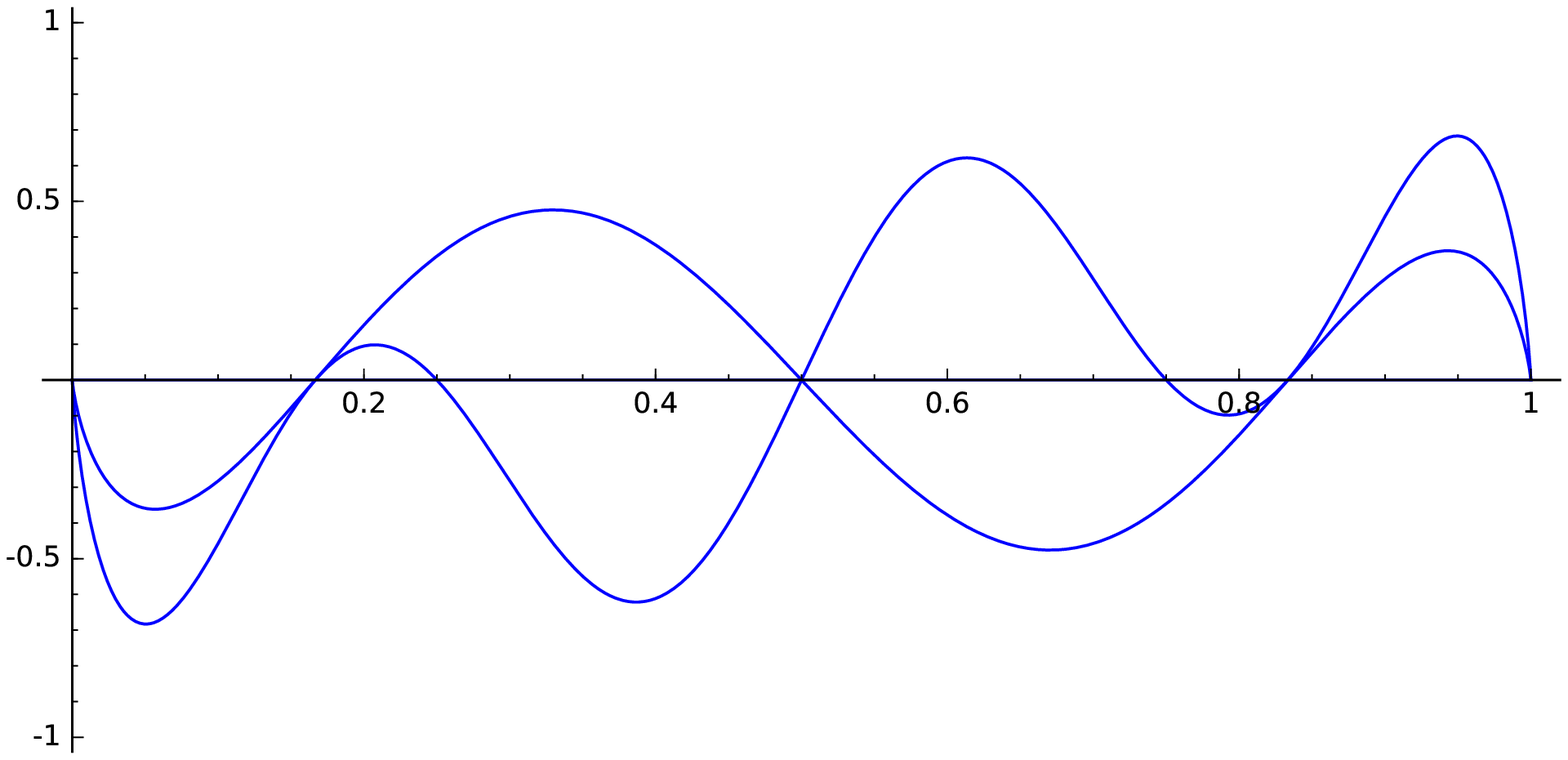}%
}
\]
\newline The real part is $\operatorname{Re}f(t)=\log\left\vert 1-e^{2\pi
it}\right\vert \cdot\cos(2\pi mt)$. For the limit $t\rightarrow0$, we easily
find that we have $\log\left\vert 1-e^{2\pi it}\right\vert \rightarrow-\infty$
and since $\cos(0)=+1$ and $r>0$, we get $\lim_{t\rightarrow0,t\in
(0,1)}\operatorname{Re}f(t)=-\infty$. For the limit $t\rightarrow1$, the same
happens. For all $t\in(0,\frac{1}{2})$, the function $-\log\left\vert
1-e^{2\pi it}\right\vert $ is monotonously decreasing, and for $t>0$ staying
sufficiently small, $\cos(2\pi mt)$ is also monotonously decreasing (or if
$m=0$ constant). Either way, their product is monotonously decreasing. Thus,
the negative is monotonously increasing. For $t\in(\frac{1}{2},1)$, proceed
symmetrically.\ This finishes the real part. The imaginary part is
$\operatorname{Im}f(t)=\log\left\vert 1-e^{2\pi it}\right\vert \sin(2\pi mt)$.
We compute, using Formula \ref{eq_logabs_to_sine},%
\begin{equation}
\underset{t\rightarrow0,t\in(0,1)}{\lim}\operatorname{Im}f(t)=\underset
{t\rightarrow0,t\in(0,1)}{\lim}\log\left(  2\sin\pi t\right)  \sin(2\pi
mt)\text{.}\label{lMb1}%
\end{equation}
We need a case distinction: If $m=0$, this is constantly zero. In particular,
the limit is zero. If $m\neq0$ so that $\sin(2\pi mt)\neq0$ near $t=0$, we may
rewrite this as
\[
\underset{t\rightarrow0,t\in(0,1)}{\lim}\log\left(  2\sin\pi t\right)
\sin(2\pi mt)=\underset{t\rightarrow0,t\in(0,1)}{\lim}\frac{\log\left(
2\sin\pi t\right)  }{\left(  \frac{1}{\sin(2\pi mt)}\right)  }\text{,}%
\]
and for $t\rightarrow0$, $t>0$, both numerator and denominator diverge to
$+\infty$, so by the L'H\^{o}pital Rule,%
\[
=-\underset{t\rightarrow0,t\in(0,1)}{\lim}\frac{\cos(\pi t)}{2m\cos(2\pi
mt)}\frac{\sin(2\pi mt)^{2}}{\sin(\pi t)}\text{.}%
\]
The first factor converges to some non-zero value, and since the sine is
$\sin(z)=z+O(z^{2})$ to first order, the second factor tends to zero as
$t\rightarrow0$. It follows that the limit in line \ref{lMb1} is always zero.
A very similar computation can be made for the limit $t\rightarrow1$. We leave
this to the reader. It follows that $\operatorname{Im}f(t)$ may be continued
to a continuous function on all of $[0,1]$, so it even lies in
$\operatorname*{BS}(\varnothing)$.

This being settled, we observe that%
\[
\underset{N\rightarrow\infty}{\lim}\frac{1}{N}\sum_{n=1}^{N}\log\left\vert
1-e^{2\pi in\theta}\right\vert e^{2\pi imn\theta}=\underset{N\rightarrow
\infty}{\lim}\frac{1}{N}\sum_{n=1}^{N}\log\left\vert 1-e^{2\pi i\{n\theta
\}}\right\vert e^{2\pi im\{n\theta\}}%
\]
and we want to apply Baxa--Schoi\ss engeier Equidistribution, Theorem
\ref{Thm_BaxaSchoissengeierEquidist}. To be able to invoke this result, it
remains to check the condition%
\begin{equation}
\underset{n\rightarrow\infty}{\lim}\frac{f(\{n\theta\})}{n}=0\text{.}%
\label{lBB1}%
\end{equation}
(Case A) If $u:=e^{2\pi i\theta}$ is algebraic, it cannot be a root of unity,
since that would contradict our assumption $\dim_{\mathbf{Q}}\left\langle
1,\theta\right\rangle =2$. We also have $\left\vert e^{2\pi i\theta
}\right\vert =1$, so we can use the typical diophantine estimate: Namely,
\[
\frac{\left\vert f(\{n\theta\})\right\vert }{n}=\frac{\left\vert
\log\left\vert 1-e^{2\pi in\theta}\right\vert e^{2\pi im\{n\theta
\}}\right\vert }{n}\leq\frac{1}{n}\log\left\vert 1-e^{2\pi in\theta
}\right\vert =\frac{1}{n}\log\left\vert 1-u^{n}\right\vert \text{.}%
\]
We have $\left\vert u\right\vert =1$ and since $\theta$ is irrational, $u$
cannot be a root of unity. Thus, Lemma \ref{Lemma_CriticalUpperBound} (which
in turn hinges on the Gelfond estimate) is available and implies%
\[
\frac{\left\vert f(\{n\theta\})\right\vert }{n}\leq\frac{C_{u}\log n}{n}%
\]
for a suitable constant $C_{u}>0$ and this converges to zero for
$n\rightarrow+\infty$. Thus, Equation \ref{lBB1} is fine.\newline(Case B) If
$\theta$ is badly approximable, Equation \ref{lBB1} automatically holds, see
Theorem \ref{Thm_BaxaSchoissengeierEquidist}.

Thus, in either case, Theorem \ref{Thm_BaxaSchoissengeierEquidist} applies. We
get%
\[
\underset{N\rightarrow\infty}{\lim}\frac{1}{N}\sum_{n=1}^{N}\log\left\vert
1-e^{2\pi i\{n\theta\}}\right\vert e^{2\pi im\{n\theta\}}=\int_{0}^{1}%
\log\left\vert 1-e^{2\pi it}\right\vert e^{2\pi imt}\,\mathrm{d}t=W_{m}%
\]
(with $W_{m}$ as in line \ref{Def_Wm}) and thus Lemma \ref{Lemma_WmCompute}
and Prop. \ref{prop_W_m_FractionalVals} yield our claim.\newline\textit{(Claim
2)} The proof is similar, but now we work on $[0,1]^{2}$. We have%
\[
\underset{N\rightarrow\infty}{\lim}\frac{1}{N}\sum_{n=1}^{N}\log\left\vert
1-e^{2\pi in\theta}\right\vert e^{2\pi imn\alpha}=\underset{N\rightarrow
\infty}{\lim}\frac{1}{N}\sum_{n=1}^{N}\log\left\vert 1-e^{2\pi i\{n\theta
\}}\right\vert e^{2\pi im\{n\alpha\}}\text{.}%
\]
As $\dim_{\mathbf{Q}}\left\langle 1,\theta,\xi\right\rangle =3$ by assumption,
the sequence $(\{n\theta\},\{n\alpha\})$ is uniformly distributed in
$[0,1]^{2}\subset\mathbf{R}^{2}$ (Lemma
\ref{lemma_QLinearIndependentMeansTorusUniformlyDist}). We claim that the
function%
\[
f(s,t):=\log\left\vert 1-e^{2\pi is}\right\vert \cdot e^{2\pi imt}%
\]
lies in $\operatorname*{BSU}^{2}(\{0,1\})$ for the singular weight
$g(s):=\log\left\vert 1-e^{2\pi is}\right\vert $ and $h(s,t):=e^{2\pi imt}$.
Since $g$ is the special case $m=0$ of the function in Equation \ref{lftog},
we have already settled that $g\in\operatorname*{BS}(\{0,1\})$ and $h$ is
clearly Riemann-integrable on $[0,1]^{2}$. Moreover, the condition of Equation
\ref{lBB1} only needs to be verified for the singular weight, which we have
also already done in the proof of \textit{Claim 1}. Thus, Theorem
\ref{thm_bs_multi} applies and we get%
\[
=\int_{0}^{1}\int_{0}^{1}\log\left\vert 1-e^{2\pi is}\right\vert \cdot e^{2\pi
imt}\,\mathrm{d}s\,\mathrm{d}t=\left(  \int_{0}^{1}\log\left\vert 1-e^{2\pi
is}\right\vert \,\mathrm{d}s\right)  \cdot\left(  \int_{0}^{1}e^{2\pi
imt}\,\mathrm{d}t\right)  \text{.}%
\]
The first integral, by definiton, agrees with $W_{0}$, but we already know
that $W_{0}=0$ by Lemma \ref{Lemma_WmCompute} (in particular, we do not even
need to use that for $m\neq0$, the second factor also vanishes). This finishes
the proof.
\end{proof}

\section{Limit values near the unit circle}

In this section we shall handle the only remaining case: $R_{x}$ for
$\left\vert x\right\vert =1$ and $x$ is not a root of unity.

\begin{theorem}
\label{thm_RxAtUnimodularHasNaturalBoundary}Suppose $x\in\mathbf{C}$ is an
algebraic integer with $\left\vert x\right\vert =1$ and $x$ is not a root of
unity. Then for every point $p\in\wp$ in $\wp:=\{x^{m}\mid m\in\mathbf{Z}\}$,
we have%
\[
\operatorname*{limrad}_{z\rightarrow p}\,(1-\left\vert z\right\vert
)R_{x}(z)=-\frac{1}{2\left\vert m\right\vert }\delta_{m\neq0}\text{.}%
\]
In particular, the unit circle is the natural boundary for $R_{x}$. For any
fractional exponent $m$ one still has%
\[
\operatorname*{limrad}_{z\rightarrow p}\,(1-\left\vert z\right\vert
)R_{x}(z)\in\mathbf{Q}(\mu_{\infty},\pi,\{L(1,\chi)\}_{\chi})\text{,}%
\]
where $\chi$ runs through a finite set (depending on $p$) of non-principal
Dirichlet characters of various moduli. Write $x=e^{2\pi i\theta}$. Then for
every point $p=e^{2\pi i\alpha}$ such that $1,\theta,\alpha$ are $\mathbf{Q}%
$-linearly independent inside the real numbers, we have%
\[
\operatorname*{limrad}_{z\rightarrow p}\,(1-\left\vert z\right\vert
)R_{x}(z)=0\text{.}%
\]

\end{theorem}

It will be more natural to handle this type of result by working with the
argument $\theta$ as opposed to $x=e^{2\pi i\theta}$ itself, so we switch to
this viewpoint in this section: For $\theta\in(0,1)$ irrational, we define the
power series%

\[
Y_{\theta}(z):=\sum_{n\geq1}\log\left\vert 1-e^{2\pi in\theta}\right\vert
\cdot z^{n}\text{.}%
\]
We need a special notation: For a point $p$ on the unit circle, a sequence
$(z_{n})$ \emph{converges} \emph{radially} to $p$ if $\arg z_{n}$ is constant
for all sufficiently large $n$. The corresponding concept of limit is%
\begin{equation}
\operatorname*{limrad}_{z\rightarrow p}\,f(z_{n}):=\lim_{\substack{z:=r\cdot p
\\r\rightarrow1,r\in(0,1)}}f(z)\text{.}\label{luiav9}%
\end{equation}
The case of interest for this definition are functions $f$ which vary wildly
with the argument.

Based on Theorem \ref{thm_ortho}, we can now prove the following
characterization of the radial limit values when we approach the unit circle:

\begin{theorem}
\label{Thm_CriticalThmOnSingBvrOnRadiusOfConvergence}Suppose an irrational
$\theta\in(0,1)$ is given. Moreover, assume

\begin{itemize}
\item $u:=e^{2\pi i\theta}$ is an algebraic integer, or

\item $\theta$ is badly approximable.
\end{itemize}

For every point $p\in\wp$ with $\wp:=\{e^{2\pi im\theta}\mid m\in\mathbf{Z}%
\}$, we have%
\[
\operatorname*{limrad}_{z\rightarrow p}\,(1-\left\vert z\right\vert
)Y_{\theta}(z)=-\frac{1}{2\left\vert m\right\vert }\delta_{m\neq0}\text{.}%
\]
For any fractional exponent $m$ one still has%
\[
\operatorname*{limrad}_{z\rightarrow p}\,(1-\left\vert z\right\vert
)Y_{\theta}(z)\in\mathbf{Q}(\mu_{\infty},\pi,\{L(1,\chi)\}_{\chi})\text{,}%
\]
where $\chi$ runs through a finite set (depending on $p$) of non-principal
Dirichlet characters of various moduli. For every point $p=e^{2\pi i\alpha}$
such that $1,\theta,\alpha$ are $\mathbf{Q}$-linearly independent inside the
real numbers, we have%
\[
\operatorname*{limrad}_{z\rightarrow p}\,(1-\left\vert z\right\vert
)Y_{\theta}(z)=0\text{.}%
\]
In particular, the unit circle is the natural boundary for $Y_{\theta}$.
\end{theorem}

If we only wanted the statement about the natural boundary, we could try to
invoke the following result:

\begin{theorem}
[Carroll, Kemperman]Suppose $g:[0,1]\rightarrow\mathbf{C}$ is a
Lebesgue-integrable function. Then one and only one of the following
statements is true:

\begin{enumerate}
\item For almost all $\alpha\in\mathbf{R}$, the power series%
\[
F_{\alpha}:=\sum_{n=0}^{\infty}g(\{n\alpha\})z^{n}%
\]
has the unit circle as a natural boundary.

\item The function $g$ agrees almost everywhere with a trigonometric
polynomial $\theta\mapsto\sum_{m\in\mathbf{Z}}c_{m}e^{2\pi im\theta}$ of
period $1$.
\end{enumerate}
\end{theorem}

This is \cite[Theorem 1.1]{MR0173754}. We can apply this to $g(t):=\log
\left\vert 1-e^{2\pi it}\right\vert $. As will be implicit in our proof of
Theorem \ref{Thm_CriticalThmOnSingBvrOnRadiusOfConvergence} below, we can rule
out possibility (2), so that (1)\ must be true. However, this is far too weak
for our purposes. We are interested in the case of $e^{2\pi i\theta}$ being
algebraic, which forms a countable set, or $\theta$ being badly approximable,
which is a set of measure zero\footnote{measure zero, but uncountable (the
badly approximable numbers can be identified with the set of bounded sequences
by using the partial quotients).}. Thus, our Theorem
\ref{Thm_CriticalThmOnSingBvrOnRadiusOfConvergence} makes a statement about a
set of measure zero. Since the Carroll--Kemperman result works only almost
everywhere, it is of no help. This is similar to the issue explained in Remark
\ref{rmk_ergodic_approach} that prevents us from exploiting ergodicity
directly.\medskip

Our proof is based on the following very classical result:

\begin{lemma}
[{Frobenius, \cite[Lemma 1]{MR0183707}}]\label{Lemma_FrobeniusPoleComputation}%
Suppose $C\in\mathbf{C}$ and $(a_{n})$ is a sequence of complex numbers with%
\[
\underset{N\rightarrow\infty}{\lim}\frac{1}{N}(a_{1}+\cdots+a_{N})=C\text{.}%
\]
Define a power series%
\[
F(z):=\sum_{n\geq1}a_{n}z^{n}\text{.}%
\]
If $F$ has radius of convergence $\geq1$, then $\lim_{r\rightarrow1,r\in
(0,1)}\,(1-r)F(r)=C$.
\end{lemma}

While this describes the behaviour of a radial limit point at $z=1$, the idea
is that by `rotating' a given function, we can bring any point on the unit
circle to lie at $z=1$ and apply Frobenius' Lemma there.

\begin{proof}
\textit{(1) }For any point in $\wp=\{e^{2\pi im\theta}\mid m\in\mathbf{Z}\}$,
let $m$ be chosen accordingly. We have%
\begin{equation}
Y_{\theta}(ze^{2\pi im\theta})=\sum_{n\geq1}\log\left\vert 1-e^{2\pi in\theta
}\right\vert (e^{2\pi im\theta})^{n}z^{n}\label{luia1}%
\end{equation}
and this is itself a power series in $z$, which we may temporarily denote by
$V(z)$. By Lemma \ref{Lemma_ConvergeOfRx} the series $V$ has radius of
convergence $\geq1$. Then it follows that%
\begin{equation}
\lim_{\substack{z:=re^{2\pi im\theta} \\r\rightarrow1,r<1}}(1-\left\vert
z\right\vert )Y_{\theta}(z)=\underset{r\rightarrow1,r\in(0,1)}{\lim
}(1-r)Y_{\theta}(re^{2\pi im\theta})=\underset{r\rightarrow1,r\in(0,1)}{\lim
}(1-r)V(r)\text{.}\label{luia3}%
\end{equation}
By Lemma \ref{Lemma_FrobeniusPoleComputation}, this limit equals%
\[
=\underset{N\rightarrow\infty}{\lim}\frac{1}{N}\sum_{n=1}^{N}(n\text{-th
coefficient of }V)=\underset{N\rightarrow\infty}{\lim}\frac{1}{N}\sum
_{n=1}^{N}\log\left\vert 1-e^{2\pi in\theta}\right\vert e^{2\pi imn\theta
}=-\frac{1}{2\left\vert m\right\vert }\delta_{m\neq0}%
\]
by the Orthogonality Theorem, Theorem \ref{thm_ortho}. This proves the first
part of the claim. Moreover, $-\frac{1}{2\left\vert m\right\vert }%
\delta_{m\neq0}$ is non-zero for all $m\neq0$. In particular, $\wp
\setminus\{1\}$ lies in the set of singular points on the radius of
convergence. As $\wp$ is already dense in the unit circle, the same must be
true for the set of singular points. Thus, the unit circle is the natural
boundary of the power series. For fractional $m$, use the corresponding
statement of Theorem \ref{thm_ortho}.\newline\textit{(2)} Now suppose that a
point $p=e^{2\pi i\alpha}$ is given such that $1,\theta,\alpha$ are
$\mathbf{Q}$-linearly independent inside the real numbers. The idea of the
following proof is taken from the proof of \cite[Ch. I, Theorem 6.6]%
{MR0419394}. As in line \ref{luia1}, we `rotate' the function $Y_{\theta}(z)$:
This time, consider%
\[
Y_{\theta}(ze^{2\pi i\alpha})=\sum_{n\geq1}\log\left\vert 1-e^{2\pi in\theta
}\right\vert (e^{2\pi i\alpha})^{n}z^{n}%
\]
and write $V$ for this function, viewed as a power series in $z$. Proceed as
in line \ref{luia3} and again by Lemma \ref{Lemma_FrobeniusPoleComputation},
the limit turns out to be%
\[
=\underset{N\rightarrow\infty}{\lim}\frac{1}{N}\sum_{n=1}^{N}\log\left\vert
1-e^{2\pi in\theta}\right\vert e^{2\pi imn\alpha}\text{.}%
\]
However, this vanishes by Orthogonality, Theorem \ref{thm_ortho}.
\end{proof}

Clearly this theorem immediately implies Theorem
\ref{thm_RxAtUnimodularHasNaturalBoundary}.

\section{Proof of the main theorems}

As we had explained in the introduction, the tools of the previous sections
can be used in quite varied applications. This has to do with the fact that
the underlying counting problem has shown up in a variety of contexts, which
often have no immediate philosophical connection, yet on a technical level
lead to formally entirely equivalent problems.

\subsection{Classical knots}

Let $K\subset S^{3}$ be a knot\footnote{for us, a knot is always a tame knot
embedded into the $3$-sphere.} and $X_{K}$ the knot exterior, i.e. $X_{K}$ is
a compact real $3$-manifold with boundary. One always has%
\[
H_{1}(X_{K},\mathbf{Z})\simeq\mathbf{Z}\text{.}%
\]
This isomorphism is not canonical, there are two possible choices, but a
posteriori it turns out that the choice does not matter. The Hurewicz Theorem
gives us a canonical surjection%
\[
\pi_{1}(X_{K},\ast)\twoheadrightarrow H_{1}(X_{K},\mathbf{Z})\cong
\mathbf{Z}\text{,}%
\]
which is just the abelianization of the fundamental group. As quotients of the
fundamental group correspond to Galois covering spaces with the corresponding
deck transformation action, this surjection defines an infinite covering
$X_{\infty}\longrightarrow X$ with a canonical $\mathbf{Z} $-action, as well
as finite coverings $X_{r}\longrightarrow X$ with $\mathbf{Z}/r\mathbf{Z}%
$-actions. The $\mathbf{Z}$-action on $X_{\infty}$ induces an action to
homology, so the finitely generated group $H_{1}(X_{\infty},\mathbf{Z})$
carries an action by the group ring of $\mathbf{Z}$. Hence, it is canonically
a $\mathbf{Z}[t,t^{-1}]$-module. By classical work of Alexander, this module
structure has a rather simple structure, namely%
\[
H_{1}(X_{\infty},\mathbf{Z})\cong\mathbf{Z}[t,t^{-1}]/(\Delta_{K}%
)\qquad\text{with}\qquad\Delta_{K}\in\mathbf{Z}[t,t^{-1}]\text{.}%
\]
The element $\Delta_{K}$ is the \emph{Alexander polynomial}. It is only
well-defined up to a unit $\mathbf{Z}[t,t^{-1}]^{\times}\simeq\left\langle
\pm1\right\rangle \times t^{\mathbf{Z}}$. Various normalizations are possible,
but for us any choice of a representative in $\mathbf{Z}[t]$ will be
fine\footnote{This is unnatural from the viewpoint of skein relations, but
more convenient ring-theoretically.}.

\begin{definition}
We call the roots of $\Delta_{K}$ the \emph{Alexander roots}. Such a root
$\beta$ is called \emph{diophantine} if $\left\vert \beta\right\vert =1$ and
$\beta\notin\mu_{\infty}$, i.e. if it lies on the unit circle, but is not a
root of unity.
\end{definition}

We prove a refinement of the Silver--Williams theorem \cite[Theorem
2.1]{MR1923995} in the case of knots. There is a fundamental dichotomy,
depending on whether there is a diophantine Alexander root or not. Let us
begin with the (typical) case in which there is no diophantine root.

\begin{theorem}
\label{thm_A_for_knots}Let $K\subset S^{3}$ be a knot and $\Delta_{K}$ its
Alexander polynomial. If each root $\beta_{i}$ of $\Delta_{K}$ either has
absolute value $\left\vert \beta_{i}\right\vert \neq1$ or is a root of unity,
then the generating function of torsion homology growth%
\[
E_{K}(z):=\sum_{r=1}^{\infty}\log\left\vert H_{1}(X_{r},\mathbf{Z}%
)_{\operatorname*{tor}}\right\vert \cdot z^{r}%
\]
has radius of convergence $1$. However,

\begin{enumerate}
\item $E_{K}$ admits a meromorphic continuation to the entire complex plane.

\item Its poles are located at all integer powers of roots of $\Delta_{K}$
which lie outside the open unit disc, i.e.%
\[
\{\beta^{n}\mid\Delta_{K}(\beta)=0\text{ and }n\in\mathbf{Z}\}\setminus
(\text{open unit disc})\text{.}%
\]
At $z=1$ the pole has order $1$ or $2$. All other poles are of order $1$.

\item The Laurent expansion at $z=1$ begins with%
\begin{align*}
E_{K}(z)  & =\frac{1}{(z-1)^{2}}\log\mathcal{M}(\Delta_{K})\\
& +\frac{1}{z-1}\left(  \log\mathcal{M}(\Delta_{K})+\sum_{\substack{\beta
_{i}\in\mu_{\infty} \\m:=\operatorname*{ord}(\beta_{i})}}\frac{1}{m}%
\log\left(  \frac{1}{m}\right)  \right) \\
& -\sum_{\beta_{i},\left\vert \beta_{i}\right\vert \neq1}\log\left\vert
F(\beta_{i}^{\pm})\right\vert -\sum_{\substack{\beta_{i}\in\mu_{\infty}
\\m:=\operatorname*{ord}(\beta_{i})}}\left(  \frac{m-1}{2}\frac{1}{m}%
\log\left(  \frac{1}{m}\right)  +\frac{1}{m}\sum_{l=1}^{m-1}l\cdot
\log\left\vert 1-\beta_{i}^{l}\right\vert \right) \\
& +(z-1)\cdot\operatorname*{holomorphic}\text{.}%
\end{align*}
Here the sums are taken over the roots $\beta_{1},\ldots,\beta_{n}$ of the
Alexander polynomial.
\end{enumerate}
\end{theorem}

In principle, by assembling our results, we can give the entire Laurent
expansion at $z=1$ as a closed formula. We leave this to the interested reader.

\begin{remark}
It might be worth to sketch the information on the pole loci of the theorem in
graphical format. We find%
\[%
{\includegraphics[
height=2.1162in,
width=3.2586in
]%
{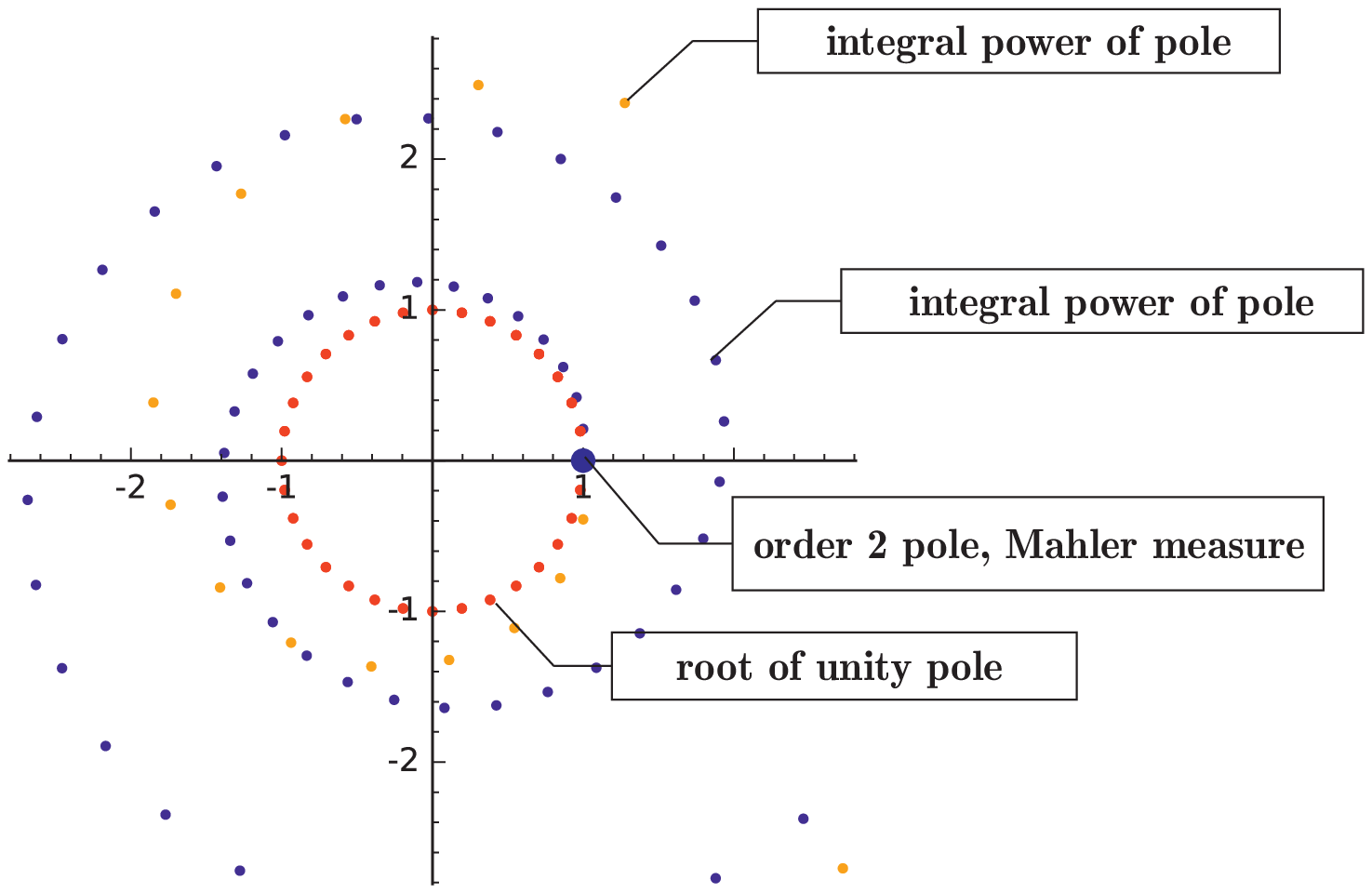}%
}
\]
We write \textquotedblleft Mahler measure\textquotedblright\ for the pole at
$z=1$ as the Theorem shows that we can read off the Mahler measure from the
principal part of the Laurent expansion at this point. As we had explained in
\S \ref{sect_HeuristicMotivation}, the presence of the Mahler measure in this
expansion is quite literally equivalent to the Silver--Williams asymptotic.
\end{remark}

Under the assumptions of the theorem, we immediately recover the following result:

\begin{theorem}
[Fried \cite{MR980956}]\label{thm_Fried}Let $K\subset S^{3}$ be a knot such
that its Alexander polynomial $\Delta_{K}$ has no roots in $\mu_{\infty}$.
Then knowing the values $H_{1}(X_{r},\mathbf{Z})_{\operatorname*{tor}}$ for
all $r\geq1$ uniquely determines $\Delta_{K}$.
\end{theorem}

We point out that Fried's result does not come with an easy description how
$\Delta_{K}$ is to be recovered from the homology torsion cardinalities.

\begin{proof}
[New proof of special case]Suppose $\Delta_{K}$ has no roots on the unit
circle. By our Theorem, the poles of $E_{K}$ tell us all integer pole powers,
so since $\mathbf{Z}$ has only two possible generators, $+1$ or $-1$, we can
reconstruct $z$ or $z^{-1} $ for each root of $\Delta_{K}$. However, the
Alexander polynomial is reciprocal, so if $z$ is a root, $z^{-1}$ is also a root.
\end{proof}

Fried's result has recently found the following application:

\begin{theorem}
[{Boileau--Friedl \cite[Prop. 4.10]{bfprofinite}}]Let $K_{1},K_{2}\subset
S^{3}$ be knots such that their knot groups have isomorphic profinite
completion. If neither Alexander polynomial has a root in $\mu_{\infty}$, both
knots have the same Alexander polynomial (up to as unique Alexander
polynomials are).
\end{theorem}

So, if neither Alexander polynomial has a root on the unit circle, the
proof\textit{\ loc. cit.} also works with Theorem \ref{thm_A_for_knots}
instead, but use residual finiteness as in \cite[Lemma 4.2]{bfprofinite2}. A
recent development is due to Ueki \cite{ueki}.\medskip

As we had already explained in the narrative of the introduction, everything
changes drastically if $\Delta_{K}$ \textit{does} have a diophantine root.
Then an analytic continuation is impossible. Nonetheless, one can read of a
lot of the data which was previously packaged in the poles from the singular
values on the radius of convergence:

\begin{definition}
\label{def_MultiplicativeIndependence}A set of elements $x_{1},\ldots,x_{r}%
\in\mathbf{C}$, all on the unit circle, will be called \emph{multiplicatively
(in)dependent} if the real numbers%
\[
\{\arg x_{1},\ldots,\arg x_{r}\}\subset\mathbf{R}%
\]
are linearly (in)dependent over the rationals.
\end{definition}

Moreover, we say that something holds \textquotedblleft for all sufficiently
divisible $m$\textquotedblright\ if there exists some integer $N$ such that
the statement holds for all $n$ which are divisible by $N$.

\begin{theorem}
\label{thm_A_for_knots_NatBdryCase}Let $K\subset S^{3}$ be a knot and
$\Delta_{K}$ its Alexander polynomial. If it has at least one diophantine
root, then $E_{K}$ has the unit circle as its natural boundary. Let $p$ be a
point of the unit circle.

\begin{enumerate}
\item If $p$ is multiplicatively independent from all diophantine roots, then%
\[
\operatorname*{limrad}_{z\rightarrow p}\,(1-\left\vert z\right\vert
)E_{K}(z)=0\text{.}%
\]

\item If $p$ is multiplicatively dependent of the diophantine roots, then for
all sufficiently divisible $m\geq1$,%
\[
\operatorname*{limrad}_{z\rightarrow p^{m}}\,(1-\left\vert z\right\vert
)E_{K}(z)\in\mathbf{Q}_{\lneqq0}%
\]
is a (strictly) negative rational number.

\item If $p$ is multiplicatively dependent of the diophantine roots, then%
\[
\operatorname*{limrad}_{z\rightarrow p^{m}}\,(1-\left\vert z\right\vert
)E_{K}(z)\in\mathbf{Q}(\mu_{\infty},\pi,\{L(1,\chi)\}_{\chi})\text{,}%
\]
where $\chi$ runs through a finite set (depeding on $p$) of non-principal
Dirichlet characters of various moduli\footnote{Depending on the worst
denominator in the multiplicative dependency relation, one can bound the
necessary supply of moduli for the required $\chi$; see Theorem
\ref{thm_ortho}. We leave it to the reader to spell this out.}.
\end{enumerate}

In particular, the rational span%
\[
\mathbf{Q}\left\langle (\arg\beta_{i})_{\beta_{i}\text{ a diophantine root of
}\Delta_{K}}\right\rangle
\]
inside the real numbers can be read off the boundary value behaviour of
$E_{K}$ at the unit circle.
\end{theorem}

In order to prove Theorem \ref{thm_A_for_knots} and Theorem
\ref{thm_A_for_knots_NatBdryCase}, we need some preparations along the lines
of \cite{MR1923995}.

\begin{remark}
[Branched coverings of the $3$-sphere]\label{rmk_BranchedCoverings}%
Historically, this story was being looked at from a slightly different
perspective. In \cite{MR0295327}, \cite{MR1923995} one considers
\textsl{branched} coverings $\widehat{X}_{\infty}$ resp. $\widehat{X}_{r}$
over $S^{3}$, instead of the spaces $X_{\infty}$ resp. $X$ over the knot
complement. This is explained e.g. \cite[Ch. 8, E, \S 8.18]{MR1959408}. They
sit in a square%
\[%
\begin{array}
[c]{ccc}%
X_{r} & \hookrightarrow & \widehat{X}_{r}\\
\downarrow &  & \downarrow\\
X_{K} & \hookrightarrow & S^{3}%
\end{array}
\]
and one has $H_{1}(X_{r},\mathbf{Z})\cong\mathbf{Z}\oplus H_{1}(\widehat
{X}_{r},\mathbf{Z})$. See \cite[8.19 (d), Prop.]{MR1959408}.
\end{remark}

\begin{theorem}
[Fox]\label{thm_Fox}Let $K\subset S^{3}$ be a knot. If $\Delta_{K}$ has no
roots which are roots of unity, the homology groups $H_{1}(\widehat{X}%
_{r},\mathbf{Z})$ are finite. In this case, $\left\vert H_{1}(\widehat{X}%
_{r},\mathbf{Z})\right\vert =\prod_{\zeta\in\mu_{r}}\left\vert \Delta
_{K}(\zeta)\right\vert $
\end{theorem}

This formula is due to Fox \cite[\S 6, (6.1) and (6.3)]{MR0095876}, modulo a
some corrections in the proof due to Weber \cite{MR570312}.

\begin{proof}
[Proof of Theorem \ref{thm_A_for_knots} and Theorem
\ref{thm_A_for_knots_NatBdryCase}]Let $\Delta_{K}(t)=a\prod_{i=1}^{n}%
(t-\beta_{i})\in\mathbf{Z}[t]$ be the Alexander polynomial, factored over
$\mathbf{C}$. According to Remark \ref{rmk_BranchedCoverings} and Fox'
formula, Theorem \ref{thm_Fox}, we have%
\begin{equation}
\left\vert H_{1}(X_{r},\mathbf{Z})_{\operatorname*{tor}}\right\vert
=\left\vert H_{1}(\widehat{X}_{r},\mathbf{Z})\right\vert =\prod_{\zeta\in
\mu_{r}}\left\vert \Delta_{K}(\zeta)\right\vert =\left\vert a\right\vert
^{r}\prod_{i=1}^{n}\left\vert \prod_{\zeta\in\mu_{r}}(\beta_{i}-\zeta
)\right\vert =\left\vert a\right\vert ^{r}\prod_{i=1}^{n}\left\vert
1-\beta_{i}^{r}\right\vert \label{l_C_translate_root}%
\end{equation}
since $T^{r}-1=\prod_{\zeta\in\mu_{r}}(T-\zeta)$.\ Hence,%
\begin{align*}
E_{K}(z)  & =\log\left\vert a\right\vert \cdot\frac{z}{(z-1)^{2}}+\sum
_{i=1}^{n}\sum_{r=1}^{\infty}\log\left\vert 1-\beta_{i}^{r}\right\vert \cdot
z^{r}\\
& =\log\left\vert a\right\vert \frac{z}{(z-1)^{2}}+\sum_{\left\vert \beta
_{i}\right\vert <1}R_{\beta_{i}}(z)+\sum_{\left\vert \beta_{i}\right\vert
>1}R_{\beta_{i}}(z)+\sum_{\left\vert \beta_{i}\right\vert =1}R_{\beta_{i}%
}(z)\text{.}%
\end{align*}
We use Lemma \ref{Lemma_RSwapXToXInverse} for each summand with $\left\vert
\beta_{i}\right\vert >1$. This yields
\begin{align}
& =\frac{z}{(z-1)^{2}}\log\mathcal{M}(\Delta_{K})+\sum_{\left\vert \beta
_{i}\right\vert <1}R_{\beta_{i}}(z)\label{lWA1}\\
& \qquad\qquad\qquad+\sum_{\left\vert \beta_{i}\right\vert >1}R_{\beta
_{i}^{-1}}(z)+\sum_{\substack{\left\vert \beta_{i}\right\vert =1 \\\beta
_{i}\in\mu_{\infty}}}R_{\beta_{i}}(z)+\sum_{\substack{\left\vert \beta
_{i}\right\vert =1 \\\beta_{i}\notin\mu_{\infty}}}R_{\beta_{i}}(z)\text{,}%
\nonumber
\end{align}
where $\mathcal{M}(\Delta_{K})$ denotes the Mahler measure of the Alexander
polynomial. It is independent of the choice of the representative for
$\Delta_{K}$. Now, we need a case distinction.\newline(Case A)\ Suppose that
there is no Alexander root $\beta_{i}$ with $\left\vert \beta_{i}\right\vert
=1$ and $\beta_{i}\notin\mu_{\infty}$. Hence, Equation \ref{lWA1} simplifies
to%
\begin{equation}
=\frac{z}{(z-1)^{2}}\log\mathcal{M}(\Delta_{K})+\sum_{\left\vert \beta
_{i}\right\vert <1}R_{\beta_{i}}(z)+\sum_{\left\vert \beta_{i}\right\vert
>1}R_{\beta_{i}^{-1}}(z)+\sum_{\substack{\left\vert \beta_{i}\right\vert =1
\\\beta_{i}\in\mu_{\infty}}}R_{\beta_{i}}(z)\text{.}\label{lBBA2}%
\end{equation}
By Prop. \ref{prop_Rx_AtRootOfUnity} the power series $R_{\beta_{i}}$ with
$\beta_{i}\in\mu_{\infty}$ admit a meromorphic continuation to the entire
complex plane with poles precisely at the finite set $\{\beta_{i}^{r}\mid
r\in\mathbf{Z}\}$, and all these are of order $1$. The other summands only
feature the power series $R_{\beta}$ for a parameter $\beta$ such that
$\left\vert \beta\right\vert <1$. By Theorem
\ref{thm_AnalyticContinuationForRx} any such $R_{\beta}$ admits a meromorphic
continuation to the entire complex plane whose sole poles are at
$\{\beta^{\mathbf{Z}_{\leq-1}},\overline{\beta}^{\mathbf{Z}_{\leq-1}}\}$, each
of order one. It follows that the sum of all these analytic continuations is a
meromorphic function in all of $\mathbf{C}$ whose poles are at the following
locations: (1) a pole at $z=1$ from the initial summand (only if the Mahler
measure is $\neq1$), as well as (2) poles coming from the $R_{\beta_{i}}$,
i.e. in total%
\begin{equation}
\bigcup_{i,\left\vert \beta_{i}\right\vert \neq1}\{\beta_{i}^{\pm
\mathbf{Z}_{\leq-1}},\overline{\beta}_{i}^{\pm\mathbf{Z}_{\leq-1}}%
\}\cup\bigcup_{i,\left\vert \beta_{i}\right\vert =1}\{\beta_{i}^{\mathbf{Z}%
}\}\cup\{1\}\text{,}\label{lBBA1}%
\end{equation}
where the first union runs through all roots of the Alexander polynomial
which, and we use \textquotedblleft$+$\textquotedblright\ if $\left\vert
\beta_{i}\right\vert <1$ and \textquotedblleft$-$\textquotedblright\ if
$\left\vert \beta_{i}\right\vert >1$. The pole at $z=1$ is always in this set
because the Alexander polynomial of a knot is always non-trivial, so either
the Mahler measure is $\neq1$ (so that there is an order $2$ pole at $z=1$),
or the Mahler measure is $=1$, but then there must be at least one root at a
root of unity and by Prop. \ref{prop_Rx_AtRootOfUnity} this also causes a pole
at $z=1$. We claim that this set agrees with%
\[
\{\beta^{r}\mid\Delta_{K}(\beta)=0\text{ and }r\in\mathbf{Z}\}\setminus
(\text{open unit disc})\text{.}%
\]
To see this: In Equation \ref{lBBA1} all elements lie outside the open unit
disc. If we replace all exponents $\pm\mathbf{Z}_{\leq-1}$ by $\mathbf{Z}%
_{\neq0}$, then all additional elements we get this way lie inside the open
unit disc. Thus, the set in Equation \ref{lBBA1} agrees with%
\[
\left(  \bigcup_{i,\left\vert \beta_{i}\right\vert \neq1}\{\beta
_{i}^{\mathbf{Z}_{\neq0}},\overline{\beta}_{i}^{\mathbf{Z}_{\neq0}}%
\}\cup\bigcup_{i,\left\vert \beta_{i}\right\vert =1}\{\beta_{i}^{\mathbf{Z}%
}\}\right)  \setminus(\text{open unit disc})\text{.}%
\]
All elements in this set are an integral power of an Alexander root since
$\Delta_{K}$ is a real polynomial, so if $\beta$ is a solution, so is
$\overline{\beta}$. The converse inclusion is clear. All the poles coming from
the functions $R_{\beta_{i}}$ has order $1$, so the only possibility to get a
pole of higher order is the order $2$ pole at $z=1$ potentially coming from
the initial summand in Equation \ref{lBBA2} if it is non-zero. This finishes
the proof of Theorem \ref{thm_A_for_knots}.\newline(Case B) Suppose there
exists at least one Alexander root $\beta_{i}$ with $\left\vert \beta
_{i}\right\vert =1$, but $\beta_{i}\notin\mu_{\infty}$. Then Equation
\ref{lWA1} contains the corresponding summand $R_{\beta_{i}}$. By Theorem
\ref{thm_RxAtUnimodularHasNaturalBoundary} for all $p\in\{\beta_{i}^{m}\mid
m\in\mathbf{Z}\}$, we have%
\begin{equation}
\operatorname*{limrad}_{z\rightarrow p}\,(1-\left\vert z\right\vert
)R_{\beta_{i}}(z)=-\frac{1}{2\left\vert m\right\vert }\delta_{m\neq0}%
\text{.}\label{lc1}%
\end{equation}
This yields a dense set of singular points of the unit circle, making the unit
circle the natural boundary for the summand $R_{\beta_{i}}$. We need to study
whether the summation of functions $R_{\beta_{i}}$ (for varying $i$) in
Equation \ref{lWA1} may lead to a cancellation of singular points. We claim
that this is not possible, because: (1) Each summand $R_{\beta_{j}}$ with
$\left\vert \beta_{j}\right\vert <1$ admits an analytic continuation to the
entire complex plane without any poles on the unit circle, so it satisfies%
\[
\operatorname*{limrad}_{z\rightarrow p}\,(1-\left\vert z\right\vert
)R_{\beta_{j}}(z)=0\text{.}%
\]
(2) A summand $R_{\beta_{j}}$ with $\beta_{j}\in\mu_{\infty}$ only has poles
at finitely many roots of unity. Since $p$ is not a root of unity, we again
get $\operatorname*{limrad}_{z\rightarrow p}\,(1-\left\vert z\right\vert
)R_{\beta_{j}}(z)=0$. (3) Each summand $R_{\beta_{j}}$ with $\beta_{j}$
multiplicatively independent from $\beta_{i}$ (and not a root of unity) also
satisfies%
\[
\operatorname*{limrad}_{z\rightarrow p}\,(1-\left\vert z\right\vert
)R_{\beta_{j}}(z)=0
\]
by the second statement of Theorem \ref{thm_RxAtUnimodularHasNaturalBoundary}.
(4) This only leaves summands $R_{\beta_{j}}$ with $\beta_{j}$
multiplicatively dependent on $\beta_{i}$ as candidates for cancellation.
Indeed, by the first part of Theorem
\ref{thm_RxAtUnimodularHasNaturalBoundary} they may contribute a non-zero
value. However, at least after taking a sufficiently divisible
power\footnote{More precisely:\ We only want integral powers of Alexander
roots, so the exponent must be sufficiently divisible to clear all
denominators in the multiplicative dependency relation.}, these may only add
up values of the shape%
\[
-\frac{1}{2\left\vert M^{\prime}\right\vert }%
\]
for suitable $M^{\prime}\geq1$, i.e. (if $p$ is a sufficiently divisible
power)%
\[
\operatorname*{limrad}_{z\rightarrow p}\,(1-\left\vert z\right\vert
)E_{K}(z)=\sum_{k}\left(  -\frac{1}{2\left\vert M_{k}^{\prime}\right\vert
}\right)  \qquad\text{(finite sum).}%
\]
Along with Equation \ref{lc1}, all these values are $<0$, so no non-empty sum
of them can be zero. In particular, no cancellation is possible. It follows
that $E_{K}$ has a dense set of singular points on its radius of convergence.
Hence, the unit circle is the natural boundary for this power series. Along
the way, we have shown the claimed behaviour at boundary values. This finishes
the proof of Theorem \ref{thm_A_for_knots_NatBdryCase}.
\end{proof}

For the sake of completeness, let us also state a structure result regarding
the torsion homology order along with a (rather innocent) bound on the error:

\begin{theorem}
Let $K\subset S^{3}$ be a knot and suppose the Alexander polynomial has no
diophantine roots. Let $m$ be the least common multiple of all orders of roots
of unity which are roots of $\Delta_{K}$, and $m=1$ if there are none. Then
there exists an $m$-periodic sequence $(a_{r})_{r\geq0}$, i.e.%
\[
a_{r+m}=a_{r}\qquad\text{for all }r\geq0
\]
such that%
\[
\left\vert \log\left\vert H_{1}(X_{r},\mathbf{Z})_{\operatorname*{tor}%
}\right\vert -(a_{r}+\log\mathcal{M}(\Delta_{K})r)\right\vert \leq\sum
_{\beta,\left\vert \beta\right\vert \neq1}\frac{\left\vert \beta^{\pm
1}\right\vert ^{r}}{1-\left\vert \beta^{\pm1}\right\vert ^{r}}\text{,}%
\]
where we take $\beta$ if $\left\vert \beta\right\vert <1$ and $\beta^{-1}$ if
$\left\vert \beta\right\vert >1$.
\end{theorem}

\begin{proof}
We have%
\begin{align*}
E_{K}(z)  & =\sum_{\beta,\left\vert \beta\right\vert <1}R_{\beta}%
(z)+\sum_{\beta,\left\vert \beta\right\vert >1}R_{\beta}(z)+\sum
_{\beta,\left\vert \beta\right\vert =1}R_{\beta}(z)\\
& =\frac{z}{(z-1)^{2}}\log\mathcal{M}(\Delta_{K})+\sum_{\beta,\left\vert
\beta\right\vert <1}R_{\beta}(z)+\sum_{\beta,\left\vert \beta\right\vert
>1}R_{\beta^{-1}}(z)+\sum_{\beta,\left\vert \beta\right\vert =1}R_{\beta
}(z)\text{,}%
\end{align*}
where $\beta$ runs through the roots of the Alexander polynomial. By
assumption each root $\beta$ with $\left\vert \beta\right\vert =1$ is a root
of unity, say of $m$-th order, and thus (by the definition of $R_{x}$,
Definition \ref{Def_FunctionR}) the coefficients in the power series expansion
of $R_{\beta}$ at $z=0$ are periodic of period $m$. Thus, taking the least
common multiple of these orders, we can split off the summand $\sum
_{\beta,\left\vert \beta\right\vert =1}$ an encode it as the sequence
$(a_{i})_{i\geq0}$ in our claim. Moreover,%
\[
\frac{z}{(z-1)^{2}}\log\mathcal{M}(\Delta_{K})=\log\mathcal{M}(\Delta_{K}%
)\sum_{r\geq1}rz^{r}\text{,}%
\]
so we can also understand the contribution of this summand to the coefficients
easily. Next, note that%
\begin{equation}
\left\vert \log\left\vert 1-x\right\vert \right\vert \leq\frac{\left\vert
x\right\vert }{1-\left\vert x\right\vert }\text{.}\label{lvda1}%
\end{equation}
(By the two-sided triangle inequality%
\begin{equation}
\left\vert 1-\left\vert x\right\vert \right\vert \leq\left\vert 1-x\right\vert
\leq1+\left\vert x\right\vert \text{.}\label{lvion11}%
\end{equation}
Note that $s\mapsto\left\vert \log s\right\vert $ is monotonously decreasing
for real $s\in(0,1]$ and monotonously increasing for $s\geq1$. The case $x=0$
is trivial, so let us first look at the case $0<\left\vert x\right\vert <1$:
We need a further case distinction: (Case A) $\left\vert 1-x\right\vert
\in(0,1]$. Then Equation \ref{lvion11} implies $\left\vert \log(1-\left\vert
x\right\vert )\right\vert \geq\left\vert \log\left\vert 1-x\right\vert
\right\vert $. For any real number $t>-1,t\neq1$ one has the classical
inequality $\frac{t}{t+1}<\log(1+t)$, so plugging in $-\left\vert x\right\vert
\in(-1,1)\setminus\{0\}$, we get $\left\vert \log\left\vert 1-x\right\vert
\right\vert \leq\left\vert \log(1-\left\vert x\right\vert )\right\vert
\leq\left\vert \frac{\left\vert x\right\vert }{1-\left\vert x\right\vert
}\right\vert $ and line \ref{lvda1} is true. (Case B) Now suppose $\left\vert
1-x\right\vert \geq1$. In this case Equation \ref{lvion11} implies $\left\vert
\log\left\vert 1-x\right\vert \right\vert \leq\left\vert \log(1+\left\vert
x\right\vert )\right\vert $. For any real number $t>0$, one has the classical
inequality $\log(t)\leq t-1$, so $\left\vert \log\left\vert 1-x\right\vert
\right\vert \leq\left\vert \log(1+\left\vert x\right\vert )\right\vert
\leq\left\vert x\right\vert $ and again line \ref{lvda1} is true.) Hence,%
\[
\left\vert \log\left\vert H_{1}(X_{r},\mathbf{Z})_{\operatorname*{tor}%
}\right\vert -(a_{r}+\log\mathcal{M}(\Delta_{K})r)\right\vert \leq\sum
_{\beta,\left\vert \beta\right\vert \neq1}\frac{\left\vert \beta^{\pm
1}\right\vert ^{r}}{1-\left\vert \beta^{\pm1}\right\vert ^{r}}\text{,}%
\]
where we take $\beta$ if $\left\vert \beta\right\vert <1$ and $\beta^{-1}$ if
$\left\vert \beta\right\vert >1$, and we do not sum anymore over the roots
with $\left\vert \beta\right\vert =1$.
\end{proof}

We can now use Theorem \ref{thm_A_for_knots} to obtain new ways to isolate the
family of knots whose torsion homology is periodic:

\begin{theorem}
\label{thm_refined_gordon_thm}Let $K\subset S^{3}$ be a knot. The following
are equivalent:

\begin{enumerate}
\item The values $\left\vert H_{1}(X_{r},\mathbf{Z})_{\operatorname*{tor}%
}\right\vert $ are periodic in $r$.

\item All Alexander roots are roots of unity.

\item The values $\log\left\vert H_{1}(X_{r},\mathbf{Z})_{\operatorname*{tor}%
}\right\vert $ satisfy a linear recurrence equation.

\item The values $\log\left\vert H_{1}(X_{r},\mathbf{Z})_{\operatorname*{tor}%
}\right\vert $ are periodic in $r$.

\item $E_{K}$ is a rational function.

\item $E_{K}$ has an analytic continuation to a domain containing $z=1$ and a
pole of order one there.

\item $E_{K}$ has an analytic continuation to the entire complex plane with
only finitely many poles.
\end{enumerate}
\end{theorem}

This is a strengthening of Gordon's classical result \cite{MR0295327}.

\begin{proof}
($5\Leftrightarrow2$) Given (5), i.e. $E_{K}$ is rational, it admits an
analytic continuation to the entire complex plane, so all roots of $\Delta
_{K}$ on the unit circle are roots of unity. As soon as there is a root
$\beta$ of the Alexander polynomial of absolute value $\left\vert
\beta\right\vert \neq1$, $E_{K}$ has infinitely many poles. As there are only
finitely many poles, all roots satisfy $\left\vert \beta\right\vert =1$, so
the previous remark covers all roots, i.e. we get (2). The converse is clear.
($1\Leftrightarrow2$) \cite{MR0295327}. ($1\Leftrightarrow4$) obvious,
($5\Leftrightarrow3$) Standard algebra. ($6\Leftrightarrow2$) The Mahler
measure is $+1$ since all roots lie on the unit circle and in this case the
leading coefficient is $\pm1$, too (for the standard normalized Alexander
polynomial representative this follows for example from $\Delta_{K}(1)=\pm1$).
Thus, $E_{K}(z)=$%
\[
\frac{1}{(z-1)^{2}}\log\mathcal{M}(\Delta_{K})+\frac{1}{z-1}\left(
\log\mathcal{M}(\Delta_{K})+\sum_{\substack{\beta_{i}\in\mu_{\infty}
\\m=\text{primitive order of }\beta_{i}}}\frac{1}{m}\log\left(  \frac{1}%
{m}\right)  \right)  +(\ldots)
\]
simplifies to a pole of order one at $z=1$, because the sum over the strictly
negative terms $\frac{1}{m}\log\left(  \frac{1}{m}\right)  $ is always
non-zero. Conversely, if the pole at $z=1$ has order one, we must have
$\mathcal{M}(\Delta_{K})=1$, so if $a$ denotes the leading coefficient of
$\Delta_{K}$, we get%
\[
1=\left\vert a\right\vert \prod_{\beta_{i},\left\vert \beta_{i}\right\vert
\geq1}\left\vert \beta_{i}\right\vert \qquad\Rightarrow\qquad\mathbf{Z}%
_{\geq1}\ni\left\vert a\right\vert =\frac{1}{\prod_{\beta_{i},\left\vert
\beta_{i}\right\vert \geq1}\left\vert \beta_{i}\right\vert }\leq1\text{.}%
\]
Hence, we must have $\left\vert a\right\vert =1$ and all roots lie on the unit
circle. If any root were not a root of unity, the analytic continuation around
$z=1$ cannot exist. Thus, we get (2). The converse is clear.
($7\Leftrightarrow2$) As used before, if a root $\beta$ has absolute value
$\left\vert \beta\right\vert \neq1$, the analytic continuation has infinitely
many poles, so all roots lie on the unit circle, and by the existence of an
analytic continuation, they must be roots of unity. (2) follows. The converse
is again clear.
\end{proof}

\subsection{Higher-dimensional knots and Reidemeister torsion}

Many variations of this theme are possible: For example, higher-dimensional
knots in homology spheres, thanks to work of Porti \cite{MR2073320}. One would
proceed as follows, we only sketch the necessary modifications:

Let $K^{n}\subset M^{n+2}$ be a PL $n$-knot, where $M^{n+2}$ is a
PL\ $(n+2)$-dimensional homology sphere, e.g. the ordinary sphere $S^{n+2}$
itself. This is sufficient to ensure that the fundamental group of the
complement abelianizes to $\mathbf{Z}$, and thus one has a similar
construction of cyclic branched coverings%
\[
\widehat{X}_{\infty}\longrightarrow\widehat{X}_{r}\longrightarrow M^{n+2}%
\]
generalizing those of Remark \ref{rmk_BranchedCoverings}.

\begin{theorem}
\label{thm_PortiFoxFormula}Let $K^{n}\subset M^{n+2}$ be a PL $n$-knot, where
$M^{n+2}$ is a PL\ $(n+2)$-dimensional homology sphere. If $\Delta_{K^{n},i}$
denotes the $i$-th Alexander polynomial, and none of the $\Delta_{K^{n},i}$
has a root in $\mu_{\infty}$, then the generating function of the Reidemeister
torsion%
\[
J_{K^{n}}(z):=\sum_{r\geq1}\log(\tau_{r})\cdot z^{r}%
\]
with%
\[
\tau_{r}:=\prod_{i=1}^{n}\left\vert H_{i}(\widehat{X}_{r},\mathbf{Z}%
)\right\vert ^{(-1)^{i+1}}%
\]
has the following property:

\begin{enumerate}
\item If no root of any of the Alexander polynomials $\Delta_{K^{n},i}$ has
absolute value $1$, the function admits a meromorphic continuation to the
entire complex plane. Its poles are located at most at all integer powers of
all roots of all $\Delta_{K^{n},i}$ which lie outside the open unit disc.

\item If some $\Delta_{K^{n},i}$ has a root of absolute value $1$ and no other
$\Delta_{K^{n},j}$ (with $j\neq i$) has a root at the same value, then
$J_{K^{n}}$ has the unit circle as its natural boundary. An analytic
continuation beyond the unit circle is impossible.
\end{enumerate}
\end{theorem}

As before, we can also completely describe the Laurent expansion at $z=1$,
including an alternating sum of log-Mahler measures now, and can understand
the boundary behaviour in case (2). We leave it to the reader to spell out
such details.

The key ingredient would be the work of Porti on identifying Reidemeister
torsion with higher Alexander polynomials, specifically:

\begin{theorem}
[{Porti \cite[Theorem 6.1]{MR2073320}}]Let $K^{n}\subset M^{n+2}$ be a PL
$n$-knot, where $M^{n+2}$ is a PL\ $(n+2)$-dimensional homology sphere. If
$\Delta_{K^{n},i}$ denotes the $i$-th Alexander polynomial, and none of the
$\Delta_{K^{n},i}$ has a root in $\mu_{\infty}$, then%
\[
\prod_{i=1}^{n}\left\vert H_{i}(\widehat{X}_{r},\mathbf{Z})\right\vert
^{(-1)^{i+1}}=\prod_{i=1}^{n+1}\prod_{\zeta\in\mu_{r}}\left\vert \Delta
_{K^{n},i}(\zeta)\right\vert ^{(-1)^{i+1}}\text{.}%
\]

\end{theorem}

Now, one may use this formula instead of Fox' formula in the proof of Theorem
\ref{thm_A_for_knots} and unravel it as in Equation \ref{l_C_translate_root}
to a statement in terms of functions $R_{x}$. Then%
\begin{align*}
\log\prod_{i=1}^{n}\left\vert H_{i}(\widehat{X}_{r},\mathbf{Z})\right\vert
^{(-1)^{i+1}}  & =\sum_{i=1}^{n}(-1)^{i+1}\log\left\vert H_{i}(\widehat{X}%
_{r},\mathbf{Z})\right\vert \\
& =\sum_{i=1}^{n+1}(-1)^{i+1}\left(  \log\left\vert a_{i}\right\vert +\sum
_{j}\log\left\vert 1-\alpha_{i,j}^{r}\right\vert \right)
\end{align*}
with $a_{i}$ the leading coefficients of $\Delta_{K^{n},i}$ and its
$\alpha_{i,j}$ the roots. The viewpoint changes a little here since instead of
the generating function of an individual (torsion) homology group, we now get
a generating function for Reidemeister torsion%
\[
J(z)=\sum_{n=1}^{\infty}\log\left\vert \tau(\widehat{X}_{r})\right\vert \cdot
z^{r}%
\]
via the identification of the Reidemeister torsion with the\ Alexander
function, based on Milnor and\ Turaev, \cite[Thm. 1.1.1]{MR832411}. We leave
the details and further variations of the same theme to the interested reader.
For example, Porti's paper \cite{MR2073320} goes further, generalizing the
formulae for branched cyclic coverings of link complements \`{a} la
Hosokawa--Kinoshita \cite{MR0125579} and Mayberry's thesis (see
\cite{MR648083}).

\begin{remark}
I do not know to what extent the different Alexander polynomials can have
joint roots. If they have, this opens up the possibility that the
corresponding terms $R_{x}$ in the expansion of $J_{K^{n}}$ cancel out if they
come from homology groups of different parity. For example, it could happen
that two roots lying on the unit circle annihilate each other so that
$J_{K^{n}}$ admits an analytic continuation although roots on the unit circle
are present. This is the analytic counterpart of the problem that Reidemeister
torsion usually does not allow us to control any individual torsion homology group.
\end{remark}

\subsection{\label{sect_CyclicResultants}Application to cyclic resultants}

Suppose $f\in\mathbf{C}[t]$ is a polynomial. It comes with a sequence of
complex numbers $(r_{m})_{m\geq1}$ defined by%
\[
r_{m}:=\operatorname*{Res}(f,t^{m}-1)\text{,}%
\]
where \textquotedblleft$\operatorname*{Res}$\textquotedblright\ refers to the
resultant of two polynomials. The values $r_{m}$ are known as the \emph{cyclic
resultants}.

\begin{example}
The classical example stems from the work of Pierce and Lehmer. For $f(t)=t-2
$, one has $r_{m}=2^{m}-1$ (the\ Mersenne sequence). Inspired by Mersenne's
method to find large prime numbers, Lehmer suggested the following heuristic principle:
\end{example}

\begin{heuristic}
[Lehmer]If $f$ has Mahler measure \textquotedblleft very close to
$1$\textquotedblright, then sequence $r_{m}$ should contain \textquotedblleft
a lot\textquotedblright\ of prime numbers. See \cite{MR1783409}.
\end{heuristic}

One can rephrase the definition of the $r_{m}$ in terms of evaluating $f$ at
roots of unity. Thus, it can be rephrased in a format close to the expression
in the formula of Fox, Theorem \ref{thm_Fox}, and Fried's Theorem, Theorem
\ref{thm_Fried}, might suggest that it could be possible to reconstruct $f$
from the values $r_{m}$. However, this turns out to be false. In general, the
values $r_{m}$ do not uniquely pin down $f$. Hillar shows that generically we
should expect $2^{\deg(f)-1}$ polynomials with the same cyclic resultants
\cite[Corollary 1.5]{MR2167674}. His paper provides a number of examples of
distinct polynomials with equal cyclic resultants. \textit{Loc. cit.} also
shows that there is a Zariski dense open in the affine space of all
\textsl{monic} polynomials of any bounded degree for whose polynomials the
cyclic resultants uniquely pin down the polynomial. Work of Hillar and Levine
discusses criteria ensuring that agreement of finitely many cyclic resultants
(depending on the degree of $f$) is sufficient to prove $f=g$ \cite{MR2286068}%
. Hillar \cite{MR2167674} also addresses how to solve the problem of
reconstructing $f$ from $(r_{m})$ algorithmically. This is possible since one
`just' has to solve a system of multi-variable polynomial equations, namely%
\[
r_{1}=\operatorname*{Res}(f,t-1),\qquad r_{2}=\operatorname*{Res}%
(f,t^{2}-1),\qquad r_{3}=\operatorname*{Res}(f,t^{3}-1),\ldots\text{.}%
\]
If one has an upper bound on the possible degree of $f$, such a system can be
solved algorithmically using Gr\"{o}bner basis techniques. However, in general
it will have several solutions.

The situation is much simpler for reciprocal polynomials:

\begin{theorem}
[{Hillar \cite[Corollary 1.12]{MR2167674}}]Suppose $f,g$ are reciprocal
polynomials and none of their roots is a root of unity. Then if their cyclic
resultants agree, it follows that $f=g$.
\end{theorem}

This generalizes Fried's Theorem, Theorem \ref{thm_Fried}. Since Alexander
polynomials are always reciprocal, this explains why Fried's reconstruction of
the Alexander polynomial is always possible from the torsion homology data,
while one cannot reconstruct a general polynomial from the cyclic resultants.

We may, nonetheless, apply our methods to a general $f$. To this end, we define:

\begin{definition}
Let $f\in\mathbf{C}[t]$ be a polynomial. Define%
\[
T_{f}(z):=\left.  \sum\limits_{m\geq1}\right.  ^{\prime}\log\left\vert
\operatorname*{Res}(f,t^{m}-1)\right\vert \cdot z^{m}\text{,}%
\]
where the notation $\left.  \sum\right.  ^{\prime}$ means: We omit the $m$-th
summand if $\operatorname*{Res}(f,t^{m}-1)=0$ (this happens if and only if $f$
has an $m$-th root of unity as one of its roots).
\end{definition}

We obtain a meromorphic continuation:

\begin{theorem}
\label{thm_TFuncOfCyclicResultants}Suppose $f\in\mathbf{C}[t]$ is a polynomial
with roots $(\beta_{i})$, none of which is diophantine. Then the function
$T_{f}$ admits a meromorphic continuation to the entire complex plane with
poles of order $1$ at%
\begin{equation}
\bigcup_{i,\left\vert \beta_{i}\right\vert \neq0,1}\{\beta_{i}^{\pm
\mathbf{Z}_{\leq-1}},\overline{\beta}_{i}^{\pm\mathbf{Z}_{\leq-1}}%
\}\cup\bigcup_{i,\left\vert \beta_{i}\right\vert =1}\{\beta_{i}^{\mathbf{Z}%
}\}\label{lCA2}%
\end{equation}
and perhaps a pole of order $1$ or $2$ at $z=1$ (or no pole there).
\end{theorem}

\begin{proof}
This is essentially shown as in the proof of Theorem \ref{thm_A_for_knots}: If
$f$ factors as $a\prod_{i=1}^{n}(t-\beta_{i})\in\mathbf{C}[t]$, then%
\[
\left\vert \operatorname*{Res}(f,t^{m}-1)\right\vert =\left\vert a\right\vert
^{m}\prod_{i=1}^{n}\left\vert 1-\beta_{i}^{n}\right\vert \text{.}%
\]
Thus, $T_{f}=\log\left\vert a\right\vert \cdot\frac{z}{(z-1)^{2}}%
+\sum_{i\text{ with }\beta_{i}\neq0}R_{\beta_{i}}(z)$. Now we may proceed as
in the proof of Theorem \ref{thm_A_for_knots}, with slight modifications. We
arrive at%
\[
=\frac{z}{(z-1)^{2}}\log\mathcal{M}(f)+\sum_{0<\left\vert \beta_{i}\right\vert
<1}R_{\beta_{i}}(z)+\sum_{\left\vert \beta_{i}\right\vert >1}R_{\beta_{i}%
^{-1}}(z)+\sum_{\substack{\left\vert \beta_{i}\right\vert =1 \\\beta_{i}\in
\mu_{\infty}}}R_{\beta_{i}}(z)
\]
and can invoke our results about the meromorphic continuation of the functions
$R_{\beta}$ for $\left\vert \beta\right\vert \in(0,1)$. We leave the details
to the reader.
\end{proof}

Of course, there is also an analogue of Theorem
\ref{thm_A_for_knots_NatBdryCase} in the case of diophantine roots. We will
not spell this out in detail as it is entirely analogous to the treatment in
the case of Alexander polynomials for knots.

Whenever the hypotheses of the above theorem are met, we obtain a new proof of
the following result of Hillar from 2002:

\begin{theorem}
[{Hillar \cite[Theorem 1.8]{MR2167674}}]Let $f,g\in\mathbf{R}[t]$ be
polynomials such that their cyclic resultants are all non-zero. Then the
absolute values of the cyclic resultants agree, i.e.%
\begin{equation}
\left\vert r_{m}(f)\right\vert =\left\vert r_{m}(g)\right\vert \qquad\left(
\text{for }m\geq1\right)  \text{,}\label{lCA1}%
\end{equation}
if and only if there exist $u,v\in\mathbf{C}[t]$ with $u(0)\neq0$ and integers
$\ell_{1},\ell_{2}\geq0$ such that%
\begin{align*}
f(t)  & =\pm t^{\ell_{1}}v(t)u(t^{-1})t^{\deg u}\\
g(t)  & =t^{\ell_{2}}v(t)u(t)\text{.}%
\end{align*}

\end{theorem}

We shall now give a new proof of this result under slightly more restrictive
hypotheses: \textsl{We need to assume that no root of }$f$ \textsl{(regarded
over the complex numbers) lies on the unit circle}. Hillar's condition that
all cyclic resultants are non-zero only rules out that no roots of unity
appears as roots, so this is a strictly stronger assumption:

\begin{proof}
[New proof (under this assumption)]Condition \ref{lCA1} means that
$T_{f}=T_{g}$. Thus, by Theorem \ref{thm_TFuncOfCyclicResultants} for both
$f,g$ the sets of poles%
\[
\bigcup_{i\text{ with }\beta_{i}\neq0}\{\beta_{i}^{\pm\mathbf{Z}_{\leq-1}%
}\}\cup\{1\}
\]
agree, and so do the residues at these poles. Note that since no root of unity
is a root by assumption, we could discard the union $\bigcup_{i,\left\vert
\beta_{i}\right\vert =1}\{\beta_{i}^{\mathbf{Z}}\}$ in Equation \ref{lCA2},
and since the polynomials are real, the complex conjugate of each root is a
root itself, so we could discard the elements $\overline{\beta_{i}}$ in
Equation \ref{lCA2} as well, since they are contained in the set of all root
powers anyway. Since we can read off the multiplicity of a root (or its
inverse) from the residue at the pole in $T_{f}=T_{g}$, we deduce that%
\[
f=at^{\ell_{1}}\prod(t-\beta_{i}^{S_{i}})\qquad\text{and}\qquad g=bt^{\ell
_{2}}\prod(t-\beta_{i}^{T_{i}})
\]
for suitable choices of $S_{i},T_{i}\in\{\pm\}$, $\ell_{1},\ell_{2}\geq0 $ and
$a,b\in\mathbf{R}$. Now, define $v(t):=b\prod(t-\beta_{i})$ with the product
running over all $\beta_{i}$ such that $T_{i}\neq S_{i}$ (opposite parity),
and $u(t)=\prod(t-\beta)$ running over all $\beta_{i}$ such that $T_{i}=S_{i}$
(same parity). One checks that this choice of $u,v$ settles the claim. This
last part of the proof agrees verbatim with Hillar's proof (\cite[end of Proof
of Theorem 1.1]{MR2167674}). The converse is immediate.
\end{proof}

\begin{remark}
[Comparison]Let us compare this to Hillar's proof. Similar to Fried's
approach, he studies the analytic properties of a function formed from the
cyclic resultants. In their setup, this generating function is always
rational, which at first sight might appear more convenient than $T_{f}$. As
for $T_{f} $, the poles of their function depend explicitly on the roots one
is interested in, however, the dependency is more complicated. Inverting it
requires an algebraic technique to compare factorizations in the semi-group
ring $\mathbf{C}[G]$, with $G\subset\mathbf{C}^{\times}$ the subgroup
generated by the non-zero roots $\beta_{i}$ (\cite[\S 2]{MR2167674}). Such a
step is not needed since our function $T_{f}$ allows us to read off the roots
essentially directly. Hillar's method has the advantage that it also works in
the (highly non-generic) case of diophantine roots, where our $T_{f} $ fails
to admit a meromorphic continuation.
\end{remark}

\subsection{Exceptional units}

Let $K$ be a number field and $u\in\mathcal{O}_{K}^{\times}$ a non-torsion
unit. Write $\mathcal{N}$ for the ideal norm. The power series%
\begin{equation}
G_{u}(z):=\sum_{r\geq1}\log\mathcal{N}(1-u^{r})\cdot z^{r}\label{lai1}%
\end{equation}
always has radius of convergence \textit{precisely} $1$, and diverges
elsewhere. Besides our interest in torsion homology, the function encodes
several invariants which have been studied before in different contexts:

If $u\in\mathcal{O}_{K}^{\times}$ is a unit, it is called \emph{exceptional}
if $1-u$ is also a unit. More geometrically, an exceptional unit is an
$\mathcal{O}_{K}$-integral point of $\mathbf{P}^{1}\setminus\{0,1,\infty\}$.
This is a classical Diophantine problem, and a number of cases have been
worked out in the literature, e.g. \cite{MR1151859}, \cite{MR1464147}. We
shall later need the following non-trivial fact:

\begin{proposition}
[Siegel]\label{Prop_Siegel_FinManyExceptionalUnits}A number field $K$ has only
finitely many exceptional units.
\end{proposition}

Lang shows in \cite{MR0130219} how this reduces to Siegel's theorem on the
finiteness of integral points of genus $\geq1$ curves. The original result of
Siegel is \cite[Satz 10]{MR1544471}. The result was stated in the above form
both by Nagell \cite[Thm. 8]{MR0190128} as well as Chowla \cite{MR0142538}. A
textbook version including a proof can be found in \cite[Thm. D.8.1]%
{MR1745599}.

\begin{definition}
[Silverman \cite{MR1359419}]If $u\in\mathcal{O}_{K}^{\times}$ is a unit,
denote by $E(u)$ the number of values for $n\geq1$ such that $1-u^{n}$ is also
a unit. Equivalently, $E(u)$ is the number of vanishing coefficients in the
power series $G_{u}(z)$.
\end{definition}

By Siegel's finiteness result, Prop. \ref{Prop_Siegel_FinManyExceptionalUnits}%
, $E(u)$ is well-defined.

\begin{definition}
[Stewart \cite{MR2997578}]Stewart defines $E_{0}(u)$ as the largest integer
such that $1-u^{n}$ is a unit for all $n$ with $1\leq n\leq E_{0}(u)$, or zero
if no such $n$ exists. Equivalently, the zero of $G_{u}(z)$ at $z=0$ has order
precisely $E_{0}(u)+1 $.
\end{definition}

\begin{remark}
[Quantitative aspects]There are also quantitative versions of Siegel's and
Silverman's results. Notably, Evertse \cite[Thm. 1]{MR735341} implies that
there are at most $3\cdot7^{n}$ exceptional units in $K$, where
$n:=[K:\mathbf{Q}]$. A result due to Silverman \cite{MR1359419} states that
there exists an absolute constant $C$ such that%
\[
E(u)\leq C\cdot n^{1+\frac{7}{10\log\log n}}\text{.}%
\]
Moreover, Stewart \cite[Corollary 1]{MR2997578} provides the upper bound%
\[
E_{0}(u)\leq C^{\prime}\cdot\tfrac{n\left(  \log(n+1)\right)  ^{4}}{\left(
\log\log(n+2)\right)  ^{3}}\text{,}%
\]
for some other absolute constant $C^{\prime}$.
\end{remark}

As before, we obtain:

\begin{theorem}
\label{thm_guGlobalAnalyticContinuation}Let $K$ be a number field and
$u\in\mathcal{O}_{K}^{\times}$ a unit. Suppose no embedding $\sigma
:K\hookrightarrow\mathbf{C}$ has $\left\vert \sigma u\right\vert =1$. Then the
function $G_{u}$ admits a meromorphic continuation to the entire complex
plane, with poles at:%
\[
\{\text{\emph{all Galois conjugates of }}u^{n}\text{ \emph{for} }%
n\in\mathbf{Z}\}\setminus\text{(open unit disc)}%
\]
and locally at $z=1$, we have%
\[
G_{u}(z-1)=\log\left(  \mathcal{M}(u)^{[K:\mathbf{Q}(u)]}\right)  \frac
{z}{(z-1)^{2}}-\sum_{\sigma}\log\left\vert F(\sigma u^{\pm1})\right\vert
+\mathsf{O}(z-1)\text{,}%
\]
where $\mathcal{M}(u)$ is the Mahler measure of $u$, $F$ the generating
function of the partition function, $\sigma$ runs through all embeddings
$\sigma:K\hookrightarrow\mathbf{C}$, and \textquotedblleft$\pm$%
\textquotedblright\ stands for $+$ if $\left\vert \sigma u\right\vert >1$ and
$-$ if $\left\vert \sigma u\right\vert <1$.
\end{theorem}

We leave the proof to the reader; it is just a variation of what we have done
for knots. Note that in the case at hand the underlying polynomial is the
minimal polynomial. It need not be reciprocal.

\begin{theorem}
Let $u,v\in\mathcal{O}_{K}^{\times}$ be units such that no $\sigma
:K\hookrightarrow\mathbf{C}$ sends either into the unit circle. Then the
following are equivalent:

\begin{enumerate}
\item Equality $G_{u}(z)=G_{v}(z)$,

\item The unit $v$ is Galois conjugate to $u$ or $u^{-1}$.
\end{enumerate}
\end{theorem}

The infinity of the poles implies that the function $G_{u}$ cannot be
rational. We deduce:

\begin{corollary}
Let $u\in\mathcal{O}_{K}^{\times}$ be a unit such that no $\sigma
:K\hookrightarrow\mathbf{C}$ sends it into the unit circle. Then the sequence%
\[
a_{n}:=\log\mathcal{N}(1-u^{n})
\]
does not satisfy any linear recurrence equation with constant coefficients.
\end{corollary}


\subsection{Further variations}

\begin{example}
By work of Boden and Friedl, one can also count irreducible metabelian
representations of $\pi_{1}(X_{K})$ to $\operatorname*{SL}_{n}(\mathbf{C})$ in
terms of a formula similar to Fox' Formula, Theorem \ref{thm_Fox}, so our
methods also apply to these values, ranging over $n$. See \cite{MR2443505},
Theorem 1.2 and most explicitly Corollary 1.3. We have not worked out the details.
\end{example}

\section{Special $L$-values}

There is a well-known relation between (multi-variable log-)Mahler measures
and special $L$-values. This was realized, first experimentally, by the
surprising computations of Smyth in \cite{MR615132}, e.g.%
\[
\mathcal{M}(1+x+y)=\frac{3\sqrt{3}}{4\pi}L(2,\chi)
\]
(where $\chi$ is a certain Dirichlet character) and later theoretically
explained through the Beilinson conjectures by Deninger \cite{MR1415320}. We
will not re-tell this story and refer to \cite{MR1618282}, \cite{MR1691309}
for explanations. Inspired by this, it feels noteworthy that there is a
\textit{genuinely different} way how special $L$-values appear in our
computations, related to the function $R_{x}$ when $x$ is a root of unity. We
may re-interpret Proposition \ref{prop_Rx_AtRootOfUnity} as follows:

\begin{proposition}
\label{Prop_RxAtRootOfUnity_ViaLValuesAtOne}Suppose $x\in\mu_{m}$ is an $m$-th
root of unity. Then%
\[
R_{x}\in\mathbf{Q}(\zeta_{m})(\{L(1,\chi)\}_{\chi\in M})
\]
for some set $M$ of Dirichlet characters $\chi$ modulo $m$. That is: $R_{x}$
is a rational function over a finitely generated field extension of the
rationals, generated by the $m$-th roots of unity and a finite number of
special $L$-values of Dirichlet characters at $s=1$.
\end{proposition}

\begin{proof}
Let $m\geq2$ be any integer and $f:\mathbf{Z}/m\rightarrow\mathbf{C}$ be a
function. Using the Fourier theory of the group $(\mathbf{Z}/m,+)$, we get%
\begin{equation}
f(n)=\sum_{l=0}^{m-1}\widehat{f}(l)e^{2\pi i\frac{ln}{m}}\qquad\text{for}%
\qquad\widehat{f}(n)=\frac{1}{m}\sum_{l=0}^{m-1}f(l)e^{-2\pi i\frac{ln}{m}%
}\text{.}\label{lFourierOnFG}%
\end{equation}
In particular, $\widehat{f}(0)=\frac{1}{m}\sum_{l=0}^{m-1}f(l)$. Now, suppose
we have $\widehat{f}(0)=0$. In this case, the Dirichlet series associated to
$f$ has the shape%
\[
L(s,f)=\sum_{n\geq1}\frac{f(n)}{n^{s}}\text{.}%
\]
Expanding $f$ as its Fourier series over $\mathbf{Z}/m$, this becomes%
\[
L(s,f)=\sum_{l=0}^{m-1}\widehat{f}(l)\sum_{n\geq1}\frac{(e^{2\pi i\frac{l}{m}%
})^{n}}{n^{s}}=\sum_{l=0}^{m-1}\widehat{f}(l)L(s,\chi_{l})
\]
with $\chi(n):=e^{2\pi i\frac{l}{m}n}$. Since $\widehat{f}(0)=0$ by
assumption, only the summands with $l\neq0$ appear in the sum, and for these
$\chi_{l}$ is a non-principal character. Thus, each $L(s,\chi_{l})$ admits a
holomorphic continuation to the entire complex plane and the value at $s=1$ is
given by the convergent series%
\begin{align*}
L(1,f)  & =\sum_{l=1}^{m-1}\widehat{f}(l)L(1,\chi_{l})=\sum_{l=1}%
^{m-1}\widehat{f}(l)\sum_{n\geq1}\frac{\chi_{l}(n)}{n}\\
& =\sum_{l=1}^{m-1}\widehat{f}(l)\sum_{n\geq1}\frac{(e^{2\pi i\frac{l}{m}%
})^{n}}{n}=-\sum_{l=1}^{m-1}\widehat{f}(l)\log(1-e^{2\pi i\frac{l}{m}%
})\text{.}%
\end{align*}
Now, if all the Fourier coefficients are real, i.e. $\widehat{f}%
(l)=\overline{\widehat{f}(l)}$, then
\[
-\frac{1}{2}L(1,f+\overline{f})=-\frac{1}{2}L(1,f)-\frac{1}{2}\overline
{L(1,f)}=\sum_{l=1}^{m-1}\widehat{f}(l)\log\left\vert 1-e^{2\pi i\frac{l}{m}%
}\right\vert \text{.}%
\]
Given $1\leq a\leq m-1$, define $\widehat{f}(l):=\delta_{l=a}$ (which
determines $f$ by Fourier inversion). It follows that $\log\left\vert
1-e^{2\pi i\frac{l}{m}}\right\vert =-\frac{1}{2}L(1,f+\overline{f})$ for this
particular $f$, and by Fourier expansion, $f+\overline{f}$ can itself be
expanded in terms of characters. Hence, every $\log\left\vert 1-e^{2\pi
i\frac{l}{m}}\right\vert $ is a finite linear combination of special
$L$-values of Dirichlet characters at $s=1$ with coefficients in the
cyclotomic field $\mathbf{Q}(\zeta_{m})$ (by Equation \ref{lFourierOnFG}).
Thus, Prop. \ref{prop_Rx_AtRootOfUnity} implies that%
\[
R_{x}\in\mathbf{Q}(\zeta_{m})(\{L(1,\chi)\}_{\chi\in M})\text{,}%
\]
where the set $M$ encompasses the Dirichlet characters appearing in the
Fourier expansion of $f+\overline{f}$.
\end{proof}

\begin{remark}
Moreover, the somewhat unwieldy expression $A_{m}:=\sum_{l=1}^{m-1}l\cdot
\log\left\vert 1-\zeta_{m}^{l}\right\vert $ in the expansion of $R_{x}$ at
$z=1$ can be interpreted this way. One gets $A_{m}=-\frac{1}{2}L(1,f+\overline
{f})$ for $f$ determined by $\widehat{f}(l)=l$ for $l=0,1,\ldots,m-1$.
\end{remark}

\bibliographystyle{amsalpha}
\bibliography{ollinewbib}

\end{document}